\documentclass{lms}

\usepackage{amsmath, geometry,mathtools,amsfonts, amssymb,tikz,verbatim,curves,color,enumitem,
bbm, bm, mathrsfs, makecell}

\usepackage[retainorgcmds]{IEEEtrantools}
\usepackage[thinlines]{easytable}
\usepackage{float, hyperref}
\usetikzlibrary{arrows,automata,positioning,shapes,shadows}

\usepackage[pagewise]{lineno}

\newtheorem{theorem}{Theorem}[section] \newtheorem{lemma}[theorem]{Lemma}     \newtheorem{cor}[theorem]{Corollary}
\newtheorem{prop}[theorem]{Proposition}

\newnumbered{conjecture}{Conjecture}  
\newnumbered{definition}{Definition}
\newnumbered{prerk}{Remark}
\newnumbered{predef}{Definition}
\newnumbered{problem}{Problem}
\newnumbered{question}{Question}
\newnumbered{example}{Example}
\newunnumbered{ntn}{Notation} 
\newnumbered{cla}[theorem]{Claim}

\newcommand{\dimr}{\dim_{\mathrm{rep}}}

\renewcommand{\theenumi}{\alph{enumi}}

\newcommand{\End}{\mathop{\mathrm{End}}\nolimits}
\newcommand{\Aut}{\mathop{\mathrm{Aut}}\nolimits}
\newcommand{\core}{\mathop{\mathrm{core}}\nolimits}
\newcommand{\xn}{X_{n}}
\newcommand{\xnp}{X_{n}^{+}}
\newcommand{\xns}{\xn^{\ast}}
\newcommand{\xnz}{X_n^{\mathbb{Z}}}
\newcommand{\xnn}{X_n^{\mathbb{N}}}
\newcommand{\xnN}{\xn^{-\mathbb{N}}}
\newcommand{\wn}{W_{n}}
\newcommand{\wns}{\wn^{\ast}}
\newcommand{\wnl}[1]{\wn^{#1}}
\newcommand{\rwnl}[1]{\mathsf{W}_{n}^{#1}}
\newcommand{\spn}[1]{\widetilde{\mathcal{P}_{#1}}}
\newcommand{\pn}[1]{\mathcal{P}_{#1}}
\newcommand{\shn}[1]{\widetilde{\mathcal{H}_{#1}}}
\newcommand{\hn}[1]{\mathcal{H}_{#1}}
\newcommand{\On}[1]{\mathcal{O}_{#1}}
\newcommand{\SOn}[1]{\widetilde{\mathcal{O}_{#1}}}
\newcommand{\Mn}[1]{\mathcal{M}_{#1}}
\newcommand{\Nn}[1]{\mathcal{N}_{#1}}
\newcommand{\SLn}[1]{\widetilde{\mathcal{L}_{#1}}}
\newcommand{\ALn}[1]{\mathcal{AL}_{#1}}
\newcommand{\ASLn}[1]{\widetilde{\mathcal{ASL}_{#1}}}
\newcommand{\Ln}[1]{\mathcal{L}_{#1}}

\newcommand{\gen}[1]{\left\langle #1 \right \rangle}
\newcommand{\aut}[1]{\mathop{\mathrm{Aut}}({#1})}
\newcommand{\out}[1]{\mathop{\mathrm{Out}}({#1})}
\newcommand{\shift}[1]{\sigma_{#1}}
\newcommand{\Shift}[1]{\mathfrak{S}_{#1}}
\newcommand{\rev}[1]{\overleftarrow{#1}}
\newcommand{\im}[1]{\mathop{\mathrm{Image}}{(#1)}}
\newcommand{\id}{\mathbbm{1}}
\newcommand{\Z}{\mathbb{Z}}
\newcommand{\N}{\mathbb{N}}
\newcommand{\Q}{\mathbb{Q}}
\newcommand{\spnprod}[1]{\ast_{\spn{n}}}
\newcommand{\rsig}{\overline{\mathop{\mathrm{sig}}}}
\newcommand{\rot}{\sim_{\mathrm{rot}}}
\newcommand{\rotc}[1]{[#1]_{\rot}}
\newcommand{\sym}[1]{\mathop{\mathrm{Sym}}(#1)}
\newcommand{\tran}[1]{\mathop{\mathrm{Tr}}(#1)}
\newcommand{\Un}[1]{\mathfrak{U}_{#1}}
\newcommand{\seteq}{:=}

\author{J. Belk, C. Bleak, P. J. Cameron and F. Olukoya}
\date{\today}
\extraline{The authors wish to gratefully acknowledge support from EPSRC research grant EP/R032866/1. The first author is also grateful for support from the National Science Foundation under Grant No.\ \mbox{DMS-185436}.
The fourth author is also grateful for support from the Leverhulme Trust Research Project Grant RPG-2017-159.}
\classno{20E36 (primary), 28D15, 20F10, 54H15, 28D15 (secondary)}

\begin{document}
\title[Automorphisms of the two-sided shift]
 {Automorphisms of shift spaces and the Higman--Thompson groups: the two-sided case} 

\maketitle
\begin{abstract}
 In this article, we further explore the nature of a connection between the groups of automorphisms of full shift spaces and the groups of outer automorphisms  of the Higman--Thompson groups $\{G_{n,r}\}$.  

 We show that the quotient of the group of automorphisms of the (two-sided) shift dynamical system $\aut{\xnz, \shift{n}}$ by its centre embeds as a particular subgroup  $\Ln{n}$ of the outer automorphism group $\out{G_{n,n-1}}$ of $G_{n,n-1}$.
It follows by a result of Ryan that we have the following central extension:
$$\gen{\shift{n}} \hookrightarrow \aut{\xnz, \shift{n}} \twoheadrightarrow \Ln{n}$$ where here, $\gen{\shift{n}}\cong\Z$. We prove that this short exact sequence splits if and only if $n$ is not a proper power, and, in all cases, we compute the 2-cocycles and 2-coboundaries for the extension. We also use this central extension to prove that for $1 \le r < n$, the groups $\out{G_{n,r}}$ are centreless and have undecidable order problem. 

Note that the group $\out{G_{n,n-1}}$ consists of finite transducers (combinatorial objects arising in automata theory), and elements of the group $\Ln{n}$ are easily characterised within $\out{G_{n,n-1}}$ by a simple combinatorial property.  In particular, the short exact sequence allows us to determine a new and efficient purely combinatorial representation of elements of $\aut{\xnz, \shift{n}}$, and we demonstrate how to compute products using this new representation.

In previous work, the last three authors show that the group $\aut{\xnn, \shift{n}}$ of automorphisms of the one-sided shift dynamical system  over an $n$-letter alphabet naturally embeds as a subgroup of the group $\out{G_{n,r}}$ of outer-automorphisms of the Higman--Thompson group $G_{n, r}$, $1 \le r < n$. While the current article can be seen as continuing this line of research, the extension to the 2-sided case requires the development of substantial new combinatorial tools.

\end{abstract}

\section{Introduction}
One of the main results of the article \cite{BleakCameronOlukoya1} is  that the group of automorphisms of the one-sided shift dynamical system $\aut{\xnn, \shift{n}}$ over an $n$-letter alphabet naturally embeds as a subgroup $\hn{n}$ of the group $\out{G_{n,r}}$ of outer-automorphisms of the Higman--Thompson group $G_{n, r}$, for $1 \le r < n\in\N$. An application in \cite{BleakCameronOlukoya1} of this embedding  is to give an efficient decomposition of elements of $\aut{\xnn, \shift{n}}$ as products of finite order elements. The situation is more complex for  the group $\aut{\xnz, \shift{n}}$ of automorphisms of the (two-sided) shift dynamical system. However, we are able to prove a similar embedding result here. We show that the quotient of the group $\aut{\xnz, \shift{n}}$  by its centre embeds as a particular subgroup  $\Ln{n}$  of the outer automorphism group $\out{G_{n,n-1}}$. Note here that  $\hn{n} \le \Ln{n} \le \out{G_{n,n-1}}$, and more importantly, that it is easy to independently characterize elements of the subgroup $\Ln{n}$ as an important subgroup of
$\out{G_{n,n-1}}$ (indeed, we named this subgroup and had already been working with it well before we appreciated the connection to the group $\aut{\xnz, \shift{n}}$).

Recall Ryan's Theorem \cite{Ryan2}  which states that the centre of the group $\aut{\xnz, \shift{n}}$ is the group $\gen{\shift{n}}\cong\Z$.  From this and our discussion above we obtain the following theorem.
\begin{theorem}\label{Thm:shortexactaut}
	The group $\aut{\xnz, \shift{n}}$ is a central extension of $\gen{\shift{n}}$ by~$\Ln{n}$.  That is, there is a short exact sequence 
	$$\gen{\shift{n}} \hookrightarrow \aut{\xnz, \shift{n}} \twoheadrightarrow \Ln{n}$$ where $\gen{\shift{n}}\cong\Z$ and embeds as the centre of $\aut{\xnz, \shift{n}}$.
\end{theorem}

The groups $\Ln{n}$ consist of transducers (a transducer is a type of finite automaton which conveniently transforms sequential data) under a natural and easy-to-compute product.  The results here allow us to represent automorphisms of the shift by annotated transducers, which are finite combinatorial objects (combining the transducers in $\Ln{n}$ with some data representing accumulated ``lag'') which in their own turn can be worked with directly and conveniently.

Next, we use the dimension representation of Kreiger (see \cite{KriegerDimension}) to classify exactly when the short exact sequence of Theorem \ref{Thm:shortexactaut} splits.

\begin{theorem}\label{thm:belksplittingintro}
    The group $\aut{\xnz, \shift{n}}$ is isomorphic to $\Z \times \Ln{n}$ if and only if $n$ is not a power of a smaller integer.
\end{theorem}

As mentioned above, we can compute products efficiently and concretely in $\aut{\xnz,\shift{n}}$ using our combinatorial representation of elements of $\aut{\xnz,\shift{n}}$, regardless of whether or not the sequence of Theorem \ref{Thm:shortexactaut}
splits.  In Section \ref{Sec:cocyclesAndCoboundaries}, we characterise the 2-cocycles and 2-coboundaries for the extension.
 
 {{The innate connection between the groups $\aut{\xnz,\shift{n}}$ and $\out{G_{n,r}}$ stems from results in the papers \cite{Hedlund69} and \cite{AutGnr}. The group $\aut{\xnz, \shift{n}}$ acts on the Cantor space $\xnz$, while the group $\out{G_{n,r}}$ acts on the set of bi-infinite strings. The seminal paper of Hedlund \cite{Hedlund69} demonstrates that elements of $\aut{\xnz, \shift{n}}$ require only a \textit{finite window} around a given index in order to transform the symbol at that index. These are the \textit{sliding block codes}. In the article \cite{AutGnr}, the authors demonstrate that an  element of $\out{G_{n,r}}$ can be represented by a transducer where the subsequent action (possibly not length-preserving) on a finite string at a given index is determined by a finite window around that index. This is the \textit{strong synchronizing condition} of \cite{AutGnr}.  Formalising this connection between sliding block codes and the strong synchronizing condition almost gives rise to an action of $\out{G_{n,r}}$ on $\xnz$. In order to obtain an action, we restrict to the subgroup $\Ln{n}$ of length-preserving transformations.  Therefore one can think of the groups $\out{G_{n,r}}$ as generalizations of $\aut{\xnz, \shift{n}}$, where the role of $r$ is linked to the divisibility structure of $n$ and where we remove the restriction on $\aut{\xnz, \shift{n}}$ that indices are preserved.
}}

In the paper \cite{AutGnr} it is shown that $\out{G_{n,1}} \le \out{G_{n,r}} \le \out{G_{n,n-1}}$.  The following theorem now follows immediately from results in  \cite{KariOllinger,OlukoyaAutTnr}.

\begin{theorem}
\label{thm:orderProblem}	Let $n > r \ge 1$, then the group $\out{G_{n,r}}$ has unsolvable order problem.
\end{theorem}

A consequence of the strong synchronizing condition  is that for any word $w$ over the alphabet there is a unique state of the transducer with a circuit based at that state and labelled by the word $w$. This gives a map $\Pi$ from the set  $\out{G_{n,r}}$ to the transformation monoid on the set on all finite prime words. This map is a faithful representation of $\out{G_{n,r}}$ in the symmetric group on the set of equivalence classes of finite prime words under rotation (prime cyclic words, below), and it is strongly related to the \textit{periodic orbit representation} of \cite{BoyleKrieger}.

\begin{theorem}\label{Thm:symrep}
	The action of $\out{G_{n,r}}$ on prime cyclic  words gives a representation of $\out{G_{n,r}}$ in the symmetric group on the set of all prime cyclic words over an alphabet of size $n$. The restriction of this map to the subgroup $\Ln{n}$ is the periodic orbit representation of the group $\aut{\xnz, \shift{n}}/\gen{\shift{n}}$. 
\end{theorem}

 For $1\leq r<n\in\N$ we can also extend the definition of $\Ln{n}$ above by setting $\Ln{n,r}~\seteq~\Ln{n}~\cap~\out{G_{n,r}}$. The above results then imply:
 
 \begin{cor} We have:
 \begin{itemize}
     \item The group $\out{G_{n,r}}$ is centerless.
     \item The index of $\Ln{n,r}$ in the group $\out{G_{n,r}}$ is infinite for all valid $n$ and $r$. 
 \end{itemize}
 \end{cor} 

Viewed as a group of non-initial transducers, the group $\out{G_{n,r}}$  naturally embeds as a subgroup $\On{n,r}$ of a monoid $\Mn{n}$ of non-initial strongly synchronizing transducers (discussed further below).  The monoid $\Mn{n}$ further contains  a submonoid $\SLn{n}$ of which $\Ln{n}$ is a subgroup.  The table in  Figure \ref{tab:monoids} lists monoids of interest we develop here, together with some of their properties.

Although Theorems~\ref{Thm:shortexactaut} and \ref{Thm:symrep} concern the group $\aut{\xnz, \shift{n}}$, we prove our results in the framework of the monoid $\End(\xnz, \shift{n})$ of endomorphisms of the shift dynamical system. Abusing notation slightly, write $\gen{\shift{n}}$ for the congruence on $\End(\xnz, \shift{n})$ which relates two elements $\psi$ and $\varphi$ if one is a power of a shift times the other. Thus we may form the quotient monoid $\End(\xnz, \shift{n})/\gen{\shift{n}}$. The monoid $\Mn{n}$ plays for $\End(\xnz, \shift{n})$ the role  that $\On{n}$ plays for $\aut{\xnz, \shift{n}}$. 

\begin{theorem}
	The quotient $\End(\xnz, \shift{n})/ \gen{\shift{n}}$ is isomorphic to the monoid $\SLn{n}$. The action of the monoid  $\Mn{n}$ on prime cyclic words yields a monomorphism $\Pi$ into the transformation monoid on all finite prime cyclic words on an alphabet of size~$n$. The restriction of $\Pi$ to the monoid $\SLn{n}$ is the  \textit{periodic orbit representation} of $\End(\xnz, \shift{n})/ \gen{\shift{n}}$.
\end{theorem}

We further characterise $\On{n}$ as precisely the preimage under $\Pi$ of the symmetric group on all finite prime cyclic words on an alphabet of size $n$.

\begin{prop}
	Let $T \in \Mn{n}$. Then $(T)\Pi$  is a bijection if and only if $T \in \On{n}$.
\end{prop}

\subsection*{Some words on the flow of ideas and connections}
The monoid $\Mn{n}$ consists of all non-initial strongly synchronizing transducers over the alphabet $\xn$ under a suitable equivalence relation called \emph{$\omega$-equivalence}. That is, the elements of $\Mn{n}$  are $\omega$-equivalence classes of strongly synchronizing transducers. The paper \cite{GNSenglish} gives a procedure for reducing a transducer $T$ in a class of $\Mn{n}$ to a unique minimal representative. We therefore conflate elements of $\Mn{n}$ with their minimal representatives. The elements of the submonoid $\SLn{n}$ then correspond to those elements of $\Mn{n}$ which  satisfy the following \textit{Lipschitz condition}:  on any circuit, the length of the output word coincides with the length of the circuit. 

The strong synchronizing condition together with the Lipschitz condition means that we may associate to an element $T$ of $\SLn{n}$ a map $\alpha$, called an \emph{annotation}, from the states of $T$ to $\Z$.   The pair $(T, \alpha)$ then induces in a natural way an endomorphism of the shift dynamical system $(\xnz, \shift{n})$.  It turns out that if $\beta$ is another annotation of $T$ then $(T, \alpha)$ and $(T,\beta)$ represent the same element of $\End(\xnz, \shift{n})/ \gen{\shift{n}}$.  Moreover for two elements $T,U \in \SLn{n}$ and annotations $\alpha$ of $T$ and $\beta$ of $U$, $(T,\alpha)$ and $(U, \beta)$ induce the same map on $\xnz$ if and only if  $T$ and $U$ represent the same element of $\SLn{n}$ and $\alpha = \beta$.  The group $\Ln{n}$ consists precisely of those elements $T \in \SLn{n}$ for which there is an annotation $\alpha$ such that $(T,\alpha)$ is an automorphism of the shift dynamical system.

To show that $\SLn{n}$ is in bijective correspondence with elements of $\End(\xnz, \shift{n})/\gen{\shift{n}}$, we prove that any element of $\End(\xnz, \shift{n})$  is induced by a pair $(T, \alpha)$ for $T \in \Mn{n}$ and an annotation   $\alpha$ of  $T$. By the previous paragraph,  it thus follows that the map $T \mapsto (T,\alpha)\gen{\shift{n}}$ is an isomorphism from $\SLn{n}$ to $\End(\xnz, \shift{n})/\gen{\shift{n}}$.

Write $\ASLn{n}$ for the set of pairs $(T, \alpha)$ where $T \in \SLn{n}$ and $\alpha$ is an annotation of~$T$. We define an  associative binary operation on this set which takes a pair  $(T,\alpha)$, $(U, \beta)$ of elements of $\ASLn{n}$ and returns an element $(TU, \overline{\alpha + \beta}) \in \ASLn{n}$. The element $TU$ is simply the product of $T$ and $U$ in $\Mn{n}$, and, $\overline{\alpha + \beta}$ is an annotation obtained from the annotation $\alpha + \beta : Q_{T} \times Q_{U} \to \Z$, $(t,u) \mapsto (t)\alpha + (u)\beta$,  together with certain combinatorial data arising from the \textit{transducer product} of $T$ with $U$. The set $\ASLn{n}$ together with the binary operation is a monoid isomorphic to $\End{(\xnz, \shift{n})}$. More specifically, for $(T,\alpha)$, $(U, \beta)$ of elements of $\ASLn{n}$, the element  $(TU, \overline{\alpha + \beta})$ corresponds to the composition of the induced maps $(T,\alpha)$, $(U, \beta)$ of $\xnz$.

In connection to the embedding in \cite{BleakCameronOlukoya1}, for an element $T$ of $\hn{n}$ the map $\bm{0}$ from $Q_{T}$ to $\Z$  which assigns $0$ to every state is a valid annotation. Moreover, for two elements $(T, \bm{0}), (U, \bm{0})$, $T,U \in \hn{n}$, the element $(TU, \bm{0})$ is precisely the product of $(T, \bm{0})$ with $(U, \bm{0})$ in $\ASLn{n}$.

To prove Theorem~\ref{thm:belksplittingintro} characterising exactly when the short exact sequence $$1  \to \gen{\shift{n}} \to  \aut{\xnz, \shift{n}} \to\Ln{n}\to 1$$  splits, we make use of the well-known \textit{dimension representation} $\dimr$ of Kreiger (\cite{KriegerDimension}). This is a map from $\aut{\xnz, \shift{n}}$ to the subgroup $\mathcal{G}(n)$ of $\Q$ generated by prime divisors of $n$. The key facts we need are: when $n$ is not a power of a smaller integer $\gen{n}$ has a complementary subgroup $\gen{n}^{c}$ in $\mathcal{G}(n)$, and, if $n$ is power of a smaller integer $\shift{n}$ has proper roots in $\aut{\xnz, \shift{n}}$. The former fact follows since $\mathcal{G}(n)$ is a free abelian group while the latter is a well-known fact about $\aut{\xnz, \shift{n}}$ when $n$ is a proper power. The result is an immediate corollary of these two facts.
\vspace{.1 in}

{\flushleft{\it Acknowledgements:}}
The authors are grateful to Elliot Cawtheray for numerous suggested improvements to the text after he carried out a careful reading of a draft of this article. The fourth author gratefully acknowledges the hospitality of the University of Aberdeen where some of this research was conducted.

\section{The Curtis, Hedlund, Lyndon Theorem} \label{section:preliminaries}

In this paper, as in the paper \cite{BleakCameronOlukoya1},  operators will be on the right of their arguments and sequences will be indexed from left to
right in the usual way. Our notation and definitions will also be mostly consistent with the paper \cite{BleakCameronOlukoya1}. We begin by laying down this ground-work.

We denote by $X_n$ the $n$-element set $\{0,1,\ldots,n-1\}$. Then $X_n^*$
denotes the set of all finite strings (including the empty string
$\varepsilon$) consisting of elements of $X_n$. For an element $w\in X_n^*$,
we let $|w|$ denote the length of $w$ (so that $|\varepsilon|=0$). We further
define
\[X_n^+=X_n^*\setminus\{\varepsilon\},\quad
X_n^k=\{w\in X_n^*:|W|=k\},\quad X_n^{\le k}=\bigcup_{0\le i\le k}X_n^i.\]
We denote the concatenation of strings $x,y\in X_n^*$ by $xy$; in this notation
we do not distinguish between an element of $X_n$ and the corresponding
element of $X_n^1$.

For $x,x_1,x_2\in X_n^*$, if $x$ is the concatenation $x_1x_2$ of $x_1$ and
$x_2$, we write $x_2=x-x_1$.  

A \textit{bi-infinite} sequence is a map $x:\mathbb{Z}\to X_n$. We sometimes write this
sequence as $\dots x_{-1}x_0x_1x_2\dots$, where $x_i=x(i)\in X_n$ (note that we use maps on the left for sequences). We denote
the set of such sequences by $\xnz$. In a similar way, a \emph{(positive) singly infinite
sequence} is a map $x:\mathbb{N}\to X_n$ (where, by convention,
$0\in\mathbb{N}$); we write such a sequence by $x_0x_1x_2\dots$, where
$x_i=x(i)$, and denote the set of such sequences by $\xnn$. Finally  we also set $\xn^{-\N}$, the \emph{(negative) singly infinite sequences} for the set of all maps from ${-\N} \to \xn$, we write $\ldots x_{-2}x_{-1}x_0$ for such a map.

We can concatenate a string $x\in X_n^*$ with an singly infinite string
$y\in \xnn$, by prepending $x$ to $y$. We may can also subtract a string $x$ from a singly infinite string $y$ which has $x$ has a prefix by deleting the prefix $x$.

For a string $\nu \in \xns$ we write $U_{\nu}$ for the set of all elements of $\xn^{-\N}$ with $\nu$ as a suffix. For example $U_{\varepsilon} = \xnN$.

Let $F(X_n,m)$ denote the set of functions from $X_n^m$ to $X_n$. Then, for
all $m,r>0$, and all $f\in F(X_n,m)$, we define a map $f_r:X_n^{m+r-1}\to X_n^{r}$
as follows.
\begin{quote}
Let $x=x_{-m-r+2}\ldots x_0$. For $-r+1 \le i\le 0$, set
$y_i=(x_{i-m+1}x_{i-m+2}\ldots x_{i})f$. Then $xf_r=y$, where $y=y_{-r+1}\ldots y_{0}$.
\end{quote}

Thus, as in the paper  \cite{BleakCameronOlukoya1}, we have a ``window'' of length $m$ which slides along the sequence $x$ and at the $i$\textsuperscript{th} step we apply the map $f$ to the symbols visible in the window. We think of the map $f$ as acting on the rightmost letter in the viewing window according to the $m-1$ digits of history.  This procedure can
be extended to define a map $f_\infty:\xnz\to\xnz$, by setting $xf_\infty=y$
where $y_i=(x_{i-m+1}\ldots x_i)f$ for all $i\in\mathbb{Z}$; and similarly for
$\xnN$.

A function $f\in F(X_n,m)$ is called \emph{right permutive} if, for distinct
$x,y\in X_n$ and any fixed block $a\in X_n^{m-1}$, we have $(ax)f\ne(ay)f$. Analogously, a function $f \in  F(X_n, m)$ is called \emph{left permutive} if  the map from $X_n$ to itself given by $x\mapsto(xa)f$ is a permutation for all $a\in X_n^{m-1}$. If $f$ is not right permutive, then
the induced map $f_\infty$ from $\xnN$ to itself is not injective. Moreover, it is not always the case that a right permutive map $f \in F(X_n, m)$ induces a bijective map $f_{\infty}: \xnN\to \xnN$. However, a right permutive map always induces a surjective map from $\xnN$ to itself.

\begin{prerk}
Observe that, if $f\in F(X_n,m)$ and $k\ge1$, then the map $g\in F(X_n,m+k)$
given by $(x_{-m-k+1}\ldots x_0)g=(x_{-m+1}\ldots x_0 )f$, satisfies
$g_\infty=f_\infty$.
\label{F(X_n,m)containedinF(X_n,m+1)}
\end{prerk}

\medskip

The sets $\xnz$ and $\xnN$ are topological spaces, equipped with the Tychonoff
product topology derived from the discrete topology on $X_n$. Each is
homeomorphic to Cantor space. The set $\{ U_{\nu} \mid \nu \in \xns\}$ is a basis of clopen sets for the topology on $\xnN$.

The \emph{shift map} $\shift{n}$ is the map which sends a sequence $x$ in
$\xnz$ or $\xnN$ to the sequence $y$ given by $y(i)=x(i-1)$ for all
$i$ in $\mathbb{Z}$ or $-\mathbb{N}$ respectively.

\medskip

The following result is due to Curtis, Hedlund and Lyndon~\cite[Theorem 3.1]{Hedlund69}: 

\begin{theorem}
Let $f\in F(X_n,m)$. Then $f_\infty$ is continuous and commutes with the
shift map on $\xnz$.
\label{t:hed1}
\end{theorem}

A continuous function from $\xnz$ to itself which commutes with the shift map
is called an \emph{endomorphism} of the shift dynamical system
$(\xnz,\shift{n})$. If the function is invertible, since $\xnz$ is compact and Hausdorff, its inverse is continuous:
it is an \emph{automorphism} of the shift system. The sets of endomorphisms and
of automorphisms are denoted by $\End(\xnz,\shift{n})$ and $\Aut(\xnz,\shift{n})$
respectively. Under composition, the first is a monoid, and the second a group.

Analogously, a continuous
function from $\xnN$ to itself which commutes with the shift map on this
space is an \emph{endomorphism of the one-sided shift} $(\xnN,\shift{n})$; if it is invertible, it is an \emph{automorphism} of this shift system. The sets of
such maps are denoted by $\End(\xnN,\shift{n})$ and $\Aut(\xnN,\shift{n})$; again
the first is a monoid and the second a group.

Note that $\shift{n}\in\Aut(\xnz,\shift{n})$, whereas
$\shift{n}\in\End(\xnN,\shift{n})\setminus\Aut(\xnN,\shift{n})$.

Define
\begin{eqnarray*}
F_\infty(X_n) &=& \bigcup_{m\ge0}\{f_\infty:f\in F(X_n,m)\},\\
RF_\infty(X_n) &=& \bigcup_{m\ge0}\{f_\infty:f\in F(X_n,m), f
\mbox{ is right permutive}\}.
\end{eqnarray*}

Theorem~\ref{t:hed1} shows that $F_\infty(X_n)\subseteq\End(\xnz)$. In fact $F_{\infty}(X_{n})$ and $RF_{\infty}(X_n)$ are submonoids of $\End(\xnz)$. Note that $\shift{n}\in F_\infty(X_n)$, but $\shift{n}^{-1}$ is not an element of $F_{\infty}(X_{n})$. Now, \cite[Theorem 3.4]{Hedlund69} shows:

\begin{theorem}
$\End(\xnz,\shift{n}) = \{ \shift{n}^{i} \phi \mid i \in \mathbb{Z}, \phi \in F_{\infty}(X_{n}) \}$.
\label{t:hed2}
\end{theorem}

The following result is a corollary:
\begin{theorem}
$RF_\infty(X_n)$ is a submonoid of $\End(\xnz, \shift{n})$ and $\Aut(\xnN,\shift{n})$ is the largest inverse closed subset of $RF_\infty(X_n)$.
\end{theorem}

\section{Connections to transducers}

In this section we state a result of \cite{BleakCameronOlukoya1} which shows how the elements of $F_\infty(X_n)$ can be
described by a certain class of finite \emph{synchronous} transducers. We shall later build on this result to show that elements of $\End(\xnz, \shift{n})$ can be described by a pair consisting of a possibly \emph{asynchronous} transducer, and combinatorial data called an \emph{annotation}. We therefore begin with a general definition of automata and transducers. 

\subsection{Automata and transducers}

An \emph{automaton}, in our context, is a triple $A=(X_A,Q_A,\pi_A)$, where
\begin{enumerate}
\item $X_A$ is a finite set called the \emph{alphabet} of $A$ (we assume that
this has cardinality $n$, and identify it with $X_n$, for some $n$);
\item $Q_A$ is a finite set called the \emph{set of states} of $A$;
\item $\pi_A$ is a function $X_A\times Q_A\to Q_A$, called the \emph{transition
function}.
\end{enumerate}
We regard an automaton $A$ as operating as follows. If it is in state $q$ and
reads symbol $a$ (which we suppose to be written on an input tape), it moves
into state $\pi_A(a,q)$ before reading the next symbol. As this suggests, we
can imagine that the inputs to $A$ form a string in $X^\mathbb{N}$; after
reading a symbol, the read head moves one place to the right before the next
operation.

We can extend the notation as follows. For $w\in X_n^m$, let $\pi_A(w,q)$ be
the final state of the automaton with initial state $q$ after successively
reading the symbols in $w$. Thus, if $w=x_0x_1\ldots x_{m-1}$, then
\[\pi_A(w,q)=\pi_A(x_{m-1},\pi_A(x_{m-2},\ldots,\pi_A(x_0,q)\ldots)).\]
By convention, we take $\pi_A(\varepsilon,q)=q$.

For a given state $q\in Q_A$, we call the automaton $A$ which starts in
state $q$ an \emph{initial automaton}, denoted by $A_q$, and say that it is
\emph{initialised} at $q$.

An automaton $A$ can be represented by a labelled directed graph, whose
vertex set is $Q_A$; there is a directed edge labelled by $a\in X_n$ from
$q$ to $r$ if $\pi_A(a,q)=r$.

A \emph{transducer} is a quadruple $T=(X_T,Q_T,\pi_T,\lambda_T)$, where
\begin{enumerate}
\item $(X_T,Q_T,\pi_T)$ is an automaton;
\item $\lambda_T:X_T\times Q_T\to X_T^*$ is the \emph{output function}.
\end{enumerate}
Such a transducer is an automaton which can write as well as read; after
reading symbol $a$ in state $q$, it writes the string $\lambda_T(a,q)$ on an
output tape, and makes a transition into state $\pi_T(a,q)$. An \emph{initial
transducer} $T_q$ is simply a transducer which starts in state $q$.

In the same manner as for automata, we can extend the notation to allow
the transducer to act on finite strings: let $\pi_T(w,q)$ and $\lambda_T(w,q)$
be, respectively, the final state and the concatenation of all the outputs
obtained when the transducer reads the string $w$ from state $q$.

A transducer $T$ can also be represented as an edge-labelled directed graph.
Again the vertex set is $Q_T$; now, if $\pi_T(a,q)=r$, we put an edge with
label $a|\lambda_T(a,q)$ from $q$ to $r$. In other words, the edge label
describes both the input and the output associated with that edge.

For example, Figure~\ref{fig:shift2} describes a transducer over the alphabet
$X_2$.

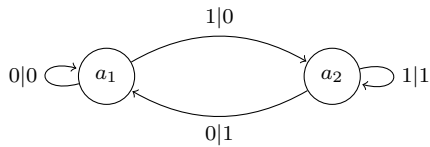
\begin{figure}[htbp]
\begin{center}
 \begin{tikzpicture}[shorten >=0.5pt,node distance=3cm,on grid,auto]
 \tikzstyle{every state}=[fill=none,draw=black,text=black]
    \node[state] (q_0)   {$a_1$};
    \node[state] (q_1) [right=of q_0] {$a_2$};
     \path[->]
     (q_0) edge [loop left] node [swap] {$0|0$} ()
           edge [bend left]  node  {$1|0$} (q_1)
     (q_1) edge [loop right]  node [swap]  {$1|1$} ()
           edge [bend left]  node {$0|1$} (q_0);
 \end{tikzpicture}
 \end{center}
 \caption{A transducer over $X_2$ \label{fig:shift2}}
\end{figure}

A transducer $T$ is said to be \emph{synchronous} if $|\lambda_T(a,q)|=1$
for all $a\in X_T$, $q\in Q_T$; in other words, when it reads a symbol, it
writes a single symbol. It is \emph{asynchronous} otherwise. Thus, an
asynchronous transducer may write several symbols, or none at all, at a given
step. Note that this usage of the word differs from that of Grigorchuk
\emph{et al.}~\cite{GNSenglish}, for whom ``asynchronous'' includes ``synchronous''.
The transducer of Figure~\ref{fig:shift2} is synchronous.

We can regard an automaton, or a transducer, as acting on an infinite string
in $\xnn$, where $X_n$ is the alphabet. This action is given by iterating
the action on a single symbol; so the output string is given by
\[\lambda_T(xw,q) = \lambda_T(x,q)\lambda_T(w,\pi_T(x,q)).\]

Throughout this paper, we will (as in \cite{GNSenglish}) make the following
assumption:

\paragraph{Assumption} A transducer $T$ has the property that, when it reads
an infinite input string starting from any state, it writes an infinite  output string.

\medskip

The property above is equivalent to the property that any circuit in the underlying automaton has non-empty concatenated output. 

From the assumption, it follows that the transducer writes an infinite output string on reading any infinite input string from any state.
Thus $T_q$ induces a map $w\mapsto\lambda_T(w,q)$ from $\xnn$ to itself; it is
easy to see that this map is continuous. If it is a
homeomorphism, then we call the state $q$ a \emph{homeomorphism state}. We write $\im{q}$ for the image of the map induced by $T_{q}$.

Two states $q_1$ and $q_2$ are said to be \emph{$\omega$-equivalent} if the
transducers $T_{q_1}$ and $T_{q_2}$ induce the same continuous map. (This can
be checked in finite time, see~\cite{GNSenglish}.)  More generally, we say that two
initial transducers $T_q$ and $T'_{q'}$ are \emph{$\omega$-equivalent} if they
induce the same continuous map on $\xnn$. 

A transducer is said to be \emph{weakly minimal} if no two states are
$\omega$-equivalent. For a synchronous transducer $T$, two states $q_1$ and $q_2$ are equivalent if $\lambda_T(a, q_1) = \lambda_T(a,q_2)$ for any finite word $a \in X_n^{*}$.
For a state $q$ of $T$ and a word $w\in X_n^*$, we let $\Lambda(w,q)$ be the
greatest common prefix of the set $\{\lambda(wx,q):x\in\xnn\}$. The state
$q$ is called a \emph{state of incomplete response} if
$\Lambda(\varepsilon,q)\ne\varepsilon$; the length $|\Lambda(\varepsilon,q)|$ of the string $\Lambda(\varepsilon,q)$ is the \emph{extent of incomplete response} of the state $q$. Note that for a state $q \in  Q_{T}$, if the initial transducer $T_{q}$  induces a map from $\xnn$ to itself with image size at least $2$, then $|\Lambda(\varepsilon, q)| < \infty$. 

Let  $T$ be a transducer and $q \in Q_{T}$ a state. Suppose $T_{q}$ has image size exactly one. By definition the number of states of $T_{q}$ is finite, therefore there is a state $p \in Q_{T}$, accessible from $q$ and a word $\Gamma \in \xnp$ such that $\pi_{T}(\Gamma, p)= p$. Let $x = \lambda_{T}(\Gamma, p) \in \xnp$ and let $Z_{x} = (\xn, \{x\}, \pi_{Z_{x}}, \lambda_{Z_{x}})$ be the transducer such that $\lambda_{Z_{x}}(i, x) = x$ for all $i \in \xn$. Then $T_{q}$ is $\omega$-equivalent to a transducer $T'_{q'}$ which contains $Z_{x}$ as a subtransducer and satisfies the following condition: there  is a $k \in \N$ so that, for any word $y \in \xns$, $\pi_{T'}(y,q')= x$.

We say that an initial transducer $T_{q}$ is \emph{minimal} if it is weakly minimal, has no states of
incomplete response and every state is accessible from the initial state $q$. A non-initial transducer $T$  is called minimal if for any state $q \in Q_{T}$ the initial transducer $T_{q}$ is minimal.  Therefore a non-initial transducer $T$ is minimal if it is weakly minimal, has no states of incomplete response and is strongly connected as a directed graph. 

Observe that the transducer $Z_{x}$, for any $x \in \xnp$, is not minimal as  its only state is, by definition, a state of incomplete response. However, it will be useful when stating some results later on to have a notion of minimality for the single state transducers $Z_{x}$ considered as initial and non-initial transducers. We shall use the word minimal for this new meaning of minimality appealing to the context to clarify any confusion. The definition below, which perhaps appears somewhat contrived, is a consequence of this.  

Let $x= x_1 x_2 \ldots x_{r} \in \xns$ be a word. The transducer $Z_{x}$ is called \emph{minimal as an initial transducer} if $x$ is a prime word; $Z_{x}$ is called \emph{minimal as a non-initial transducer} if  $Z_{x}$ is minimal as an initial transducer and $x$ is the minimal  element, with respect to the lexicographic ordering, of the set $\{ x_{i} \ldots x_{r}x_{1}\ldots x_{i-1} \mid 1 \le i \le r \}$. Note that if $|x| = 1$ then $Z_{x}$ is minimal as an initial transducer if and only if it is minimal as a non-initial transducer. In most applications, we consider the transducers $Z_{x}$ as non-initial transducers and  in this case we will omit the phrase `as a non-initial transducer'.

 Two (weakly) minimal non-initial transducers $T$ and $U$  are said to be  \emph{$\omega$-equal} if there is a bijection $f: Q_{T} \to Q_{U}$, such that for any $q \in Q_{T}$, $T_{q}$ is $\omega$-equivalent to $U_{(q)f}$. Two (weakly) minimal initial transducers $T_{p}$ and $U_{q}$ are said to be $\omega$-equal if there is a bijection  $f: Q_{T} \to Q_{U}$, such that $(p)f = q$ and for any $t \in Q_{T}$,  $T_{t}$ is $\omega$-equivalent to $U_{(t)f}$. We shall use the symbol `$=$' to represent $\omega$-equality of initial and non-initial transducers. Two non-initial transducers are said to be $\omega$-equivalent if they have $\omega$-equal minimal representatives.

In the class of synchronous transducers, the  $\omega$-equivalence class of  any transducer has a unique weakly minimal representative.
Grigorchuk \textit{et al.}~\cite{GNSenglish} prove that the \mbox{$\omega$-equivalence}
class of an initialised transducer $T_q$ has a unique minimal representative, if one permits infinite outputs from finite inputs,
and give an algorithm for computing this representative.  The first step of this algorithm is to create a new transducer $T_{q_{-1}}$ which is $\omega$-equivalent to $T_{q}$ and has no states of incomplete response. As we do not allow our transducers to write infinite strings on finite inputs,  we impose an additional condition: we apply the process for removing incomplete response only in instances where all states accessible from the initial one have finite extent of incomplete response.  

\begin{prop}\label{prop:algorithmforremovingincompleteresponse}
	Let $A_{q_0} = (X_n, Q_A, \pi_A, \lambda_A)$ be a finite initial transducer such that $|\Lambda(\varepsilon, q)|<\infty$ for all $q \in  Q_{A}$. Set $Q'_A:= Q \cup\{q_{-1}\}$, where $q_{-1}$ is a symbol not in $Q_{A}$.  Define:
	\begin{IEEEeqnarray}{rCl}
		\pi'_A(x,q_{-1}) &=& \pi_{A}(x,q_0), \label{transition same}\\
	    \pi'_A(x,q) &=& \pi_{A}(x,q),\\
		\lambda'_A(x, q_{-1}) &= & \Lambda(x,q_0),\\
		\lambda'_A(x,q) &=& \Lambda(x, q) - \Lambda(\varepsilon,q), \label{redifining output}
	\end{IEEEeqnarray}
	for all $x \in X_n$ and $q \in Q_{A}$.
	
	Then, the transducer $A' = (X_n, Q'_A, \pi'_A, \lambda'_A)$ with the initial state $q_{-1}$ is a transducer which is $\omega$-equivalent to the transducer $A_{q_0}$ and has no states of incomplete response.
	
	In  addition, for every $w \in X_n^{\ast}$, the following inequality holds:
	\begin{equation}
	\lambda'_A(w, q_{-1}) = \Lambda(w, q_0). \label{resulthasnoincompleteresponse}
	\end{equation}
\end{prop}

\medskip

Throughout this article, as a matter of convenience, we shall not distinguish between \mbox{$\omega$-equivalent} transducers. Thus, for example, we introduce various groups as if the elements of those groups are transducers and not $\omega$-equivalence classes of transducers. 

Given two transducers $T=(X_n,Q_T,\pi_T,\lambda_T)$ and
$U=(X_n,Q_U,\pi_U,\lambda_U)$ with the same alphabet $X_n$, we define their
product $T*U$. The intuition is that the output for $T$ will become the input
for $U$. Thus we take the alphabet of $T*U$ to be $X_n$, the set of states
to be $Q_{T*U}=Q_T\times Q_U$, and define the transition and rewrite functions
by the rules
\begin{eqnarray*}
\pi_{T*U}(x,(p,q)) &=& (\pi_T(x,p),\pi_U(\lambda_T(x,p),q)),\\
\lambda_{T*U}(x,(p,q)) &=& \lambda_U(\lambda_T(x,p),q),
\end{eqnarray*}
for $x\in X_n$, $p\in Q_T$ and $q\in Q_U$. Here we use the earlier 
convention about extending $\lambda$ and $\pi$ to the case when the transducer
reads a finite string.  If $T$ and $U$ are initial with initial states $q$ and $p$ respectively then the state $(q,p)$ is considered the initial state of the product transducer $T*U$.

We say that an initial transducer $T_q$ is \emph{invertible} if there is an initial transducer $U_p$ such that $T_q*U_p$ and $U_p*T_q$ each induce the identity map
on $\xnn$. We call $U_p$ an \emph{inverse} of $T_q$.  When this occurs we will denote $U_p$ as $T_q^{-1}$.

In automata theory a synchronous (not necessarily initial) transducer $$T = (X_n, Q_{T}, \pi_{T}, \lambda_T)$$ is \emph{invertible} if for any state $q$ of $T$, the map $\rho_q:=\pi_{T}(\centerdot, q): X_{n} \to X_{n}$ is a bijection. In this case the inverse of $T$ is the transducer $T^{-1}$ with state set $Q_{T^{-1}}:= \{ q^{-1} \mid q \in Q_{T}\}$, transition function $\pi_{T^{-1}}: X_{n} \times Q_{T^{-1}} \to Q_{T^{-1}}$ defined by $(x,p^{-1}) \mapsto q^{-1}$ if and only if $\pi_{T}((x)\rho_{p}^{-1}, p) =q$, and output function  $\lambda_{T^{-1}}: X_{n} \times Q_{T^{-1}} \to X_{n}$ defined by  $(x,p) \mapsto (x)\rho_{p}^{-1}$. 

In this article, we will come across synchronous transducers which are not invertible in the automata theoretic sense but which nevertheless induce self-homeomorphims of the space $\xnz$ and so are invertible in a different sense.

\subsection{Synchronizing automata and bisynchronizing transducers}

Given a natural number $k$, we say that an automaton $A$ with alphabet $X_n$
is \emph{synchronizing at level $k$} if there is a map
$\mathfrak{s}_{k}:X_n^k\mapsto Q_A$ such that, for all $q$ and any word
$w\in X_n^k$, we have $\pi_A(w,q)=\mathfrak{s}_{k}(w)$; in other words, if the
automaton reads the word $w$ of length $k$, the final state depends only on
$w$ and not on the initial state. (Again we use the extension of $\pi_A$ to
allow the reading of an input string rather than a single symbol.) We
call $\mathfrak{s}_{k}(w)$ the state of $A$ \emph{forced} by $w$; the map $\mathfrak{s}_{k}$ is called the \emph{synchronizing map at level $k$}. An automaton $A$ is called \emph{strongly synchronizing} if it is synchronizing at level $k$ for some $k$.

We remark here that the notion of synchronization occurs in automata theory
in considerations around the \emph{\v{C}ern\'y conjecture}, in a weaker sense.
A word $w$ is said to be a \emph{reset word} for $A$ if $\pi_A(w,q)$ is
independent of $q$; an automaton is called \emph{synchronizing} if it has
a reset word~\cite{Volkov2008,ACS}. Our definition of ``synchonizing at level $k$''/``strongly synchronizing"
requires every word of length $k$ to be a reset word for the automaton.

If the automaton $A$ is synchronizing at level $k$, we define the
\emph{core} of $A$ to be the set of states forming the image of the map
$\mathfrak{s}_{k}$. It is an easy observation that, if $A$ is synchronizing at
level $k$, then its core is an automaton in its own right, and is also 
synchronizing at level $k$. We denote this automaton by $\core(A)$. Moreover,
if $T$ is a transducer which is synchronous and which (regarded as an
automaton) is synchronizing at level $k$, then the core of $T$ (similarly denoted 
 $\core(T)$) induces a continuous map $f_T:\xnz\to\xnz$. We say that an
automaton or transducer is \emph{core} if it is equal to its core.

Clearly, if $A$ is synchronizing at level $k$, then it is synchronizing at
level~$l$ for all $l\ge k$; but the map $f_T$ is independent of the level
chosen to define it.

Let $T_q$ be an initial transducer which is invertible with inverse $T_q^{-1}$. If $T_q$ is synchronizing at level $k$, and $T_q^{-1}$ is synchronizing at level $l$, 
we say that $T_q$ is \emph{bisynchronizing} at level $(k,l)$. If $T_q$ is
invertible and is synchronizing at level~$k$ but not bisynchronizing, we say
that it is \emph{one-way synchronizing} at level~$k$.

For a non-initial synchronous and invertible transducer  $T$ we also say $T$ is bi-synchronizing (at level $(k,l)$) if both $T$ and its inverse $T^{-1}$ are synchronizing at levels $k$ and $l$ respectively.

\begin{ntn}
	Let $T$ be a transducer which is synchronizing at level $k$ and Let $l \ge k$ be any natural number. Then for any word $\Gamma \in X_{n}^{l}$, we write $q_{\Gamma}$ for the state $\mathfrak{s}_{l}(\Gamma)$, where $\mathfrak{s}_{l}: X_{n}^{l} \to Q_{T}$ is the synchronizing map at level $l$.
\end{ntn}

The following result was proved in Bleak \textit{et al.}~\cite{AutGnr} .

\begin{prop}
If $T_q$ and $U_p$ are synchronous initial transducers which (as automata) are synchronizing at levels $j$ and $k$ respectively then $T*U$ is synchronizing at level at most $j+k$.
\label{p:synchlengthsadd}
\end{prop}

Before proceeding, we  explicitly define the function $f_{T}: \xnz \to \xnz$ induced by a strongly synchronizing transducer $T$. 

Let $T$ be a transducer which is core, synchronous and which is synchronizing at level $k$. The map $f_{T}$ maps an element $x \in \xnz$ to the sequence $y$ defined by $y_{i} = \lambda_T(x_i, q_{x_{i-k}x_{i-k+1}\ldots x_{i-1}})$.  We observe that for an element  $x \in  \xnz$, and $i \in \Z$, if $ (x)f_{T} = y$, then, by definition,  $y_{i}y_{i+1}\ldots = \lambda_{T}(x_i x_{i+1}\ldots, q_{x_{i-k}x_{i-k+1}\ldots x_{i-1}})$.

Now strongly synchronizing transducers may induce endomorphisms of the shift \cite{BleakCameronOlukoya1}:

\begin{prop}\label{prop:pntildeisinendo}
Let $T$ be a minimal synchronous transducer which is  synchronizing at 
level $k$ and which is core. Then $f_T\in\End(\xnz,\shift{n})$.
\end{prop}

The transducer in Figure~\ref{fig:shift2} induces the shift map on $X_{2}^{\Z}$. More generally, let  $\Shift{n} = (\xn, Q_{\Shift{n}}, \pi_{\Shift{n}}, \lambda_{\Shift{n}})$ be the transducer defined as follows. Let $$Q_{\Shift{n}}:= \{0,1,2,\ldots, n-1\},$$ and let $\pi_{\Shift{n}}: \xn \times Q_{\Shift{n}}\to Q_{\Shift{n}}$  and $\lambda_{\Shift{n}}: \xn \times Q_{\Shift{n}}\to \xn$ be defined by $\pi_{\Shift{n}}(x, i) = x$ and $\lambda_{\Shift{n}}(x, i) = i$ for all $x \in X_{n}$ $i \in Q_{\Shift{n}}$. Then $f_{\Shift{n}} = \shift{n}$.

In \cite{AutGnr}, the authors show that the set $\spn{n}$ of weakly
minimal finite synchronous core transducers is a monoid; the monoid operation
consists of taking the product of transducers and reducing it by removing non-core states and identifying $\omega$-equivalent states to obtain a weakly minimal and synchronous representative. Let $\mathcal{P}_n$ be the subset
of $\spn{n}$ consisting of transducers which induce automorphisms
of the shift. (Note that these may not be minimal.) Clearly $\Shift{n} \in \pn{n}$.

\subsection{De Bruijn graphs and $\End(\xnz, \shift{n})$}

The \emph{de Bruijn graph} $G(n,m)$ can be defined as follows, for integers
$m\ge1$ and $n\ge2$. The vertex set is $X_n^m$, where $X_n$ is the alphabet
$\{0,\ldots,n-1\}$ of cardinality $n$. There is a directed arc from
$a_{0}a_1\ldots a_{m-1}$ to $a_1\ldots a_{m-1}a_{m}$, with label $a_0$. Note that, in the literature, this directed edge is labelled by the
$(m+1)$-tuple $a_0a_1\ldots a_{m-1}a_m$. 

Figure~\ref{fig-DB-3-2-straight} shows the de Bruijn graph $G(3,2)$.

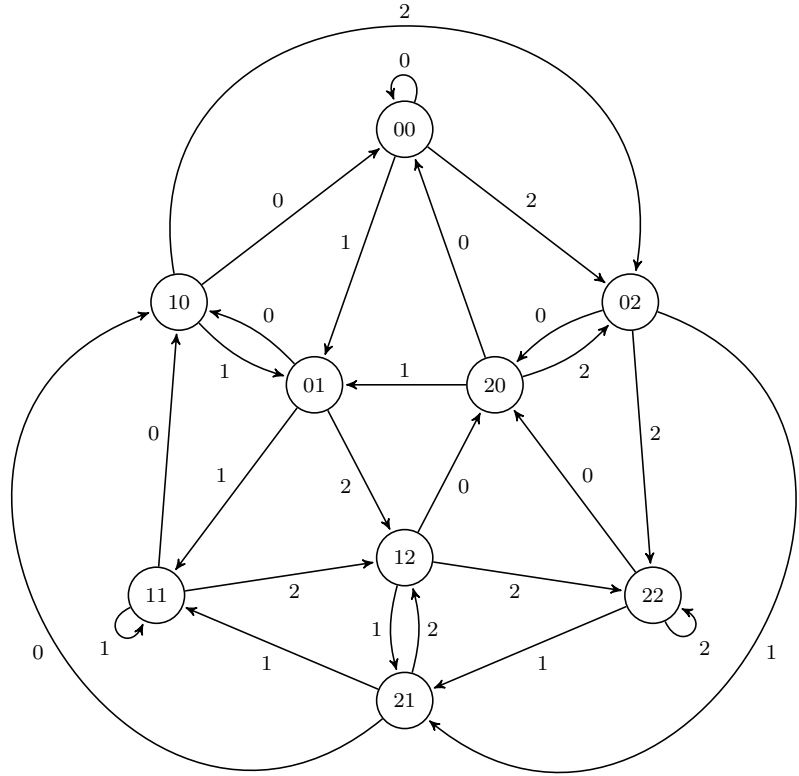
\begin{figure}[htbp]
\begin{center}
\begin{tikzpicture}[->,>=stealth',shorten >=1pt,auto,node distance=2.3cm,on grid,semithick,every state/.style={draw=black,text=black}]
   \node[at={(0,2.9)},state] (a) {$00$}; 
   \node[at={(-3.3,-3.3)},state] (b)  {$11$}; 
   \node[at={(3.3,-3.3)},state] (c) {$22$}; 
   \node[at={(-1.2,-0.5)},state] (d)   {$01$}; 
   \node[at={(-3.0,0.6)},state] (f)  {$10$}; 
   \node[at={(3.0,0.6)},state] (e) {$02$}; 
   \node[at={(1.2,-0.5)},state] (h)  {$20$}; 
   \node[at={(0,-2.8)},state] (g)  {$12$}; 
   \node[at={(0,-4.7)},state] (i){$21$}; 

    \path[use as bounding box] (-3.3,-6.0) rectangle (3.3,4.5);
    \path (a) edge [out=70,in=110,loop,min distance=0.5cm]node [swap]{$0$} (a);
    \path (a) edge node [swap]{$1$} (d);
    \path (a) edge node {$2$} (e);
    \path (b) edge [out=205,in=245,loop,min distance=0.5cm]node [swap]{$1$} (b);
    \path (b) edge node {$0$} (f);
    \path (b) edge node [swap] {$2$} (g);
    \path (c) edge [out=-65,in=-25,loop,min distance=0.5cm]node [swap]{$2$} (c);
    \path (c) edge node {$1$} (i);
    \path (c) edge node [swap]{$0$} (h);
    \path (d) edge node [swap] {$1$} (b);
    \path (d) edge [bend right=15] node [swap]{$0$} (f);
    \path (d) edge node[swap] {$2$} (g);
    \path (e) edge [bend left=100,min distance=4.45cm] node {$1$} (i);
    \path (e) edge [bend right=15] node [swap] {$0$} (h);
    \path (e) edge  node {$2$} (c);
    \path (f) edge node {$0$} (a);
    \path (f) edge [bend right=15] node [swap]{$1$} (d);
    \path (f) edge [bend left=100,min distance=4.45cm] node {$2$} (e);
    \path (g) edge node[swap] {$0$} (h);
    \path (g) edge node [swap] {$2$} (c);
    \path (g) edge [bend right=15] node [swap] {$1$} (i);
    \path (h) edge [bend right=15] node [swap] {$2$} (e);
    \path (h) edge node [swap] {$0$} (a);
    \path (h) edge node [swap] {$1$} (d);
    \path (i) edge [bend left=100,min distance=4.45cm] node {$0$} (f);
    \path (i) edge [bend right=15] node [swap] {$2$} (g);
    \path (i) edge node {$1$} (b);
\end{tikzpicture}
\end{center}
\caption{The de Bruijn graph $G(3,2)$.\label{fig-DB-3-2-straight}}
\end{figure}

Observe that the de Bruijn graph $G(n,m)$ describes an automaton over the
alphabet~$X_n$. Moreover, this automaton is synchronizing at level $m$: when
it reads the string $b_0b_1\ldots b_{m-1}$ from any initial state, it moves
into the state labelled ${b_0b_1\ldots b_{m-1}}$.

We now describe how to make the de Bruijn automaton into a transducer by
specifying outputs. Let $f\in F(X_n,m+1)$ be a function from $X_n^{m+1}$ to
$X_n$. The output function of the transducer $T_f$ will be given by 
\[\lambda_T(x,a_0a_1\ldots a_{m-1})=(a_0a_1\ldots a_{m-1}x)f.\]

In other words, if the transducer reads $m+1$ symbols, then its output is
obtained by applying $f$ to the sequence of symbols read. Note that this
transducer is synchronous; it writes one symbol for each symbol read. When
applied to $x\in\xnz$, it produces $y=(x)f_\infty\in\xnz$. Recall that the function $f \in F(X_{n}, 2)$ given by $(x_1x_2)f = x_1$ for all $x_1,x_2 \in X_n$, induces the shift map $\shift{n}$ on $\xnz$. For this map we have $T_{f} = \Shift{n}$.

\begin{prerk}\label{bijectionfromFinftytoPn}
Given any de Bruijn graph $G(n,m)$, and any transducer $T$ with underlying
directed graph $G(n,m)$ there is a function $f \in F(X_n, m+1)$ such that
$T_f = T$.
\end{prerk}

Clearly the transducer $T_f$ is synchronizing at level $m$. This remains true 
if we minimise it or identify its $\omega$-equivalent states. 
Let $T \in \spn{n}$ be the weakly-minimal representative of $T_{f}$, then $f_{T} = f_{T_f} = f_{\infty}$ holds since identifying $\omega$-equivalent states does not affect the map $f_{T}$. 

\begin{prerk}
The preceding paragraph together with  Remarks \ref{F(X_n,m)containedinF(X_n,m+1)} and \ref{bijectionfromFinftytoPn} show that there is a bijection from $F_{\infty}(X_n)$ to $\spn{n}$. The result below from \cite{BleakCameronOlukoya1} states that this bijection is a monoid homomorphism.
\end{prerk}
 
\begin{prop}\label{prop:homomorphismfromPntoFinfinitiy}
Let $A,B\in\spn{n}$. Then $ f_A \circ f_B = f_{A\ast B}$.
\end{prop}
\begin{cor}\label{cor:FinftyisoPn}\cite{BleakCameronOlukoya1}
	The monoid $F_{\infty}(X_{n})$ is isomorphic to $\spn{n}$.
\end{cor}

\section{Automorphisms of the two-sided shift dynamical system and outer-automorphisms of the Higman--Thompson groups.} \label{Section:main}

The automorphisms of the shift dynamical system can be given by the following
extension:
\begin{equation*}
1 \rightarrowtail \gen{\shift{n}} \rightarrowtail \aut{\xnz, \shift{n}}
\twoheadrightarrow \aut{\xnz, \shift{n}}/\gen{\shift{n}} \twoheadrightarrow 1
\end{equation*}

In the what follows we shall show
that the group $\aut{\xnz,\shift{n}}/ \gen{\shift{n}}$ is isomorphic to a
subgroup of  $\out{G_{n,n-1}}$. In fact, we prove something a bit more general. Observe that the equivalence relation on $\End(\xnz, \shift{n})$  which relates two elements when the first is a power of the shift times the second, is the semigroup congruence generated by the set of pairs $\{ (\shift{n}^{i}, 1) \mid i \in \Z \}$. Abusing notation, write $\gen{ \shift{n}}$ for this congruence, and denote by $\End(\xnz, \shift{n})/ \gen{\shift{n}}$ the quotient semigroup obtained by factoring out this congruence.  We show that $\End(\xnz,\shift{n})/\gen{\shift{n}}$ is isomorphic as a monoid to a monoid containing $\out{G_{n,n-1}}$.
It follows from Theorem \ref{t:hed2} that each element of
$\aut{\xnz, \shift{n}}/\gen{\shift{n}}$ has a unique representative in
$F_{\infty}(X_n)$, and so a unique representative in $\mathcal{P}_n$ by
Proposition~\ref{bijectionfromFinftytoPn}. Our  strategy, thus, will be to correct the
incomplete response in the states of elements of~$\mathcal{P}_n$. This will, for instance,
transform the transducer, an element of $\pn{n}$, depicted in  Figure
\ref{fig:shift2}, into the identity transducer. We then use results of \cite{AutGnr} showing that $\out{G_{n,n-1}}$ is isomorphic to a group $\On{n}$ of transducers which is contained in a monoid $\Mn{n}$ of transducers. The guiding philosophy in our approach is that powers of the shift are encoded as incomplete response in the states of elements of $\spn{n}$. Since the inverse of an element of $\mathcal{P}_n$, as a homeomorphism of $\xnz$, is a negative power of the shift times another element of $\pn{n}$, we expect that  by correcting the incomplete response in the states of elements of $\pn{n}$, the resulting transducers should be `invertible' as non-initial transducers. The rest of this work is devoted to formalising this principle.  

\subsection{The groups $\On{n}$ and $\Ln{n}$}
The paper \cite{AutGnr} shows that the outer-automorphism group $\out{G_{n,n-1}}$ of the Higman--Thompson group $G_{n,n-1}$ is isomorphic to a group  $\On{n}$ of non-initial transducers and contains an isomorphic copy of $\out{G_{n,r}}$ for all $1 \le r < n-1$. The remainder of this paper shall be devoted to showing that the quotient of $\aut{\xnz, \shift{n}}$ by its centre is isomorphic to a subgroup $\Ln{n}$ of $\On{n}$. We begin by defining, as in \cite{AutGnr}, a monoid containing~$\On{n}$.

Let $\Mn{n}$ be the set of all {\bfseries{non-initial}}, minimal, strongly synchronizing and core transducers. Recall that for $x \in \xnp$, $Z_{x}= (\xn, \{x\}, \pi_{Z_{x}}, \lambda_{Z_{x}})$ is the transducer such that $\lambda_{Z_{x}}(i,x) = x$ for all $i \in \xn$. Moreover $Z_{x} \in \Mn{n}$ for a prime word $x = x_0 x_1 \ldots x_r \in \xnp$ if and only if $x$ is the minimal, with respect to the lexicographic ordering, element of the set $\{ x_{i}\ldots x_{r}x_1 \ldots x_{i-1} \mid 1 \le i \le r \}$.  If $T \in \Mn{n}$ is a strongly synchronizing and core transducer, such that some state (and so all states by connectivity) has infinite extent of incomplete response, then $T = Z_{x}$ for some minimal non-initial transducer $Z_{x}$.   Let  $\SOn{n}$ be the subset of $\Mn{n}$ consisting of  all  transducers $T$ satisfying the following conditions:

\begin{enumerate}[label = {\bfseries S\arabic*}]
	\item for any state $t \in Q_{T}$, the initial transducer $T_{t}$ induces an injective map $h_{t}: \xnn \to \xnn$ and,
	\item for any state $t \in Q_{T}$ the map $h_{t}$ has clopen image which we denote by $\im{t}$.
\end{enumerate}
Note that an element of $\Mn{n}$ with more than one state has no states of incomplete response. The single state identity transducer is an element of $\Mn{n}$ and we denote it by   $\id$. 

We define a product, extending a definition given in \cite{AutGnr}, on the set $\Mn{n}$ as follows. For two transducers $T,U \in \Mn{n}$, the product transducer $T \ast U$ is strongly synchronizing. Let $TU$ be the transducer obtained from $\core(T\ast U)$ as follows. Fix a state $(t,u)$ of $\core(T\ast U)$. Then $TU$ is the core of the minimal representative of the initial transducer $\core(T\ast U)_{(t,u)}$. That is, if any state (and so all states by connectivity) of $\core(T\ast U)$ has finite extent of incomplete response  we remove the states of incomplete response from $\core(T\ast U)_{(t,u)}$ to get a new transducer initial $(TU)'_{q_{-1}}$ which is still strongly synchronizing and has all states accessible from the initial state. Then we identify the $\omega$-equivalent states of $(TU)'_{q_{-1}}$ to get a minimal, strongly synchronizing initial transducer $(TU)_{p}$. The transducer $TU$ is the core of $(TU)_{p}$. Otherwise, for any state $(t,u) \in \core(T,U)$, $\core(T,U)_{(t,u)}$ is $\omega$-equivalent to a transducer of the form $Z_{x}$ for some prime word $x = x_1 \ldots x_r \in \xnp$ and we set $TU = Z_{\overline{x}}$ where $\overline{x}$ is the minimal element, with respect to the lexicographic ordering, of the set $\{ x_i \ldots x_{r}x_1 \ldots _{i-1}\mid 1 \le i \le r \}$.

It is a result of \cite{AutGnr} that $\SOn{n}$ together with the binary operation  $\SOn{n} \times \SOn{n} \to \SOn{n}$ given by $(T,U) \mapsto TU$ is a monoid. (The single state transducer $\id$ which induces the identity permutation on $X_n$ is the identity element.) We prove below that this result extends to $\Mn{n}$ as well.

\begin{prop}
	The set $\Mn{n}$ together with the binary operation $(T,U) \mapsto TU$ is a monoid.
\end{prop} 
\begin{proof}
	It suffices to show that the binary operation defines an associative product.  We consider two cases.
	
	Let $A, B, C \in \Mn{n}$. First consider the case that the product $\core((A \ast B)\ast C)$ has a state with infinite extent of incomplete response.  
	
	Let $a, b, c$ be states of $A$, $B$ and $C$ respectively. We may choose $a \in Q_{A}$, $b \in Q_{B}$ and $c \in Q_{C}$ such that $(a,b) \in \core(A \ast B)$ and $(b,c) \in \core(B \ast C)$. By definition of the transducer product, the initial transducers $((A \ast B) \ast C)_{((a,b),c)}$ and $(A \ast (B \ast C))_{(a,(b,c))}$ are $\omega$-equivalent since they both induce the map $h_{a}h_{b}h_{c}$ on $\xn^{\omega}$. Let $y \in \xnp$ be such that $(A \ast B)\ast C)_{((a,b),c)}$ and $(A \ast (B \ast C))_{(a,(b,c))}$  are $\omega$-equivalent to $Z_{y}$.  
	
	Now suppose that $\core(A \ast B)_{(a,b)}$ is $\omega$-equivalent to $Z_{x}$ for some $x \in \xnp$ so that $AB = Z_{\overline{x}}$ for $\overline{x}$ a rotation of $x$ so that $Z_{\overline{x}}$ is minimal. As  $((A \ast B) \ast C)_{(a,b,c)}$ is $\omega$-equivalent to $Z_{y}$ and since $\core(\core(A \ast B) \ast C)$ is equal to  $\core((A \ast B) \ast C)$, it follows that $Z_{\overline{x}}C = (AB)C = Z_{\overline{y}}$ for $\overline{y}$ a rotation of $y$ so that $Z_{\overline{y}}$ is minimal. 
	
	Therefore, we may assume that all states of $\core(A \ast B)$ have finite extent of incomplete response. Let $(AB)_{e}$ be the transducer obtained by applying the algorithm of Proposition~\ref{prop:algorithmforremovingincompleteresponse} to the initial transducer $(A \ast B)_{(a,b)}$ and identifying $\omega$-equivalent states of the result. Recall that $AB = \core((AB)_{e})$.  Let $c' \in Q_{C}$ be such that $((a,b),c')$ is a state of $\core(\core(A \ast B) \ast C)$. Since $h_{e}$ is $\omega$-equivalent to $h_{(a,b)}$, we have, as $((A\ast B) \ast C)_{((a,b),c)}$ is $\omega$-equivalent to $Z_{y}$, that $h_{e}h_{c'} = h_{((a,b),c')}$ has precisely the element $(y')^{\omega}$ in its image for some rotation $y'$ of $y$. From this we deduce, since $\core (AB \ast C) = \core(((AB) \ast C)_{(e,c')})$, that, $(AB)C = Z_{\overline{y}}$ in this case also.
	
	In a similar way, making use of the fact that the transducer $(A \ast (B \ast C)))_{(a,(b,c))}$ is $\omega$-equivalent to $Z_{y}$ as well, we also deduce that $A(BC) = Z_{\overline{y}}$.
	
	The case where $\core((A \ast B) \ast C)$ (and so $\core(A \ast (B \ast C))$) has no state with infinite extent of incomplete response is dealt with as in the proof of Proposition 9.4 of the paper \cite{AutGnr} and so we omit it.  
\end{proof}

A consequence of a  construction in Section 9 of  \cite{AutGnr} is that an element  $T \in \Mn{n}$ is an element of $\SOn{n}$ if and only if  for any state $p \in Q_T$ there is a minimal, initial transducer $U_{q}$ such that the minimal representative of the product $(T\ast U)_{(p,q)}$ is strongly synchronizing and has trivial core.  If the transducer $U_{q}$ is strongly synchronizing as well then $\core(U_{q}) \in \SOn{n}$ and satisfies $T\core(U_q) = \core(U_q)T =\id$. 

The group $\On{n}$ is the largest inverse closed subset of $\SOn{n}$. 

Let $\SLn{n}$ be those elements $T \in \Mn{n}$ which satisfy the following additional Lipschitz constraint:

\begin{enumerate}[label= {\bfseries SL\arabic*}]
	\setcounter{enumi}{2}
	\item for all strings $a \in X_n^{\ast}$ and any state $q \in Q_{T}$ such that $\pi_{T}(a,q) = q$, $|\lambda_{T}(a,q)| = |a|$.
	\label{Lipshitzconstraint}
\end{enumerate}

The set $\SLn{n}$ is a submonoid of $\Mn{n}$. Set $\Ln{n} := \On{n} \cap \SLn{n}$, the largest inverse closed subset of $\SLn{n}$.

For each $1 \le r \le n-1$ the group $\On{n}$ has a subgroup $\On{n,r}$ where $\On{n,n-1} = \On{n}$; we define $\Ln{n,r}:= \On{n,r} \cap \Ln{n}$. The following theorem is from \cite{AutGnr}.
\begin{theorem}
	For $1 \le r \le n-1$, $\On{n,r} \cong \out{G_{n,r}}$.
\end{theorem}
In subsequent sections we establish a map from the monoid $\pn{n}$ to the group $\On{n}$. Theorem~\ref{t:hed2} and Corollary~\ref{cor:FinftyisoPn} show that elements of the quotient $\aut{\xnz, \shift{n}}/\gen{\shift{n}}$ can be represented by elements of $\pn{n}$, therefore the above mentioned map will be crucial in demonstrating that the group $\aut{\xnz, \shift{n}}/ \gen{\shift{n}}$ is isomorphic to $\Ln{n}$. 

In what follows it will be necessary to distinguish between the binary operation in $\spn{n}$ and the operation in $\SLn{n}$. For two elements $P, R \in \spn{n}$, we write $P \spnprod{n} R$ for the element of $\spn{n}$ which is obtained by taking the full transducer product $P * R$, identifying $\omega$-equivalent states and taking the core of the resulting weakly minimal transducer.

Table~\ref{tab:monoids} summarises the groups and monoids defined so far. The checkmark means that the requirement stipulated in that column is an enforced condition for membership in the corresponding monoid. A ``\_'' symbol means the requirement in a column is not enforced, i.e., some elements of the monoid may fail to have the given property. We note, that as an action on $\xnz$ is defined, so far, only for elements of $\spn{n}$, then, the ``induce homeomorphism of $\xnz$'' condition only applies to the monoids $\spn{n}$ and $\pn{n}$. Note also that elements of the monoids $\Mn{n}$, $\SOn{n}$, $\On{n}$, $\SLn{n}$, $\Ln{n}$ are represented by minimal transducers, whereas, elements of the monoids $\spn{n}$ and $\pn{n}$ are represented by weakly minimal transducers.

\begin{center}
\begin{figure}[t] 	\begin{TAB}(@,0.2cm,0.2cm)[1pt,9cm,9cm]{|c|c|c|c|c|c|c|}{|c|c|c|c|c|c|c|c|} 
		
		{} & \makecell{Strongly \\ Synchronizing} & \makecell{Bi-synchronizing \\ or induce  \\ homeomorphism \\ of $\xnz$} & \makecell{Injective \\ states with \\  clopen image} & Synchronous & Lipschitz & Bi-lipschitz \\
		$\Mn{n}$ &\checkmark &\_ &\_ & \_ & \_  &\_ \\
		$\SOn{n}$ & \checkmark & \_& \checkmark & \_ & \_ & \_ \\
		$\On{n}$ & \checkmark & \checkmark & \checkmark & \_& \_  & \_ \\
		$\SLn{n}$ & \checkmark & \_ & \checkmark & \_ & \checkmark  & \_ \\
		$\Ln{n}$ & \checkmark & \checkmark & \checkmark & \_ & \checkmark  & \checkmark  \\
		$\spn{n}$ & \checkmark & \_ & \_ & \checkmark & \checkmark  & \_ \\
		$\pn{n}$  & \checkmark & \checkmark & \checkmark & \checkmark & \checkmark & \checkmark 
	\end{TAB}
	\caption{Monoids and properties}
	\label{tab:monoids}
	\end{figure}
\end{center}

\subsection{From $\spn{n}$ to $\SLn{n}$}

Let $P \in \spn{n}$ and let $p \in Q_{P}$ be any state. Note that if $|\Lambda(\varepsilon,p')| = \infty$ for any (and so all by connectivity) states of $P$ then, as $\spn{n}$ consists of weakly minimal transducers, $|P| = 1$, $P$ is minimal and is in  $\SLn{n}$. If $|P|>1$, denote by $[P]$ the transducer obtained from $P$ by performing the process in Proposition~\ref{prop:algorithmforremovingincompleteresponse} to the initial transducer $P_{p}$ and taking the core of the resulting transducer $P_{p_{-1}}$ and identifying equivalent states.   By construction, the set of states of the transducer $P_{p_{-1}}$ is precisely the set $Q_{P} \sqcup Q_{p_{-1}}$. As $P$ is strongly synchronizing, it follows by equation \eqref{transition same} of Proposition~\ref{prop:algorithmforremovingincompleteresponse}, that every state of $P_{p_{-1}}$ is accessible from the state $p_{-1}$ and that $P_{p_{-1}}$ is strongly synchronizing as well. (In fact the set of states of $[P]$ is precisely the set $Q_{P}$.) However, the  output function $\lambda'_{P}$ of $P_{p_{-1}}$ restricted to $[P]$ is now defined by the rule given in equation \eqref{redifining output} of Proposition~\ref{prop:algorithmforremovingincompleteresponse}. It is the case that $[P]$ is minimal and so an element of $\Mn{n}$. We show below that $[P]$ is in fact an element of $\SLn{n}$ and if   $P \in \pn{n}$ then $[P] \in \Ln{n}$. If $|P| = 1$ then set $[P]:= P$.

We begin with the first statement.

\begin{prop}
	Let $P \in \spn{n}$. Then $[P] \in \SLn{n}$.\label{prop:frompntoln}
\end{prop}
\begin{proof}
	We may assume that $|P| >1$ since otherwise $P \in \SLn{n}$. Thus all states of $P$ have finite extent of incomplete response and Proposition~\ref{prop:algorithmforremovingincompleteresponse} applies. 
	
	It suffices to show that there is some $k \in \N$ so that for any pair $\Gamma \in \xn^{\ast}$ of length at least $k$,  and  $t \in Q_{[P]}$ for which $\pi_{[P]}(\Gamma, t) = t$, then it is also the case that $|\lambda_{[P]}(\Gamma, t)| = |\Gamma|$. 
	
	We make a further reduction. Let $q \in Q_{P}$ be any state of $P$ and form $P'_{q_{-1}}$ as in Proposition \ref{prop:algorithmforremovingincompleteresponse}. Note that as $[P]$ is equal to the transducer obtained from  $\core(P'_{q_{-1}})$ by identifying $\omega$-equivalent states, then if $\core{P'_{q_{-1}}}$ satisfies condition \ref{Lipshitzconstraint}, then $[P] \in \SLn{n}$.
	
	Let $k \in \mathbb{N}$ be large enough, and larger than the synchronizing level of $P$, such that for any word $\Gamma \in \xns$ of length at least $k$, and any state $q \in Q_{P}$, $\lambda_{P}(\Gamma, q)$ is longer than the extent of incomplete response of any state of $P$. Since $P$ has finitely many states, $k$ exists.  Let $\Gamma \in \xns$ a word of length at least $k$ and let $p \in Q_{P}$ be such that $\pi_{P}(\Gamma, p) = p$. Set $\Delta := \lambda_{P}(\Gamma, p)$, noting that $|\Gamma| = |\Delta|$ and set  $\Theta:= \lambda_{P}(\Gamma, q)$. Observe that there is a word $\overline{\Delta} \in X_{n}^{*}$ such that  $\Delta = \Lambda(\varepsilon,p)\overline{\Delta}$.
	
	By equation  \eqref{resulthasnoincompleteresponse} of Proposition~\ref{prop:algorithmforremovingincompleteresponse} we have
	
	\begin{equation}
	  \lambda'_{P}(\Gamma\Gamma,q_{-1}) = \Theta \Delta \Lambda(\varepsilon, p). \label{prop:pntoln 1}
	\end{equation}
	
	However, by equation \eqref{transition same} of Proposition~\ref{prop:algorithmforremovingincompleteresponse} we also have:
	\begin{equation}
	\lambda'_{P}(\Gamma\Gamma,q_{-1}) =   \lambda'_{P}(\Gamma, q_{-1})\lambda'_{P}(\Gamma, p) \label{prop:pntoln 2}
	\end{equation}
	
	Now, as $\lambda'_{P}(\Gamma, q_{-1}) = \Theta\Lambda(\varepsilon,p)$, equations~\eqref{prop:pntoln 1} and \eqref{prop:pntoln 2} imply that $\lambda'_{P}(\Gamma, p) = \overline{\Delta}\Lambda(\varepsilon,p)$. Therefore, in $P'_{q_{-1}}$ we have, $|\lambda_{P'}(\Gamma, p)| = |\overline{\Delta}\Lambda(\varepsilon,p)| = |\Delta| = |\Gamma|$, and $\pi_{P'}(\Gamma, p) = p$. This means that $\core(P_{q_{-1}})$ satisfies condition \ref{Lipshitzconstraint} above and so $[P]$ does as well.  
\end{proof}

Observe that for $\Shift{n} \in \pn{n}$, the transducer $[\Shift{n}]$ is precisely the identity element of $\Ln{n}$. For an element $P \in \pn{n}$ there is a negative integer $i \in \mathbb{Z}$ and  $R \in \pn{n}$ such that $ \shift{n}^{i}f_{P}f_{R}$ is the identity element in $\aut{\xnz, \shift{n}}$. Thus, by Proposition~\ref{prop:homomorphismfromPntoFinfinitiy},  $P \spnprod R =  \Shift{n}^{-i}$. Therefore, in order to show that $[P]$ has in inverse in $\SLn{n}$ (and so is an element of $\Ln{n}$) it suffices to demonstrate the following result.

\begin{lemma}\label{lem:removingincompleteresposeprod}
	Let $P, R \in \spn{n}$. Then $ [P][R] = [P \spnprod{n} R]$
\end{lemma}
\begin{proof}
	Note that for any $x \in \xn$ and any $P' \in \spn{n}$ there is a $y \in \xn$ such that $P \spnprod{n} Z_{x} = Z_{x}$ and $Z_{x} \spnprod{n} P = Z_{y}$. Therefore we may assume that neither  $P$ nor $R$ has infinite extent of incomplete response. 
	
	Let $p$ and $r$ be states of $P$ and $R$ respectively. We may choose $p$ and $r$ so that $(p,r)$ is a state of $\core(P \ast R)$. Form initial transducers $P'_{p_{-1}}$ and $R'_{r_{-1}}$ by applying the algorithm in  Proposition \ref{prop:algorithmforremovingincompleteresponse}. Now observe that since  $P'_{p_{-1}}$ and  $R'_{r_{-1}}$ are $\omega$-equivalent to $P_{p}$ and $P_{r}$ respectively, then the initial transducer $(P'\ast R')_{(p_{-1},r_{-1})}$ is $\omega$-equivalent to the initial transducer $(P\ast R)_{(p,r)}$. If the extent of incomplete response of $(p,r)$ is infinite, then both $(P\ast R)_{(p,r)}$ and $(P'\ast R')_{(p_{-1},r_{-1})}$ are have minimal representative $Z_{x}$ for some $x \in \xn$. Since $\core(P'_{p_{-1}})  = [P]$ and $\core(R'_{r_{-1}})  = [R]$, it follows that $[P][R] = [P \spnprod{n} R] = Z_{x}$. Thus, we may assume that the extent of incomplete response of $(p,r)$ is finite. Since $(p,r) \in \core(P\ast R)$ this means that the extent of incomplete response  of any state of $\core(P\ast R)$ is finite. Also, as $(P'\ast R')_{(p_{-1},r_{-1})}$ is $\omega$-equivalent to  $(P\ast R)_{(p,r)}$ it follows that any state of $(P' \ast R)_{(p_{-1},r_{-1})}$ has finite extent of incomplete response.
	
	 Let $(P'R')_{s}$ be the unique minimal transducer representing $(P'\ast R')_{(p_{-1},r_{-1})}$ and $(PR)_{t}$ be the unique  minimal transducer representing $(P \ast R)_{(p,r)}$, noting that $(P'R')_{s}$ and $(PR)_{t}$ are $\omega$-equal since they are minimal and $\omega$-equivalent. 
	
	 Observe that $\core((PR)_{t})$ is  $\omega$-equal to $[P\spnprod{n} R]$. For since $p$ and $r$ were chosen so that $(p,r) \in \core(P \ast R)$, $(PR)_{t}$ is $\omega$-equal to the minimal representative of an initial transducer $(P\spnprod{n}R)_{t'}$ for some state $t'$ of $P\spnprod{n}R$.  Furthermore, as $(PR)_{t}$ is $\omega$-equal to $(P'R')_{s}$, we have that $\core((P'R')_{s})$ is $\omega$-equivalent to $\core((PR)_{t})$. 
	 
	 Now, as $\core(P'_{p-1})$ is $\omega$-equivalent to $[P]$ and $\core(R'_{r-1})$ is $\omega$-equivalent to $[R]$ by definition, we have  $\core((P'\ast R')_{(p_{-},r_{-1})})$ is $\omega$-equivalent to $\core([P] \ast [R])$.  Observe that if $T$ and $U$ are two strongly synchronizing transducers with $\omega$-equivalent cores, then after applying  the algorithm 
	in Proposition~\ref{prop:algorithmforremovingincompleteresponse} to $T_{q}$ and $U_{q'}$ for any states $q \in Q_{T}$ and $q' \in Q_{U}$ the resulting transducers have $\omega$-equivalent cores. This is because the effect of the algorithm in Proposition~\ref{prop:algorithmforremovingincompleteresponse} on the states in the core depends only on the set of states in the core. Therefore we conclude that $\core((P'R')_{t})$ is $\omega$-equal to $[P][R]$. This yields the result.   
\end{proof}

\begin{cor}\label{cor:removeincompleteresponsepntoln}
	Let $P \in \pn{n}$, then $[P] \in \Ln{n}$.
\end{cor}

We shall later see that an element $P \in \spn{n}$ yields an element $[P] \in \Ln{n}$ if and only if $P \in \pn{n}$.

\subsection{An action of $\SLn{n}$ on $X_n^{\Z}$}

Let $\iota: \aut{\xnz, \shift{n}} \to \Ln{n}$ be defined by $\shift{n}^{i}f \mapsto [P]$ where $f \in F_{\infty}(\xn)$ and $P \in \pn{n}$ such that $f_{P} = f$. In what follows we shall show that $\iota$  is an epimorphism. We do this by establishing an action on $\xnz$ of any strongly synchronizing and core transducer which satisfies the condition \ref{Lipshitzconstraint}. In order to avoid repeating these conditions, denote by  $\Nn{n}$ the set of all strongly synchronizing and core transducers which satisfy the condition \ref{Lipshitzconstraint}. Note that not all elements of $\Nn{n}$ are minimal or even weakly minimal, however all  elements have minimal representatives in $\SLn{n}$.

\begin{predef}[Annotations]\label{def:annotatingLn}
	Let $L \in \Nn{n}$. An \emph{annotation} of $L$ is a map $\alpha: Q_{L} \to \Z$ which obeys the following rule: for any state $q \in Q_{L}$ and any word $\Gamma \in \xns$
	
	\begin{equation}
	  (\pi_{L}(\Gamma, q))\alpha = (q)\alpha + |\lambda_{L}(\Gamma, q)| - |\Gamma|. \label{eqn:annotationrule}
	\end{equation}
\end{predef}

It is easy to construct annotations for an element $L \in \Nn{n}$. We construct examples below which are canonical in the sense that all annotations of a given element of $\Nn{n}$ arise as one of these examples. Note that by \eqref{eqn:annotationrule}, an annotation is uniquely determined by the value it assigns a given fixed state of an element of $\Nn{n}$.

Suppose we have  an ordering $q_1, \ldots, q_{|L|}$  of the states of $L$. For each $1 \le i \le |L|$ fix a word  $\Gamma_{q_i} \in X_n^{\ast}$ such that $\pi_{L}(\Gamma_{q_i},q_{1}) = q_{i}$. Fix a number $j \in \mathbb{Z}$ and  define a map $\alpha_{q_1,j}: Q_{L} \to \Z$ according to the following rules:
	\begin{enumerate}[label = (\roman{*})]
		\item $(q_1)\alpha_{q_1,j} = j$, 
		\item  for $2 \le i \le |L|$, set $(q_i)\alpha_{q_1,j} = j + |\lambda_{L}(\Gamma_{q_j},q_1)| - |\Gamma_{q_j}|$ for $1 \le j \le |L|$.
	\end{enumerate}
	
We have the following lemma.

\begin{lemma}\label{lem:annotationsdonotdependonchoiceofsynchwords}
	Let $L \in \Nn{n}$. Fix an ordering $q_1, \ldots, q_{|L|}$ of the states of $L$, and fix a number $i \in \Z$. Let $\alpha_{q_{1}, i}$ be an annotation of $L$. Then $\alpha_{q_{1}, i}$ is independent of the choice of $\Gamma_{q_j} \in X_n^{k}$ for $1 \le j \le |L|$.
\end{lemma}
\begin{proof}
	  Let $l = |L|$, $\Gamma:= \{ \Gamma_{q_i} \mid 1 \le i \le l\}$ and suppose $\Delta:=\{\Delta_{q_j}| \ 1 \le j \le |L| \} \subseteq \xns$ is distinct from the set  $\Gamma$ and satisfies  $\pi_{A}(\Delta_{q_j},q_1) = q_{j}$ for all $1 \le j \le l$. Let $\alpha'_{q_{1}, i}: Q_{L} \to \Z$ be the map  arising from $\Delta$ as above.
	
	If $\alpha'_{q_{1}, i}$ is not equal to $\alpha_{q_{1}, i}$, then there is a state $q_j$, $1\le j \le l$ such that $r_1:=(q_j)\alpha'_{q_{1}, i} \ne (q_j)\alpha_{q_{1}, i}=:r_2$. By Definition \ref{def:annotatingLn} we have, $r_2=|\lambda_{L}(\Gamma_{q_j}, q_1)| - |\Gamma_{q_j}| \ne |\lambda_{L}(\Delta_{q_j}, q_1)|- |\Delta_{q_j}| = r_1$. Now constraint \ref{Lipshitzconstraint} implies that:
	\begin{equation}\label{usinglipschitzccond1}
	|\lambda_{L}(\Gamma_{q_1}, q_j)|+|\lambda_{L}(\Gamma_{q_j}, q_1)| = |\Gamma_{q_1}|+|\Gamma_{q_j}|.
	\end{equation}
	On the other hand we must also have:
	\begin{equation}\label{usinglipshitzcond2}
	|\lambda_{L}(\Gamma_{q_1}, q_j)|+|\lambda_{L}(\Delta_{q_j}, q_1)| = |\Gamma_{q_1}|+|\Delta_{q_j}|.
	\end{equation} 
	
	Now notice that $|\lambda(\Delta_{q_j}, q_1)| = |\Delta_{q_j}| + r_1$. Therefore $|\lambda_{L}(\Gamma_{q_1}, q_j)| = |\Gamma_{q_1}| - r_1$ from \eqref{usinglipshitzcond2}. Substituting this into equation \eqref{usinglipschitzccond1} we have:
	\[
	|\lambda_{L}(\Gamma_{q_j, q_1})| = r_1 - |\Gamma_{q_1}| + |\Gamma_{q_1}| + |\Gamma_{q_j}|
	\]
	so we conclude that:
	\[
	|\lambda_{l}(\Gamma_{q_j, q_1})| - |\Gamma_{q_j}| = r_1
	\]
which is a contradiction.  
\end{proof}

\begin{cor}
	Let $L \in \Nn{n}$ and $q_1, \ldots, q_{|L|}$ be an ordering of the states  of $L$. Then for any $i \in \Z$, $\alpha_{q_{1}, i}$ is an annotation of $L$.
\end{cor}
\begin{proof}
	This is essentially a restatement of Lemma~\ref{lem:annotationsdonotdependonchoiceofsynchwords}.  
\end{proof}

\begin{lemma}
	Let $L \in \Nn{n}$, then any annotation $\alpha$ of $L$ is equal to $\alpha_{q_1, i}$ for some ordering $q_1, \ldots, q_{|L|}$ of the states of $L$ and $i \in \Z$.
\end{lemma}
\begin{proof}
	Fix a state $q \in Q_{L}$. For each state $p \in Q_{L}\backslash \{q\}$ let $\Gamma_p$ be the minimal word in the short-lex ordering for which $\pi_{L}(\Gamma_p, q) = p$. Set $\Gamma_{q} = \varepsilon$. Order the states of $Q_{L}$ so that $p_1 \le p_2$ if $\Gamma_{p_1}$ is less than $\Gamma_{p_2}$ in the short-lex ordering. Let $q:=q_1, q_2, \ldots, q_{L}$ be the states of $\Gamma$ so ordered. Then the equality $\alpha = \alpha_{q_1, (q_1)\alpha}$  is immediate. 
\end{proof}
\begin{prerk}\label{rem:annotatingstepbystep}
	Let $L \in \Nn{n}$. Fix an ordering $q_1, \ldots, q_{|L|}$ of the states of $L$, and fix a number $i \in \Z$. Let $\alpha_{q_{1}, i}$ be an annotation of $L$. Let $q_j$ and $q_{k}$ be states of $L$ such that there is an $x \in \xns$ and $\pi(x, q_j) = q_k$, then by Lemma~\ref{lem:annotationsdonotdependonchoiceofsynchwords}  we have:
	\[
	(q_k)\alpha_{q_{1}, i} = (q_j)\alpha_{q_{1}, i} + |\lambda_{L}(x,q_j)| -|x|.
	\]
\end{prerk}
\begin{lemma}\label{lem:orderingdoesnotmatter}
	Let $L \in \Nn{n}$, fix a labelling $q_1, q_2 \ldots q_{|L|}$ of the states of $Q_L$ and let $i \in \mathbb{Z}$. Let $\alpha_{q_1, i}$ be an annotation of $L$ and $2 \le j \le |L|$ be arbitrary. Then $\alpha_{q_1,i} = \alpha_{q_j,(q_j)\alpha_{q_1,i}}$.
\end{lemma}
\begin{proof}
	Let $l = |L|$ and fix $1\le j \le l$. Let $\Delta_{q_k} := \Gamma_{q_1}\Gamma_{q_k}$ for $1 \le k \le l$, where the state of $L$ forced by $\Gamma_{q_k}$ is $q_k$. Define the annotation $\alpha_{q_j,\alpha_{q_1,i}(q_j)}$ with the set $\{\Delta_{q_k} | 1 \le k \le l\}$. The result is now a consequence of Remark~\ref{rem:annotatingstepbystep}, since  by  condition \ref{Lipshitzconstraint} we have $$(q_j)\alpha_{q_1,i} + |\lambda_{L}(\Gamma_1\Gamma_1, q_{j})| - |\Gamma_1\Gamma_1| = (q_j)\alpha_{q_1,i} + |\lambda_{L}(\Gamma_1 , q_{j})| - |\Gamma_1| = i.$$  
\end{proof}
Remark~\ref{rem:annotatingstepbystep} prompts the following \textit{ad hoc} method for creating an annotation $\alpha_{q_1,i}$ for a given $L \in \Nn{n}$, $i \in \mathbb{Z}$ and a fixed ordering of the states of $L$. 
\begin{enumerate}[label=(\alph*)]
	\item Set $ (q_1)\alpha_{q_1, i} = i$ and let $\mathscr{A} = \{q_1\}$.
	\item For all states $q_j$ such that there is an $x \in X_n$ for which $\pi_{L}(x, q_1) = q_j$ set $(q_j)\alpha_{q_1,i}= i + |\lambda_{L}(x,q_1)| -1$. Add these states to $\mathscr{A}$.
	\item Now consider the set of states $q_k$ in $Q_L \backslash \mathscr{A}$ such  that there is an $x \in X_n$ and $q_j \in \mathscr{A}$ for which $\pi_{L}(x, q_j) = q_k$, set  $(q_k)\alpha_{q_{1},i} = (q_j)\alpha_{q_{1},i} + |\lambda_{L}(x, q_j)| -1$. Add these new set of states to $\mathscr{A}$
	\item Repeat the previous step until $\alpha_{q_{1},i}$ is defined for all the states of $L$.
\end{enumerate}

By Lemma~\ref{lem:annotationsdonotdependonchoiceofsynchwords} and Remark~\ref{rem:annotatingstepbystep} this process will yield a well-defined function $\alpha_{q_1,i}: Q_{L} \to \mathbb{Z}$.

Figure~\ref{fig:annotationexample} depicts an element of $\pn{2}$; the numbers in red  indicate the image of  a state under the canonical annotation.

\begin{figure}[t]
	\begin{center} 
		\begin{tikzpicture}[shorten >=0.5pt,node distance=3cm,on grid,auto] 
		\tikzstyle{every label}=[red]
		\node[state] (q_0) [xshift=-2.5cm, yshift=0cm, label=left:$0$] {$a_0$}; 
		\node[state] (q_1) [xshift=0cm, yshift=0cm, label=right:$0$] {$a_1$}; 
		\node[state] (q_2) [xshift=2.5cm,yshift = 0cm, label=right:$0$]{$a_2$}; 
		\node[state] (q_3) [xshift=0cm, yshift=-2.5cm,label=right:$-1$] {$a_3$};
		\node[state] (q_4) [xshift=2.5cm, yshift=-2.5cm,label=right:$-1$] {$a_4$};
    \path[use as bounding box] (-3.6,-3.6) rectangle (3.6,1.75);
    		\path[->] 
		(q_0) edge node {$1|1$} (q_1)
		edge [loop above] node[swap] {$0|0$} ()     
		(q_1) edge[in = 95, out =265 ] node[swap] {$0|\varepsilon$} (q_3)
		edge[in=135, out=320] node{$1|\varepsilon$} (q_4)
		(q_2) edge[in=45, out= 135] node[swap]{$0|0$} (q_0)
		edge[loop above] node{$1|1$} ()   
		(q_3) edge[in=300, out = 145] node {$0|10$} (q_0)
		edge[in=275, out = 85] node[swap] {$1|01$} (q_1)
		(q_4) edge node[swap] {$1|11$} (q_2)
		edge[in =240, out = 230] node {$0|00$} (q_0);
		\end{tikzpicture}
		\caption{An annotated element of $\Ln{2}$.}
		\label{fig:annotationexample}
	\end{center}
\end{figure}

\begin{predef}[Equivalence of annotations]
	Let $L \in \Nn{n}$, $l = |L|$,  $\{ q_1, \ldots, q_l\}$ and  $\{p_1, \ldots, p_l\}$ be two orderings of the state set $Q_{L}$ and $\alpha_{q_1, i}$ and $\alpha_{p_1,j}$ be two annotations of $L$. Then $\alpha_{q_1, i}$ and $\alpha_{p_1,j}$ are equivalent if and only if there is a $k \in \mathbb{Z}$ such that $\alpha_{q_1,k+i} = \alpha_{q_1,(q_1)\alpha(p_1,j)} = \alpha_{p_1,j}$.
\end{predef}

\begin{prerk}\label{rem:equivofannotation}
	For a given $L \in \Nn{n}$ all annotations of $L$ are equivalent by Lemma~\ref{lem:orderingdoesnotmatter}.
\end{prerk}

In the definition below we pick a unique representative for the equivalence class of annotations of an element of $\Nn{n}$.

\begin{predef}[Canonical annotation]\label{def:canonicalannotation}
	Let $L \in \Nn{n}$ and let $k \in \mathbb{N}$ be the synchronizing level of $L$. For each state $q$ of $L$ let $\Gamma_q \in X_n^{k}$ be the smallest word, according to the lexicographic ordering, of length $k$ such that the state of $L$ forced by $\Gamma_q$ is $q$. Order the set $\{\Gamma_q | q \in Q_L \}$ according to the lexicographic ordering. Now label the states of $L$, $q_1, q_2, \ldots q_{|L|}$ such that $\Gamma_{q_1} < \Gamma_{q_2} < \ldots  \Gamma_{q_{|L|}}$. The annotation $\alpha_{q_1,0}$ is called the \emph{canonical annotation} of $L$. All other annotations of $L$ are equal to $\alpha_{q_1,i}$ for some $i \in \mathbb{Z}$.
\end{predef}

We are now in a position to define an action of $\Nn{n}$ on $X_n^{\mathbb{Z}}$.

\begin{predef}\label{def:Lwithannotationisanendo}
	Let $L \in \Nn{n}$, let $\alpha_{q_1,0}$ be the canonical annotation of $L$ and let $k \in \mathbb{N}$ be the synchronizing level of $L$. Let $(L, \alpha_{q_1,0}) : X_n^{\Z} \to X_n^{\Z}$ be defined as follows: for $x \in X_n^{\Z}$ and $i \in \mathbb{Z}$ let $q_j$, $1\le j \le |L|$, be the state of $L$ forced by ${x_{i-k}\ldots x_{i-1}}$. Let $r = (q_j)\alpha_{q_1,0}$ and let $w = \lambda_{L}(x_i,q_j)$. Let $y \in X_n^{\Z}$ be defined by ${y_{i+r}y_{i+r+1}\ldots y_{i+r+|w|}} := w$. Set $(x)(L, \alpha_{q_{1}, i}) := y$.    
\end{predef}

\begin{prop} \label{prop:annotationsofLareinEnd(Xn,sigma)}
	Let $L \in \Nn{n}$, and let $\alpha_{q_1,0}$ be the canonical annotation of $L$. Then $(L,\alpha_{q_1,0})$ is an element of $\End(X_n^{\Z}, \shift{n})$.
\end{prop}
\begin{proof}
	We will prove only that $(L,\alpha_{q_1,0})$ is well-defined. Continuity follows straightforwardly  from Definition~\ref{def:Lwithannotationisanendo}  and the proof that $(L,\alpha_{q_1,0})$ commutes with the shift map is almost identical to the proof of  Proposition~\ref{prop:pntildeisinendo}.
	
	To show well-definedness we need to argue that for each $i \in \Z$ and $x \in X_n^{\Z}$, there is at most one value $u \in X_n$ such that $y:=(x)(L,\alpha_{q_1,0})$ has $y_i = u$.

	Let $k \in \mathbb{N}$ be the synchronizing level of $L$. Let $i \in \mathbb{Z}$ be arbitrary and consider the block $x_{i-k}\ldots x_{i-1}x_{i}$ of $x$. Let $q_j$ be the state of $L$ forced by ${x_{i-k}\ldots x_{i-1}}$, and let $r_1 = (q_j)\alpha_{q_1,0}$. Furthermore let $w = \lambda_{L}(x_i, q_j)$, and let $r_1' = |w| - 1$. Let $q_m = \pi(x_i,q_j)$, notice that the state of $L$ forced by ${x_{i-k+1}\ldots  x_{i-1} x_{i}}$ is $q_m$.
	
	By definition, $y_{i+r_1}y_{i+r_1+1}\ldots y_{i+r_1+ |w| -1}= w$. However $(q_m)\alpha_{q_1,0} = r_1 + r_1'$. Therefore ${\lambda_{L}(x_{i+1}, q_m)}$ begins at the index: 
	\[
	i+1 + (r_1 + r_1' ) = i+1 + r_1 + |w| -1  = i + r_1 + |w|
	\]
which is the index to the left of index $i+r_1+|w|-1$. Therefore for any block $x_{i-k} \ldots x_{i-1} x_{i}$ the outputs corresponding to $x_{i-k+1} \ldots x_{i} x_{i+1}$ and $x_{i-k}\ldots x_{i-1}x_{i}$ are beside each other and do not overlap. Thus we conclude that for every index $i \in \mathbb{Z}$, $y_i$ is unique.   
\end{proof}

\begin{cor}\label{cor:annotationsinthesamecoset}
	Let $L \in \Nn{n}$, $i \in \Z$ and $\alpha_{q_1, i}$ be an annotation of $L$. Then $(L, \alpha_{q_1,i}) = \shift{n}^{i}(L, \alpha_{q_1,0})$.
\end{cor}
\begin{proof}
	Observe that by Definition \ref{def:annotatingLn}, $\alpha_{q_1,i} = i+ \alpha_{q_1,0}$. The result now follows from Definition~\ref{def:Lwithannotationisanendo}. 
\end{proof}

It is the case that there exist distinct elements  $L, M \in \Nn{n}$ with respective annotations $\alpha_{q_1, i}$ and $\beta_{p_1, j}$ such that the maps $(L, \alpha_{q_{1}, i})$ and $(M, \beta_{p_1, j})$ are equal. The following result shows that the same phenomenon does not occur in $\SLn{n}$.

\begin{prop}\label{prop:uniquenessofmaps}
	Let $L, M \in \SLn{n}$ with respective annotations $\alpha_{q_{1}, i}$ and $\beta_{p_1, j}$. Then $(L, \alpha_{q_{1}, i}) = (M, \beta_{p_1, j})$ if and only if $L =M$ and $\alpha_{q_{1},i} = \beta_{p_1, j}$ or  $L$ and $M$ both equal $Z_{a}$ for some $a \in \xn$.
\end{prop}
\begin{proof}
	Note that by Corollary~\ref{cor:annotationsinthesamecoset} we may assume that $\alpha_{q_1,i}$ is in fact $\alpha_{q_1,0}$ the canonical annotation of $L$. 
	
	Since $L \in \SLn{n}$, either $L  =  Z_{a}$ for some $a \in  \xn$, or else $L$ has no states of incomplete response (elements of $\SLn{n}$ are minimal). If $L = Z_{a}$ and $M = Z_{b}$ for some $a,b \in \xn$, then, as  the image of $(L, \alpha)$, for any annotation of $L$, is the set containing  the point $\ldots aaa \ldots \in \xnz$, we deduce that $ a= b$. Therefore, we may assume that $L$ or $M$ has no states of incomplete response. 
	
	Without loss of generality suppose that $L$ has no states of incomplete response. Let $0,b \in \xn$  and $q \in Q_{L}$ be such that $\pi_{L}(0, q) = q$ and $\lambda_{L}(0, q) = b$. Since $\Lambda(\varepsilon,q) = \varepsilon$, there is a word $\Gamma \in \xnp$ such that $\lambda_{L}(\Gamma, q) = \bar{b}\Delta$ for some element $\bar{b}$ of $\xn$ not equal to  $b$. Let $q' = \pi_{L}(\Gamma,q)$. As $(L, \alpha_{q_{1}, i}) = (M, \beta_{p_1, j})$ for annotations $\alpha_{q_{1}, i}, \beta_{p_1, j}$ of $L$ and $M$, then the corresponding state $p \in Q_{M}$ for which $\pi_{M}(0,p) = p$ also satisfies $\lambda_{M}(a, p) = b$. Notice that in the ordering of states arising from the canonical annotation of $L$ we have $q = q_1$.
	
	Let $\rho \in \xn^{N}$ be any infinite word and $\delta = \lambda_{L}(\rho, q')$. Let $x \in \xnz$ be defined such that $x_{i} = 0$ for all $i \le 0$ and $x_{1}x_{2}\ldots = \Gamma \rho$.  Consider the output $z = (x)(L,\alpha_{q_1,0})$. By definition of $(L, \alpha_{q_1, 0})$ we have $z_{i} = b$ for $i \le 0$ and $z_{1}z_{2} \ldots  =  \bar{b}\Delta \delta$. Notice that there is no $i \in \Z$ such that $(z)\shift{n}^{i} = z$ since $\bar{b} \ne b$. Therefore, since $(x)(M, \beta_{(p_1, j)})= z$ it must be the case that $M$ also has no states of incomplete response. 
	
	From the equality $(x)(M, \beta_{(p_1, j)})= z$, and since $M$ has no states of incomplete response, we also deduce that $(p)\beta_{(p_1,j)} = 0$, as $\bar{b} \ne b$. Now since $(p)\beta_{p_1, j} = 0$, it must be the case that $\lambda_{M}(\Gamma\rho, p) = \bar{b}\Delta\delta$. Moreover as $\Lambda(\varepsilon,q') = \varepsilon$, and using again that $M$ has no states of incomplete response, it must be the case that $\lambda_{M}(\Gamma, p) = \bar{b}\Delta$. Since $\rho$ was chosen arbitrarily, we therefore have that for the states $q'$ and $p':= \pi_{M}(\Gamma, p)$, $L_{q'}$ is $\omega$-equivalent to $M_{p'}$. Minimality and the strong synchronizing property of $L$ and $M$ therefore guarantees that $L = M$. Furthermore, as $(p)\beta_{(p_1,j)} = 0$, Lemma~\ref{lem:orderingdoesnotmatter} implies that $\beta_{(p_1, j)}$ is in fact equal to the canonical annotation of  $M$. Thus, $L=M$ and $\alpha_{q_1, 0} = \beta_{(p_1, j)}$ as required.   
\end{proof} 

\paragraph{Convention} Let $a \in \xn$; since all annotations of $Z_{a}$ yield the same map, we write $(Z_{a}, \infty)$ for the map $(Z_{a}, \alpha)$ for any annotation $\alpha$ of $Z_{a}$ . It shall be convenient to regard the symbol $\infty$ as the `annotation' of $Z_{a}$ that assigns the value infinity to the state of $Z_{a}$. This makes certain results easier to state without having to make exceptions for the transducers $Z_{a}$. Formally, the pair $(Z_a, \infty)$ cannot be interpreted as in Definition~\ref{def:Lwithannotationisanendo}. We also adopt the convention that $\infty + \infty = \infty$ and  $\infty + i = i+ \infty = \infty$ for any $i \in \Z$.

\subsection{Compatibility of annotations with minimising and with taking products}
Given an element $P \in \spn{n}$ with $|P| >1$ we saw that removing incomplete response and identifying $\omega$-equivalent states results in an element $[P] \in \SLn{n}$. It turns out that the extent of incomplete response in the states of $P$ together with the canonical annotation of $P$ yields in a natural way an annotation of the resulting element $[P]$. We require some preliminary observations in order to establish this fact.

\begin{lemma}\label{lem:identifyinequivstatesmakesnodiff}
	Let $T$ be a transducer with no states of incomplete response and let $p,q$ be a pair of $\omega$-equivalent states of $T$. Then  for any word $\Gamma \in \xns$, $\lambda_{T}(\Gamma,p) = \lambda_{T}(\Gamma, q)$ and $\pi_{T}(\Gamma,p)$ is $\omega$-equivalent to $\pi_{T}(\Gamma, q)$.
\end{lemma}
\begin{proof}
	The lemma follows straightforwardly  from the definitions of $\omega$-equivalence and incomplete response. 
\end{proof}

\begin{lemma}\label{lem:lagrespectsequivalence}
	Let $T \in \Nn{n}$ be such that there is  a state  of $T$ with finite extent of incomplete response. (Note that all states of $T$ consequently have finite extent of incomplete response). Let $\alpha_{p_1,0}$ be the canonical annotation of $T$, $p$ be any state of $T$, and $T'_{p_{-1}}$ be the transducer obtained from  $T_{p}$ as in Proposition~\ref{prop:algorithmforremovingincompleteresponse}. Let $q_1, q_2$ be any distinct states of $T$ such that  $q_1$ and $q_2$ are $\omega$-equivalent states of $T'_{p-1}$. (Recall that by construction $Q_{T'} = \{p_{-1}\} \sqcup Q_{T}$). Then, $\Lambda(\varepsilon, q_1) + (q_1)\alpha_{p_1,0} = \Lambda(\varepsilon, q_2) + (q_2)\alpha_{p_1,0}$.
\end{lemma}
\begin{proof}
	Let $k \ge 1$ be a synchronizing level of $T$, set $q_0 :=  p_1$ and  let $\Gamma_0$, $\Gamma_{1}, \Gamma_{2}$ be words of length $k$ which respectively force the states $q_0$, $q_1$ and $q_2$ of $T$. For $(a,b) \in \{ (0,1), (0,2), (1,0), (2,0) \}$ let $\Delta_{a,b} = \lambda_{T}(\Gamma_{a}, q_b)$. We may chose $k$ long enough so that, for valid $a,b$, $|\Delta_{a,b}|$ is longer than the extent of incomplete response of all of the states $q_0, q_1$, and  $q_2$. Let  $\overline{\Delta}_{a,b} \in \xns$ satisfy $\Delta_{a,b} = \Lambda(\varepsilon, q_b) \overline{\Delta}_{a,b}$.
	
	Since $T \in \Nn{n}$, it follows that for $a \in \{1,2\}$, $$2k = |\Delta_{a,0}\Delta_{0,a}| = |\overline{\Delta}_{a,0}\Lambda(\varepsilon, q_a)\overline{\Delta}_{0,a}\Lambda(\varepsilon, q_0)|.$$ Moreover, as $q_1$ and $q_2$ are $\omega$-equivalent states of $P'$, as $\lambda_{T'}(\Gamma_0, q_a) = \overline{\Delta}_{0,a}\Lambda(\varepsilon, q_0)$ for $a \in \{1,2\}$, and since $T'$ has no states of incomplete response, we have $\overline{\Delta}_{0,1}\Lambda(\varepsilon, q_0) = \overline{\Delta}_{0,2}\Lambda(\varepsilon, q_0)$. Thus the equality
	$$ |\overline{\Delta}_{1,0}\Lambda(\varepsilon, q_1)\overline{\Delta}_{0,1}\Lambda(\varepsilon, q_0)| = |\overline{\Delta}_{2,0}\Lambda(\varepsilon, q_2)\overline{\Delta}_{0,2}\Lambda(\varepsilon, q_0)|$$
	means that $|\overline{\Delta}_{1,0}\Lambda(\varepsilon, q_1)| = |\overline{\Delta}_{2,0}\Lambda(\varepsilon, q_2)|$ holds.  We are now done, since, as $\lambda_{T}(\Gamma_{a}, q_0) =\Delta_{a,0} = \Lambda(\varepsilon, q_0)\overline{\Delta}_{a,0}$ for $a \in \{1,2\}$ and, by definition of the canonical annotation, $(q_a)\alpha_{p_1,0} = |\Delta_{a,0}| -k$,  we have 
	\begin{IEEEeqnarray*}{rCl}
	 |\Lambda(\varepsilon, q_0)\overline{\Delta}_{1,0}| + |\Lambda(\varepsilon, q_1)| - k &=& |\Lambda(\varepsilon, q_0)|+ |\overline{\Delta}_{1,0}| + |\Lambda(\varepsilon, q_1)| - k   \\ 
	  &=& |\Lambda(\varepsilon, q_0)|+ |\overline{\Delta}_{2,0}| + |\Lambda(\varepsilon, q_2)| - k   \\
	   &=&|\Lambda(\varepsilon, q_0)\overline{\Delta}_{2,0}| + |\Lambda(\varepsilon, q_2)| - k
	 \end{IEEEeqnarray*} 
 as required. 
\end{proof}

\begin{prerk}
	Let $P \in \Nn{n}$ be synchronizing at level $k \ge 1$ and let $p$ be any state of $P$. Suppose first of all that all states of $P$ have finite extent of incomplete response. Let $L_{t}$ be the  minimal transducer representing the transducer $P'_{p_{-1}}$ obtained from $P_{p}$ as in Proposition~\ref{prop:algorithmforremovingincompleteresponse}. By construction, $P'_{p_{-1}}$ and $L_{t}$ are also synchronizing at level $k$. Furthermore, for any state $p \in P$, $p$ also denotes a state of $\core(P'_{p_{-1}})$. By definition, $\core(L_{t}) = [P]$ and the  states of $L_{t}$ are in one-to-one  correspondence with  $\omega$-equivalence classes of states of $P_{p_{-1}}$. We say that \emph{a state $q$ of $P$ corresponds to a state $s$ of $[P]$} if $s$ corresponds to the $\omega$-equivalence class of $q$ in $\core(P'_{p_{-1}})$. We denote by $(q)\kappa$ the state $s$ of $[P]$ in correspondence with the equivalence class of a state  $q \in Q_{P'}$. In this way we define a map $\kappa: Q_{P} = Q_{\core(P'_{p_{-1}})} \to Q_{[P]}$.  If there is a state of $P$ whose extent of incomplete response is infinite, then $[P] = Z_{a}$ for some $a \in \xn$  and $\kappa$ is the map which sends every state of $P$ to the single state of $Z_{a}$.
	
	Let $q_1, q_2, \ldots, q_{|P|}$ be the ordering of the states of $P$ as in Definition~\ref{def:canonicalannotation}, then this ordering induces, naturally, an ordering of the states of $[P]$, whereby a state $s$ of $[P]$ precedes another state $t$ if there are states $q_{i}, q_{j} \in Q_{P}$ with $i < j$ such that $(q_i)\kappa = s$ and $(q_j)\kappa = t$.
\end{prerk}

\begin{cor}\label{cor:identifyinequivstatesmakenodiff}
	Let $T \in \Nn{n}$ be a transducer without states of incomplete response and $\alpha_{p_1,j}$ be any annotation of $T$. Let $[T]$ be the transducer obtained from $T$ by identifying equivalent states. For a state $q \in Q_{T}$ let $[q]$ denote its $\omega$-equivalence class in $Q_{T}$, and identify $[q]$ with the corresponding state of $Q_{[T]}$. The map $\alpha'_{[p_1],j}: Q_{[T]} \to \Z$ by $[q] \mapsto (q)\alpha_{p_1,j}$ is an annotation of $[T]$. Moreover the maps $(T, \alpha_{p_1,j}): \xnz \to \xnz$ and $([T], \alpha'_{[p_1],j}): \xnz \to \xnz$ are equal.
\end{cor}
\begin{proof}
	It suffices to prove the lemma for the canonical annotation $\alpha_{q_1,0}$ of $T$ since any other annotation is simply $j + \alpha_{q_1,0}$ for some $j \in \Z$. However, this is now a straightforward application of Lemmas~\ref{lem:lagrespectsequivalence} and \ref{lem:identifyinequivstatesmakesnodiff}, since for any state $q \in Q_{T}$ we have $\Lambda(\varepsilon, q) = 0$. 
	
	The ``moreover'' part of the corollary is a consequence of the following observation. For a state $[q] \in Q_{[T]}$, any state $q \in Q_{T}$ which belongs to $[q]$ satisfies $\lambda_{T}(\Gamma,q) = \lambda_{[T]}(\Gamma, [q])$ for any word $\Gamma$, and $([q])\alpha'_{[p_1],0} = (q)\alpha_{p_1,0}$. 
\end{proof}

\begin{prop}\label{prop:incompleteresponsecorrespondstolagofelementofln}
	Let $P$ be an element of $\Nn{n}$. Let $ \beta_{q_1, 0}$ be the canonical annotation of $P$ and let $[P]$ be the unique  minimal transducer representing $P$. Then the map  $\alpha_{(q_1)\kappa,i}: Q_{[P]} \to \Z$ given by  $(q_j)\kappa \mapsto \Lambda(\varepsilon, q_j) + (q_j)\beta_{q_1,0}$, for $q_j \in Q_{P}$, and $i = \Lambda(\varepsilon, q_1)$ is a well-defined annotation of $[P]$. Moreover, the maps $([P],\alpha_{(q_1)\kappa,i}): \xnz \to  \xnz$ and $(P,\beta_{q_1, 0}): \xnz \to \xnz$ are equal.
\end{prop}
\begin{proof}
	Note that if a  state of $P$ has infinite extent of incomplete response, then all states of $P$ have infinite extent of incomplete response. In this case  $[P] = Z_{a}$ for some $a \in \xn$ and $(P, \beta_{q_{1}, 0}) = (Z_a, \infty)$. Since all states of $P$ have infinite extent  of incomplete response  we see that $\alpha_{(q_1)\kappa,i}$ defines the `annotation' $\infty$ of $Z_{a}$. Thus we may assume that all states of $P$ have finite extent of incomplete response.
	
	First observe that the map $\alpha_{(q_1)\kappa,i}$ is well-defined by Lemma~\ref{lem:lagrespectsequivalence}. Therefore it suffices to show that it defines an annotation of $[P]$. We make a further reduction.
    
    Let $q$ be a state of $P$ and let  $P'_{q_{-1}}$ be the  transducer obtained after performing the algorithm of Proposition \ref{prop:algorithmforremovingincompleteresponse} on the initial transducer $P_{q}$ (hence $[P]$ is the transducer obtained from identifying the equivalent states of $\core(P'_{q_{-1}})$). By Corollary~\ref{cor:identifyinequivstatesmakenodiff}, it suffices to show that the map, from $Q_{\core(P')} \to \Z$ given by $q_{j} \mapsto \Lambda(\varepsilon, q_j) + (q_j)\beta_{q_1, 0}$ is the annotation  $\alpha_{q_{1}, \Lambda(\varepsilon, q_1)}$ of $\core(P')$.  Since $Q_{\core(P')} = Q_{P}$, we shall use a prime above the states of $\core(P')$ to distinguish them from states of $P$. In keeping with this we have $\alpha_{q'_{1},\Lambda(\varepsilon, q_1)} := \alpha_{q_{1}, r}$. 
	
	Let $r = \Lambda(\varepsilon,q_{1})$, and let $\Gamma_{1} \in X_n^{\ast}$ be a word such that $\pi_{P}(\Gamma_{1}, q) = q_1$. Let $s$ be any state of $P$, and let $\Gamma_{s} \in \xn^{\ast}$ be such that $\pi_{P}(\Gamma_{s}, q_{1}) =s$. Note that  $\pi_{P'}(\Gamma_{s}, q'_{1}) = s'$ and $\pi_{P'}(\Gamma_{1}, q_{-1}) = q'_{1}$ must also  hold in $P'$. Consider $\lambda_{P'}(\Gamma_{1}\Gamma_{s}, q_{-1})$. By Proposition \ref{prop:algorithmforremovingincompleteresponse} we have:
	\[
	\lambda_{P'}(\Gamma_{1}\Gamma_{s}, q_{-1}) = \lambda_{P'}(\Gamma_{1}, q_{-1})\lambda_{P'}(\Gamma_{s}, q'_{1}) = \Lambda(\Gamma_{1}\Gamma_{s}, q_{1}) = \lambda_{P}(\Gamma_{1}\Gamma_{s},q)\Lambda(\varepsilon,s).
	\]
	
	By the same result we know that $$\lambda_{P'}(\Gamma_{1}, q_{-1}) = \Lambda(\Gamma_1,q) = \lambda_{P}(\Gamma_1,q))\Lambda(\varepsilon,q_{1}).$$ So
	\[
	\lambda_{P'}(\Gamma_{s}, q'_{1}) = \lambda_{P}(\Gamma_{1}\Gamma_{s},q)\Lambda(\varepsilon,s) - \lambda_{P}(\Gamma_1,q)\Lambda(\varepsilon,q_{1}).
	\]
	
	Therefore we have:
	\begin{IEEEeqnarray*}{rCl}
		r &+& |\lambda_{P}(\Gamma_{1}\Gamma_{s},q)\Lambda(\varepsilon,s) - \lambda_{P}(\Gamma_1,q)\Lambda(\varepsilon,q_{1})| - |\Gamma_{s}| = \\|\Lambda(\varepsilon,q_{1})|& + & |\lambda_{P}(\Gamma_1,q)|  + |\lambda_{P}(\Gamma_s,q_{1})| + |\Lambda(\varepsilon,s)|   
		- |\lambda_{P}(\Gamma_1,q)|\\
		&-& |\Lambda(\varepsilon,q_{1})|   - |\Gamma_s|
		=  |\Lambda(\varepsilon,s)| + |\lambda_{P}(\Gamma_s,q_{1})|  - |\Gamma_s,| \\
		&=& |\Lambda(\varepsilon,s)| + (s)\beta_{q_{1}, 0} .
	\end{IEEEeqnarray*}
	
	Hence we have in $\core(P')$ that $r+|\lambda_{P'}(\Gamma_{s}, q'_{1})| - |\Gamma_{s}| = |\Lambda(\varepsilon,s)| + (s)\beta_{q_{1}, 0}$. Therefore $(s')\alpha_{q'_{1},r} = |\Lambda(\varepsilon,s)| + (s)\beta_{q_{1}, 0}$ as required.
	
	In order to see that $(P, \beta_{q_{1},0}) =([P], \alpha_{q_{1},r})$, by Corollary~\ref{cor:identifyinequivstatesmakenodiff}, it suffices to show that the maps $(P, \beta_{q_1, 0})$ and $(\core(P'), \alpha_{q_{1},r}))$ are equal.  
	
   Let $k$ be a synchronizing level of $P$ (and so of $P'$), let $x \in X_n^{\Z}$, and $i \in \Z$ be arbitrary.  Consider the right infinite string $x_{i}x_{i+1}\ldots $  corresponding to the points of $x$ to the right of $x_{i-1}$ . Let $u'$ be the state of $\core(P')$ forced by the block ${x_{i-k}\ldots x_{i-1}}$ of $x$. Notice that $u$ is the state of $P$ forced by ${x_{i-k}\ldots x_{i-1}}$. Let $t$ be the state of $P$ forced by ${x_{i-k-1} \ldots x_{i-2}}$ noting that $\pi_{P}(x_{i-1},t) = u$  and so $\pi_{P'}(x_{i-1},t') = u'$.  Let $j = (t)\beta_{q_1, 0}$. 
	
	By Proposition \ref{prop:algorithmforremovingincompleteresponse}, choosing the state $q$ above to be the state $t$ and setting $q_{-1}:= t_{-1}$,  we have,  $z:= y_1 y_2 \ldots :=\lambda_{P'}({x_{i-1}x_{i}x_{i+1}\ldots }, t_{-1}) = \lambda_{P}({x_{i-1}x_{i}x_{i+1}\ldots }, t)$. Since $j = (t)\beta_{q_{1}, 0}$, the string ${z}$ is written in the output $(x)(P, \beta_{q_{1}, 0})$ starting at index $i-1+j$. For convenience, and to match up with this shift, index the sequence $z$ as starting at $i-1+j$ so that $z =  z_{i-1+j}z_{i+j}\ldots = y_{1}y_{2}\ldots$. 
	
	Noting the $'$ above $P$, consider where  the output $\lambda_{P'}({x_{i}x_{i+1}\ldots }, u')$ is written in the sequence $z$. Since, $\lambda(x_{i-1}, t_{-1}) = \Lambda(x_{i-1}, t) = \lambda(x_{i-1},t)\Lambda(\varepsilon,u) $, it follows that $|\lambda(x_{i-1}, t_{-1})| - |x_{i-1}|  = |\Lambda(\varepsilon,u)| + |\lambda(x_{i-1}, t)| - |x_{i-1}|$. Therefore, $\lambda_{P'}(x_{i}x_{i+1}\ldots, u')$  is written in sequence $z$ beginning  at index $i+(|\Lambda(\varepsilon, u) + |\lambda(x_{i+1}, t)| - |x_{i+1}| + j)$. The equality $|\lambda(x_{i+1}, t)| - |x_{i+1}| + j = (u)\beta_{q,0}$, means that $|\Lambda(\varepsilon, u)|+ |\lambda(x_{i+1}, t)| - |x_{i+1}| + j  = (u')\alpha_{q'_{1},r}$. Therefore, $\lambda_{P'}({ x_{i}x_{i+1}\ldots}, u')$ is equal to ${ z_{i+(u')\alpha_{q'_{1},r}}z_{i+(u')\alpha_{q'_{1},r}+1}}\ldots$ and from index $i+(u')\alpha_{q'_{1},r}$ rightward $(x)(P, \beta_{q_{1}, 0})$ is precisely $$ z_{i+(u')\alpha_{q'_{1},r}}z_{i+(u')\alpha_{q'_{1},r}+1}\ldots .$$ However, from index $i+(u')\alpha_{q'_{1},r}$ rightward,  $(x)(\core(P'), \alpha_{q'_{1},r})$ is precisely $$ z_{i+(u')\alpha_{q'_{1},r}}z_{i+(u')\alpha_{q'_{1},r}+1} \ldots .$$ Thus, we see that from index $i+(u')\alpha_{q'_{1},r}$ rightward, $(x)(\core(P'), \alpha_{q'_{1},r})$  and $(x)(P, \beta_{q_{1},0})$ coincide.  Since $i$ was arbitrarily chosen and since  the set $\{ (u')\alpha_{q'_{1},r} \mid u \in Q_{\core(P')} \}$ is finite, we conclude that $$(x)(\core(P), \alpha_{q'_{1},r}) =(x)(P, \beta_{q_{1},0}).$$  
\end{proof}

\begin{cor}
	Let $P \in \Nn{n}$, $\beta_{q_1, i}$ be any annotation of $P$ and  $[P]$ be the unique  minimal transducer representing $P$. Then, there is an annotation $\alpha_{(q_1)\kappa,l}$ of $[P]$ such that for all $q_j \in Q_{P}$ we have  $((q_j)\kappa)\alpha_{(q_1)\kappa,l} = |\Lambda(\varepsilon, q_j)| + (q_j)\beta_{q_1,i}$ . Moreover $([P],\alpha_{(q_1)\kappa,i}) = (P,\beta_{q_1, 0})$.
\end{cor}
\begin{proof}
	If a state of $P$ has infinite extent of incomplete response the result holds (regarding $\infty$ as an annotation). Thus we may assume that all states of $P$ have finite extent of incomplete response.
	
	By Remark~\ref{rem:equivofannotation} we can write $\beta_{q_1, i} = j+\beta_{p_1,0}$ for some $j \in \mathbb{Z}$ and where $\beta_{p_1,0}$ is the canonical annotation of $P$. By Proposition~\ref{prop:incompleteresponsecorrespondstolagofelementofln}, $(P, \beta_{p_1,0}) = ([P], \alpha_{(p_1)\kappa,r})$  for $r = |\Lambda(\varepsilon, p_1)|$. 
	
	Now Corollary~\ref{cor:annotationsinthesamecoset} means that $(P, \beta_{q_1, i}) = \shift{n}^{j}(P,\beta_{p_1,0}),$ thus $$\shift{n}^{j}\beta_{p_1,0} = \shift{n}^{j} ([P], \alpha_{(p_1)\kappa,r}) = ([P],  \alpha_{(p_1)\kappa,j+r}).$$  Therefore, as $\alpha_{(p_1)\kappa,j+r}$ is an annotation of $[P]$ satisfying the conditions of the corollary, we are done. 
\end{proof}

\begin{cor} \label{cor:incompleteresponsecreateslaginpn}
	Let $P \in \spn{n}$ and  let $[P]$ be the unique minimal transducer representing $P$. Then there is an annotation $\alpha_{(q_1)\kappa,i}$ of $[P]$ such that for all $q_j \in Q_{P}$, we have  $((q_j)\kappa)\alpha_{(q_1)\kappa,i} = \Lambda(\varepsilon, q_j)$. Moreover $([P],\alpha_{(q_1)\kappa,i}) = f_{P}$.
\end{cor}
\begin{proof}
	This is a direct consequence of Proposition \ref{prop:incompleteresponsecorrespondstolagofelementofln} and the observation that the canonical annotation of an element of $\spn{n}$ assigns the number zero to every state since elements of $\spn{n}$ are synchronous. 
\end{proof}

Figure~\ref{fig:elementofP2correspondingtoL} depicts an element $P$ of $\pn{n}$; the number beside a state corresponds to the extent of incomplete response of that state. The minimal representative  $[P]$ of $P$ is the transducer depicted in Figure~\ref{fig:annotationexample}. The map $\kappa: Q_{P} \to Q_{[P]}$ is given by $d_0,d_5 \mapsto a_0$, $d_1 \mapsto a_1$, $d_3 \mapsto a_3$, $d_2 \mapsto a_4$, $d_4 \mapsto a_2$.  It is immediate that the map $\alpha: Q_{[P]} \to \Z$ that assigns a state $(q)\kappa$  of $[P]$, for $q \in Q_{P}$, the value $\Lambda(\varepsilon, q)$ is an annotation of $[P]$. In fact the annotation $\alpha$ is the result of adding $2$ to the canonical annotation of $[P]$.

\begin{figure}[t]
	\begin{center} 
		\begin{tikzpicture}[shorten >=0.5pt,node distance=3cm,on grid,auto] 
		\tikzstyle{every label}=[blue]
		\node[state] (q_0) [xshift=-2.5cm, yshift=0cm, label=left:$2$] {$d_0$}; 
		\node[state] (q_1) [xshift=2.5cm, yshift=0cm, label=right:$2$] {$d_1$}; 
		\node[state] (q_2) [xshift=0cm,yshift = -1.5cm, label=left:$1$]{$d_2$}; 
		\node[state] (q_3) [xshift=4.5cm, yshift=-2cm, label=left:$1$] {$d_3$};
		\node[state] (q_4) [xshift=0cm,yshift=-3cm, label=right:$2$] {$d_4$};
		\node[state] (q_5) [xshift=4.5cm,yshift=-4cm,label=left:$2$] {$d_5$};
    \path[use as bounding box] (-3.5,-5.5) rectangle (6,1.5);
		\path[->] 
		(q_0) edge node {$1|0$} (q_1)
		edge [loop above] node[swap] {$0|0$} ()     
		(q_1) edge node {$1|0$} (q_2)
		edge[in=130, out=325] node{$0|0$} (q_3)
		(q_2) edge node{$0|1$} (q_0)
		edge node{$1|1$} (q_4)   
		(q_3) edge node {$0|1$} (q_5)
		edge[in=310, out=145] node {$1|1$} (q_1)
		(q_4) edge[loop left] node[swap] {$1|1$} ()
		edge node {$0|1$} (q_5)
		(q_5) edge[in=255, out=235] node {$0|1$} (q_0)
		edge[in= 30, out=10] node[swap] {$1|1$} (q_1);
		\end{tikzpicture}
		\caption{An element $P$ of $\pn{n}$ annotated by extent of incomplete response.}
		\label{fig:elementofP2correspondingtoL}
	\end{center}
\end{figure}
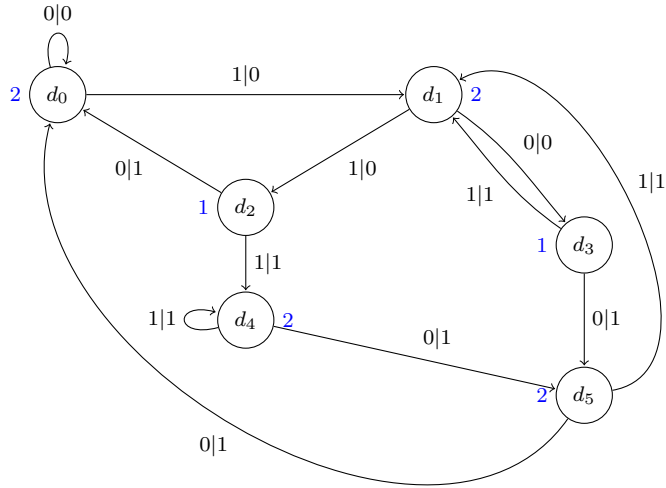

The results above show that an element of $\spn{n}$ corresponds in a natural way to a map $(L, \alpha_{q_1,i})$ for $L \in \SLn{n}$ and $\alpha_{q_1, i}$ an annotation of $L$. Moreover, these results demonstrate that the `damage' caused by minimising a transducer is `fixed' by accounting for the extent of incomplete response of each state. 

In what follows we construct an associative product  on the set of maps $(L, \alpha_{q_{1}, i})$ for $L \in \SLn{n}$ and an annotation $\alpha_{q_{1}, i}$ of $L$. Specifically, we show that given two elements $L, M \in \SLn{n}$ and annotations $\alpha_{q_1,i}$ and $\beta_{p_1, j}$ of $L$ and $M$ respectively, then the composition of the resulting maps $(L, \alpha_{q_1, i})$ and $(M, \beta_{p_1,j})$ is equivalent to a map obtained by annotating the  minimal transducer $LM$ representing the core of the product of $L$ and $M$. Moreover, the annotation of $LM$ is obtained from the annotations of $L$ and $M$  and the extent of incomplete response in the states of $\core(L \ast M)$.

We begin with the following lemma.

\begin{lemma}\label{lem:annotationofproductissumofannoation}
	Let $L$ and $M$ be elements of $\SLn{n}$ with respective annotations $\alpha_{q_1, i}$ and $\beta_{p_1,j}$. Let $(LM) = \core(L* M)$, noting that $Q_{(LM)} \subseteq Q_{L}  \times Q_{M}$. Define a map $\alpha_{q_1, i} + \beta_{p_1,j}: Q_{(LM)} \to \Z$ by $(s,t) \mapsto (s)\alpha_{q_1, i} + (t)\beta_{p_1, j}$. Then
	\begin{enumerate}[label = (\alph*)]
		\item $\alpha_{q_1, i} + \beta_{p_1,j}$ is an annotation of $(LM)$, and,
		\item $((LM), \alpha_{q_1, i} + \beta_{p_1,j}) = (L, \alpha_{q_1,i})(M,\alpha_{p_1,j})$.
	\end{enumerate}  
\end{lemma}
\begin{proof}
	 Let $(s_1, t_1)$ and $(s_2, t_2)$ be arbitrary states of $(LM)$, $a = (s_1)\alpha_{q_1, i}$ and $b = (t_1)\beta_{p_1, j}$. As $(LM) \in \SLn{n}$, let $\Gamma \in X_n^{+}$ be arbitrary such that $\pi_{(LM)}(\Gamma, (s_1,t_1)) = (s_2, t_2)$. Set $\Delta := \lambda_{L}(\Gamma, s_1)$. By definition of the product transducer we have,
	\begin{IEEEeqnarray*}{rCl}
		\pi_{L}(\Gamma, s_1) &=& s_2 \\
		\pi_{M}(\Delta, t_1)  &=& t_2.
	\end{IEEEeqnarray*}
	
	Set  $\Xi := \lambda_{M}(\Delta, t_1)$ and observe that:
	\[
	|\Xi| - |\Gamma| = (|\Xi| - |\Delta|) + (|\Delta| - |\Gamma|).
	\]
	
	 Therefore,
	\[
	a + b + |\Xi| - |\Gamma| = a + (|\Delta| - |\Gamma|) + b + (|\Xi| - |\Delta|)   =  (s_2)\alpha_{q_1, i} +  (t_2)\beta_{p_1, j}
	\]
	
	and so   $\alpha_{q_1, i} + \beta_{p_1,j}$  is an annotation of $(LM)$.
	
	We now show that $((LM), \alpha_{q_1, i} + \beta_{p_1,j}) = (L,\alpha_{q_1, i})(M,\beta_{p_1,j})$. 
	
	Let $x \in X_n^{\Z}$ and $l \in \Z$ be arbitrary; let $k_1, k_2 \in \N$ be  synchronizing levels of $L$ and $M$ respectively and let $m \in  \N$ be  minimal such that $|\lambda_{L}(\Gamma, q)| \ge k_2$ for all $\Gamma \in X_n^{m}$ and $q \in Q_{L}$. Also set $y := (x)(L,\alpha_{q_1, i})$, $z := (y)(M,\beta_{p_1,j})$ and $z' := (x)((LM), \alpha_{q_1, i} + \beta_{p_1,j})$. 
	
	Consider the word $x_{l-k_1-m}x_{l-k_1-m+1}\ldots x_l$. Let $s$ be the state of $L$ forced by ${x_{l-k_1}x_{l-k_1+1}\ldots x_{l-1}}$,  $s'$ be the state of $L$ forced by ${x_{l-k_1-m} x_{l-k_1-m+1}\ldots x_{l-m-1}}$ and $t$  the state of $M$ forced by  $\lambda_{L}({x_{l-m}x_{l-m+1}\ldots x_{l-1}}, s')$. Then $(s,t)$ is the state of $(LM)$ forced by $x_{l-k_1-m}x_{l-k_1-m+1}\ldots x_{l-1}$.
	
	By definition of the map $(L,\alpha_{q_{1}, i})$, the right infinite string ${\lambda_{L}( {x_{l}x_{l+1}\ldots}, s)}$ is written in the output $y$  at index $l+ (s)\alpha_{q_1, i}$ and so is equal  to $$ y_{l+ (s)\alpha_{q_1, i}}y_{l+(s)\alpha_{q_1, i}+1}\ldots.$$ Moreover, by the proof of Proposition~\ref{prop:annotationsofLareinEnd(Xn,sigma)}, the equality $${y_{l+(s)\alpha_{q_1, i} - m} y_{l+(s)\alpha_{q_1, i} - m +1} \ldots  y_{l+ (s)\alpha_{q_1, i}-1}} = \lambda_{L}({x_{l-m} x_{l-m+1} \ldots x_{l-1}}, s')$$ is valid. 
	
	Now, as the state of $M$ forced by $\lambda_{L}({x_{l-m} x_{l-m+1} \ldots x_{l-1}}, s')$ is $t$,  the right infinite string ${\lambda_{M}({y_{l+ (s)\alpha_{q_1, i} }y_{l- (s)\alpha_{q_1, i}+1}\ldots}, t )}$ is written in $z$ at index $l+ (s)\alpha_{q_1, i} + (t)\beta_{p_1, j}$ and is precisely $ z_{l+ (s)\alpha_{q_1, i}+ (t)\beta_{p_1, j}}z_{l+(s)\alpha_{q_1, i}+ (t)\beta_{p_1, j}+1}\ldots$. However, this string is also equal to   ${\lambda_{(LM)}({x_{l} x_{l+1}\ldots}, (s,t))}$. Since  $(s,t)(\alpha_{q_1, i}+\beta_{p_1, j}) = (s)\alpha_{q_1, i}+(t)\beta_{p_1, j}$, and, as noted above, $(s,t)$ is the state of $(LM)$ forced by ${x_{l-k_1 - m}x_{l-k_1 -m +1}\ldots x_{l-1}}$, then $z'$ coincides with $z$ from index $l+ (s)\alpha_{q_1, i}+ (t)\beta_{p_1, j}$ onwards. Since $l$ was chosen arbitrarily and the image of the map $\alpha_{q_1, i}+\beta_{p_1, j}$ is finite, $z' = z$.
\end{proof}

The following notation is useful for the next result.

\begin{ntn}
	Let $L$ and $M$ be elements of $\SLn{n}$ with respective annotations $\alpha_{q_1, i}$ and $\beta_{p_1,j}$, and let $(LM) = \core(L\ast M)$. If $(LM)$ has a state with infinite extent of incomplete response then $LM$ (the product in $\SLn{n}$ of $L$ and $M$) is  equal to $Z_{a}$ for some $a \in \xn$. Set $\overline{\alpha_{q_1, i} + \beta_{p_1, j}} = \infty$. Otherwise, all states of $(LM)$ have finite extent of incomplete response. In this case, let $(LM)'$ be the core of the transducer obtained by applying the algorithm in Proposition~\ref{prop:algorithmforremovingincompleteresponse} to the transducer $(LM)_{q}$ for any state $q \in  Q_{(LM)}$.  By definition, $LM$  is the transducer obtained by identifying equivalent states of $(LM)'$.  Recall that the map $\kappa: Q_{(LM)} \to Q_{LM}$ is defined such that for a state $(s,t)$ of $Q_{(LM)}$, $((s,t))\kappa$  corresponds to the $\omega$-equivalence class of the state $(s,t)$ in $(LM)'$. Define a map $\overline{\alpha_{q_1, i} + \beta_{p_1, j}}: Q_{LM} \to \Z$ by  $((s, t))\kappa \mapsto (s)  \alpha_{q_{1}, i} + (t)\beta_{p_1, j} + \Lambda(\varepsilon, (s, t))$ for all $(s,t) \in  Q_{(LM)}$.
\end{ntn}

\begin{prop}\label{prop:annotationsrepectsproductwithsomefudge}
	Let $L$ and $M$ be elements of $\SLn{n}$ with respective annotations $\alpha_{q_1, i}$ and $\beta_{p_1,j}$. Then $(L, \alpha_{q_1,i})(M,\alpha_{p_1,j}) = (LM, \overline{\alpha_{q_1, i} + \beta_{p_1, j}})$ where $LM$ is the product in $\SLn{n}$ of $L$ and $M$. 
\end{prop}

\begin{proof}
	 Let $(LM) = \core(L\ast M))$. By Lemma \ref{lem:annotationofproductissumofannoation} we have, $$(L,\alpha_{q_1, i})(M,\beta_{p_1,j}) = ((LM),\alpha_{q_1, i}+ \beta_{p_1,j}).$$ By Proposition \ref{prop:incompleteresponsecorrespondstolagofelementofln},
	\[
	((LM),\alpha_{q_1, i}+ \beta_{p_1,j}) = (LM, \gamma_{((q_1,p_1))\kappa,r})
	\]
	for $r =  \Lambda(\varepsilon, (q_1,p_1))$  the extent of incomplete response of the state $(q_1, p_1)$ of $(LM)$, and $\gamma_{((q_1,p_1))\kappa,r}$ an annotation of $LM$ . Moreover the annotation $\gamma_{((q_1,p_1))\kappa,r}$ of $LM$ satisfies:
	\[
	(((p,q))\kappa)\gamma_{((q_1,p_1))\kappa,r} = \Lambda(\varepsilon, (p,q)) + ((p,q))(\alpha_{q_1, i}+ \beta_{p_1,j})
	\]
	for $(p,q) \in Q_{(LM)}$. Thus the equality $\gamma_{((q_1,p_1))\kappa,r} = \overline{\alpha_{q_1, i}+ \beta_{p_1,j}}$ is valid. 
\end{proof}

\begin{predef}
	Let $ \ASLn{n}$ be the set of pairs $(L, \alpha)$ such that either $L$ is an element of $\SLn{n}$ with no states of incomplete response and $\alpha$ is an annotation of $L$,  or $L = Z_{a}$ for some $a \in \xn$ and $\alpha = \infty$. Let $\ALn{n}$ be the subset of $\ASLn{n}$ consisting of those elements $(L, \alpha)$ such that $L \in \Ln{n}$. 
\end{predef}

One should think of the symbols  $\ASLn{n}$ as meaning ``annotated $\SLn{n}$'' and, analogously, $\ALn{n}$ as meaning ``annotated $\Ln{n}$''.

\begin{prop}\label{Prop:Associativeproduct}
	The map from $\ASLn{n} \times \ASLn{n} \to  \ASLn{n}$ given by $$((L,\alpha),(M, \beta)) \mapsto (LM, \overline{\alpha+\beta})$$  defines an associative product on $\ASLn{n}$. The set $\ALn{n}$ is a subgroup of $\ASLn{n}$.
\end{prop}
\begin{proof}
	Let $L,M,P \in \SLn{n}$ with respective annotations $\alpha, \beta, \gamma$.  By Proposition~\ref{prop:annotationsrepectsproductwithsomefudge} and using the fact that composition of functions is associative, the following equalities hold,
	
	\begin{IEEEeqnarray*}{rCl}
	 ((LM)P,\overline{\overline{\alpha + \beta}+ \gamma})& =& ((L,\alpha)(M,\beta))(P,\gamma) = (L,\alpha)((M,\beta)(P,\gamma))\\
	  &=& (L(MP),\overline{\alpha+\overline{\beta+\gamma}}).
	\end{IEEEeqnarray*}
	
	Now Proposition~\ref{prop:uniquenessofmaps} implies that $L(MP) = (LM)P$ (although this fact already follow from associativity of the binary operation on $\SLn{n}$)  and $\overline{\overline{\alpha + \beta}+ \gamma} = \overline{\alpha+\overline{\beta+\gamma}}$. The pair $(\id, 0)$ (where $\id$ is the single state identity transducer and $0$ is the annotation that assigns the value $0$ to the state of $\id$) is the identity element of $\ASLn{n}$.
	
    Observe that since $\Ln{n}$ is a group, then the restriction of the multiplication on $\ASLn{n}$ to $\ALn{n}$ does in fact yield a  binary operation on $\ALn{n}$. Let $(L,\alpha) \in \ALn{n}$ observe that there is an element $M \in \Ln{n}$ such that $LM = \id$, the single state identity transducer, thus there is an annotation $\beta$ of $M$ such that $(L,\alpha)(M,\beta) = (M,\beta)(L,\alpha) = (\id, 0)$.  
\end{proof}

Thus $\ASLn{n}$ together with the binary operation defined above forms a monoid. We observe that the set of elements $\{(Z_{a}, \infty) \mid a \in \xn\}$ is a set of right zeroes of this monoid.

\subsection{$\End(\xnz, \shift{n}$) is isomorphic to $\ASLn{n}$}

Let $\varphi: \End(\xnz, \shift{n}) \to \ASLn{n}$ by $\shift{n}^{i}f \mapsto ([P], \alpha)$ where $i \in \Z$, $f \in F_{\infty}(\xn)$, $P \in \spn{n}$ is such that $f_{P} = f$, $[P]$ is the minimal representative of $P$, and ${\alpha}$ is the annotation defined by $(q)\kappa \mapsto i + \Lambda(\varepsilon, q)$. Note that the map $\alpha$ is in fact an annotation by Corollary~\ref{cor:incompleteresponsecreateslaginpn} and \ref{cor:annotationsinthesamecoset}.

\begin{theorem}\label{thm:EndisotoASLn}
	  The map $\varphi: \End(\xnz, \shift{n}) \to \ASLn{n}$ is an isomorphism of monoids. The restriction of $\varphi$ to the subgroup $\Aut(\xnz, \shift{n})$ yields an isomorphism to the group~$\ALn{n}$.
\end{theorem}
\begin{proof}
	This is a consequence of  Corollaries~\ref{cor:FinftyisoPn}, \ref{cor:incompleteresponsecreateslaginpn}, and \ref{cor:annotationsinthesamecoset}, and, Propositions~\ref{prop:uniquenessofmaps} and \ref{prop:annotationsrepectsproductwithsomefudge}. 
\end{proof}

\begin{theorem}
	The map from $\ASLn{n} \to \SLn{n}$ given by $(L, \alpha) \mapsto L$ is a monoid homomorphism with kernel the congruence generated by the pairs $\{ ((\id, i),(\id,0)) \mid i \in \Z\}$.
\end{theorem}
\begin{proof}
	Note that two element $(L,\alpha)$ and $(M, \beta)$ have the same image under this map precisely when $L=M$, and so $\beta = \alpha +i$ for some $i \in \Z$. By definition of the multiplication on $\ASLn{n}$ the map is a homomorphism.  
\end{proof}

\begin{prerk}
	Notice that the restriction to the second coordinate of the binary operation on $\ASLn{n}$ does not yield a homomorphism unto an abelian group. This is because although, by Lemma~\ref{lem:annotationofproductissumofannoation} annotations are additive under un-minimised products, the minimisation process introduces some obfuscation arising from states of incomplete response.
\end{prerk}

\begin{cor}\label{cor:mainresult}
	$\End(\xnz, \shift{n})/\gen{(\shift{n}^{i}, \shift{n}^0 \mid i \in \Z} \cong \SLn{n}$ and $\Aut(\xnz, \shift{n})/\gen{\shift{n}} \cong \Ln{n}$.
\end{cor}

\section{Algebra of central extensions: 2-cocycles, 2-coboundaries and splittings}\label{Sec:cocyclesAndCoboundaries}

In this Section, we explore the exact sequence 

\begin{equation} \label{Eqn:exactsequence}
 1 \to \gen{\shift{n}} \to \aut{\xnz,\shift{n}} \to \Ln{n} \to 1.    
\end{equation}

Specifically, we determine exactly when \ref{Eqn:exactsequence} splits.  We have already seen that we can use annotations to take products in the extension (even when it is not split).  It follows that we should be able to describe the 2-cocycles and 2-coboundaries for the extension using annotations.  We explain how to do this in the later part of the section.

Before that, we shall first give an argument using the dimension representation of $\aut{\xnz,\shift{n}}$ to characterise exactly when the sequence \ref{Eqn:exactsequence} splits.  

The  \emph{dimension representation} is introduced by Kreiger in \cite{KriegerDimension}. For the group $\aut{\xnz, \shift{n}}$ the dimension representation  is a homomorphism $\dimr: \aut{\xnz, \shift{n}} \to \mathcal{G}(n)$ where $\mathcal{G}(n)$ is the multiplicative subgroup of $\Q\backslash\{0\}$ generated by the prime divisors of $n$.  The article \cite{BoyleMarcuTrow} provides an in-depth discussion on the dimension representation.

\begin{theorem}{\label{Thm:splittingtheorem-Belk}}
 The short exact sequence $1 \to \gen{\shift{n}} \to \aut{\xnz,\shift{n}} \to \Ln{n} \to 1$ splits if and only if $n$ is not an integral power of a smaller integer. 
\end{theorem}
\begin{proof}
    If $n$ is not a power of a smaller integer, then $n$ has no proper roots in $\mathcal{G}(n)$. This means, as $\mathcal{G}(n)$ is a free abelian group, that $\mathcal{G}(n)$ is the internal direct product of $\gen{n}$ and some complementary subgroup $\gen{n}^{c}$. Using the well known fact that $\dimr(\shift{n})  =n$, we deduce that $\aut{\xnz, \shift{n}}$ is the internal direct product of $\shift{n}$ and $\dimr^{-1}(\gen{n}^{c})$. It follows that $\dimr^{-1}(\gen{n}^{c}) \cong \Ln{n}$.
    
    On the other hand, if $n$ is a power of some of a strictly smaller integer, then, it is a standard (folkloric) exercise to construct proper roots of the shift map $\shift{n}$ in $\aut{\xnz, \shift{n}}$. From this it is immediate that the short exact sequence $1 \to \gen{\shift{n}} \to \aut{\xnz,\shift{n}} \to \Ln{n} \to 1$ does not split.
\end{proof}

Now we characterise the 2-cocycles and 2-coboundaries, by using annotations.

First we fix a section of the group $\gen{\shift{n}}$ in $\aut{\xnz, \shift{z}}$.  Recall that a section for $\gen{\shift{n}}$ in $\aut{\xnz, \shift{z}}$ is a choice of unique coset representative for each coset of $\gen{\shift{n}}$ in $\aut{\xnz, \shift{z}}$.

Define a set $S$ as follows:
\[
S\seteq \left\{(T,\alpha_{T,0})\, |\, T\in\Ln{n}\right\},
    \]
where for $T \in \Ln{n}$, the notation $\alpha_{T,0}$ represents the unique annotation of the transducer $T$
that assigns $0$ to the unique state $q_0$ of $T$ such that $\pi_{T}(0, q_0) = q_0$.

By 
Theorem~\ref{thm:EndisotoASLn}  (stating that for any element $\psi \in \aut{\xnz, \shift{n}}$ there is a unique pair
$(U, \gamma)$, for $U \in \Ln{n}$ and $\gamma$ an annotation of $U$, inducing $\psi$) $S$ is a section of $\gen{\shift{n}}$ in
$\aut{\xnz, \shift{n}}$. We observe that in $S$ the representative for the
group $\gen{\shift{n}}$ is actually the pair $(\id,0)$.

Define a map $c: \Ln{n} \times \Ln{n} \to \Z$ as follows:

$$(T,U)c = {|\Lambda(\varepsilon,(q_0,p_0))|}$$ where $q_0 \in Q_{T}$ and $p_0 \in Q_{U}$ satisfy $\pi_{T}(0, q_0) = q_0$ and $\pi_{T}(0, p_0) = p_0$. We now show that the map $c$ represents the element of quotient group $H^2(\Ln{n}, \Z)$ corresponding to the central extension \eqref{Eqn:exactsequence}.

To see this we observe that for two elements $(T, \alpha_{T,0}), (U, \alpha_{U,0}) \in S$, we have the equality:

$$ (T, \alpha_{T,0})(U, \alpha_{U,0}) = (TU, \alpha_{TU,0}) \shift{n}^{|\Lambda(\varepsilon,(q_0,p_0))|}.$$
This follows from the fact that, by definition, $\overline{\alpha_{T,0} + \alpha_{U,0}}$ is the annotation of $TU$, which assigns the state $(q_0,p_0)$ the value $(q_0) \alpha_{T,0} + (p_0) \alpha_{U,0} + |\Lambda(\varepsilon,(q_0,p_0))|$ and Corollary~\ref{cor:annotationsinthesamecoset}. That $c$ satisfies the co-cycle identity is a consequence of associativity of multiplication in $\aut{\xnz, \shift{n}}$.

Suppose we had another section $S'$ of $\gen{\shift{n}}$ in $\aut{\xnz, \shift{n}}$. We derive the 2-coboundary relating the associated 2-cocyle $c'$ of the section $S'$ to the  2-cocyle $c$ derived above. 

Set $$S' := \{ (T,\beta_{T}) : T \in \Ln{n}, \beta_{T} \mbox{ an annotation of } T\}.$$ Then $c'$ is the map defined as follows. For $T,U \in \Ln{n}$, $$(T,U)c' = (q_0)\beta_{T} + (p_0) \beta_{U} + |\Lambda(\varepsilon, (q_0, p_0))| - (t_0) \beta_{TU}$$ where $q_0 \in Q_T$, $p_0 \in Q_{U}$ and $t_0 \in Q_{TU}$, $\pi_{T}(0, q_0) = 0$, $\pi_{U}(0, p_0) = 0$ and $\pi_{TU}(0, t_0) = 0$. This follows from the observation that $(T, \beta_{T})(U, \beta_{U}) = (TU, \overline{\alpha_{T}+ \beta_{U}})$, the fact that by definition: $ (t_0)\overline{\alpha_{T}+ \beta_{U}} = (q_0)\beta_{T} + (p_0) \beta_{U} + \Lambda(\varepsilon, (q_0, p_0)$ and Corollary~\ref{cor:annotationsinthesamecoset}.

Define a map $d: \Ln{n} \to \Z$ by $(T)d \mapsto (q_0) \beta_{T}$ where $q_0 \in Q_T$ satisfies $\pi_{T}(0, q_0) = q_0$. Then for $T,U \in \Ln{n}$,

\begin{IEEEeqnarray*}{rCl}
(T,U)c' - (T,U)c &=& (q_0)\beta_{T} + (p_0) \beta_{U} + |\Lambda(\varepsilon, (q_0, p_0))| - (t_0) \beta_{TU} - |\Lambda(\varepsilon, (q_0, p_0))| \nonumber \\ 
&=& (q_0)\beta_{T} + (p_0) \beta_{U} - (t_0) \beta_{TU}  \nonumber \\
&=& (T)d + (U)d - (TU)d.    
\end{IEEEeqnarray*}

Thus $d$ is the 2-coboundary relating the 2-cocycles $c$ and $c'$.

\section{Applications}\label{Section:applications}

In this section we develop some consequences of Corollary~\ref{cor:mainresult}. This will largely fall into two categories: in the first instance we use information about the groups $\On{n}$ and $\Ln{n}$ to deduce consequences for $\Aut(\xnz, \shift{n})$, in the second instance we go in the other direction. We begin with the former.

\subsection{Applications to $\aut{\xnz, \shift{n}}$}
The paper \cite{OlukoyaAutTnr} describes a map $\rsig: \On{n} \to \Un{n-1}$ (here $\Un{n-1}$ denotes the group of units of $\Z_{n-1}$). For $1 \le r \le n-1$ let $\Un{n-1}^{r}$ be the subgroup of those elements $j \in \Un{n-1}$ such that $jr \equiv r \pmod{n-1}$. We note that for all $1 \le r \le n-1$, $\On{n,r} \le \On{n}$ and  $\On{n,n-1} = \On{n}$. In \cite{OlukoyaAutTnr} the following result is proved:

\begin{theorem}\label{thm:rsig}
	The following statements are valid:
\begin{enumerate}[label= (\arabic*)]
	\item $\rsig$ is a homomorphism; \label{rsig:hom}
	\item an element $T \in \On{n}$ is an element of $\On{n,r}$ if and only if $(T)\rsig \in \Un{n-1}^{r}$; \label{rsig:stabliser}
	\item $(\Ln{n})\rsig$ contains $\Un{n-1,n}:=\gen{ d \in \Z_{n-1} \mid d \mbox{ is a divisor  of } n}$, the subgroup of $\Un{n-1}$ generated by the divisors of $n$. Note that this means $(\On{n-1})\rsig$ contains $\Un{n-1,n}$ as well.\label{rsig:divisorsofn}
\end{enumerate}

\end{theorem} 
 From points ~\ref{rsig:hom} and \ref{rsig:stabliser} of Theorem~\ref{thm:rsig}, one immediately deduces that groups $\On{n,r}$ for all $1 \le r < n-1$ are all normal subgroups of $\On{n,n-1}$. As $\Ln{n}$ is a subgroup of $\On{n}$, the previous sentence holds when $\On{n}$ is replaces with $\Ln{n}$. We deduce the immediate corollary for $\ALn{n} \cong \aut{\xnz, \shift{n}}$.

\begin{cor}
	For $1 \le r \le n-1$, the set  $\ALn{n,r}:= \{ (\alpha, L) \mid L \in \Ln{n,r}  \} \subset \ALn{n}$ is a normal subgroup of $\ALn{n}$.
\end{cor}

We now show that $(\Ln{n})\rsig = \Un{n-1,n}$. We start by defining the map $\rsig$ as in \cite{OlukoyaAutTnr}.

\begin{predef}
Let $T$ be a non-initial transducer. For  each $q \in Q_{T}$ let $B(q)$ be the smallest subset of $\xns$ such that for $\nu \in B(q)$  $U_{\nu} \subset \im{q}$ and $\bigcup_{\nu \in B(q)} U_{\nu} =  \im{q}$ and set $m_{q} := |B(q)|$.
\end{predef}

The following result is from \cite{OlukoyaAutTnr}:

\begin{prop}
	Let $T \in \SOn{n}$ then for any pair $p,q \in Q_{T}$, $m_{p} \equiv m_{q} \pmod{n-1}$.
\end{prop}

\begin{predef}
	Let $T \in \On{n}$, then set $(T)\rsig:= m_q \pmod{n-1}$ for $q \in Q_{T}$.
\end{predef}

\begin{prerk}\label{rem:rsigcomputedfromunminimised}
Let $T$ be a core,  strongly synchronizing transducer all of whose states induce continuous maps with clopen image. Let $T_{q_{-1}}$ be the transducer obtained from $T$ by applying the algorithm of Proposition~\ref{prop:algorithmforremovingincompleteresponse} to $T_{q}$ (for $q \in Q_{T}$). Recall that the same set of symbols represents the states of $\core(T')$ and $T$. To avoid confusion denote by $p'$, for $p \in Q_{T}$, the corresponding state of $T'$. Observe that by equation~\ref{redifining output} of Proposition~\ref{prop:algorithmforremovingincompleteresponse},  $\im{p'} = \im{p} - \Lambda(\varepsilon, p) = \{ \delta - \Lambda(\varepsilon, p) \mid \delta \in \im{p} \}$. Thus $m_{p}$ and $m_{p'}$ are the same for all $p \in Q_{T}$. Clearly, $\omega$-equivalent states have the same $m_{q}$ values. This means that, in order to compute $(T)\rsig$ one need not work with a minimal representative  of $T$.
\end{prerk}

\begin{theorem}\label{thm:Lndivisorsofn}
	The map $\rsig: \Ln{n} \to  \Un{n-1}$ has image $\Un{n-1,n}$.
\end{theorem}
\begin{proof}
	By results in \cite{OlukoyaAutTnr} we need only show that $(\Ln{n})\rsig \le \Un{n-1,n}$. 
	
	Let $T \in \Ln{n}$, by Theorem~\ref{t:hed2}, Proposition~\ref{prop:homomorphismfromPntoFinfinitiy} and Corollary~\ref{cor:incompleteresponsecreateslaginpn}, there is an element $P \in \pn{n}$ such that $[P]$, the minimal transducer representing $P$, is equal to $T$. Let $U$ be the inverse of $T$ in $\Ln{n}$, and let  $R \in \pn{n}$ be such that $[R] = U$. By Remark~\ref{rem:rsigcomputedfromunminimised}, it suffices to show that for any state $p \in Q_{P}$, $m_{p}$ is congruent  modulo $n-1$ to an element of $\Un{n-1,n}$. Let $m_{P} = m_{p} \pmod{n-1}$ and $m_{R} = m_{q}\pmod{n-1}$ for states $p$ and $q$ of $P$ and $R$ respectively.  
	
	Theorem~\ref{thm:EndisotoASLn} indicates that $[P \spnprod{n} R]$ the minimal representative of $P\spnprod{n}R$, is equal to $\id$, the single state identity transducer. This means that for a state $q$ of $P \spnprod{n}R$, $\im{q} = \{\Lambda(\varepsilon, q)\rho \mid \rho \in \xnn\}$. Recall that $P \spnprod{n}R$ is equal to $\core(P \ast R)$ with $\omega$-equivalent states identified. We therefore focus our analysis on the transducer $\core(P\ast R)$. Furthermore, as $\core(P \ast R)$ will occur frequently in our argument, we abuse notation and set $P \ast R := \core(P \ast R)$.
	
	Let $j:= \max\{ |\nu| \mid \nu \in B(q), q \in Q_{R} \}$ and let $k$ be the minimal synchronizing level of $R$. For $q \in Q_{R}$ let $\mathsf{B}(q):= \{ \lambda_{R}(\nu, q) \mid \nu \in \xn^{j} \}$. Observe that as $R$ is synchronous, then, for any $q \in Q_{R}$, the size of all elements of $\mathsf{B}(q)$ is $j$ . Further notice that by choice of $j$,  for any $q \in Q_{R}$, $\bigcup_{\mu \in \mathsf{B}(q)} U_{\mu} = \im{q}$. These two facts imply that $\mathsf{m}_{q}:=|\mathsf{B}(q)| \equiv m_{q} \pmod{n-1}$. 
	
	Let $\Gamma_0, \Gamma_1, \ldots, \Gamma_{n^{k}}$ be the elements of $\xn^{k}$ ordered lexicographically, and   $q_0, q_1, \ldots, q_{n^{k}}$ be the sequence of states of $R$ such that $q_i$  is the state forced by $\Gamma_i$ for all $1 \le i \le n^{k}$. Set $M_{R}:= \sum_{i=1}^{n^{k}} \mathsf{m}_{q_i}$, noting that as $\mathsf{m}_{q} \equiv m_{R} \pmod{n-1}$, then $M_{R} \equiv m_{R} \pmod{n-1}$ also. 
	
	For each state $q \in Q_{R}$, set $(\xn^{j+k})q:= \{ \lambda_{R}(\Gamma,q) \mid \Gamma \in \xn^{j+k} \}$ and notice that $|(\xn^{k+j})q|$ is precisely $n^{k}M_{R}$. This is because if, for a fixed state $q \in Q_{R}$, there were two distinct words $\Delta_{1}, \Delta_{2} \in \xn^{k}$ for which $\lambda_{R}(\Delta_1, q) = \lambda_{R}(\Delta_2, q)$, then, injectivity of the map $R_{q}$ implies that for any pair  $\nu, \mu \in \xn^{j}$, $\tau:= \lambda_{R}(\nu, \pi_{R}(\Delta_1,q)) \ne  \lambda_{R}(\mu, \pi_{R}(\Delta_2, q))=: \eta$, since $U_{\tau} \subset \im{\pi_{R}(\Delta_1,q)}$ and $U_{\eta} \subset \im{\pi_{R}(\Delta_2, q)}$. Therefore, as the value $M_{R}$ depends on $T$ only, we see that for $q \in Q_{R}$, the value $|(\xn^{k+j})q|$ is the same for all $q$. Note also that for a given $q \in Q$, by choice of $j$, $\bigcup_{\nu \in (\xn^{j+k})q} U_{\nu} = \im{q}$.
	
	Fix a state $(p,t)$ of $P \ast R$ and set $\varphi:= \Lambda(\varepsilon, (p,t))$. Since $P \ast R$ has minimal representative $\id$, it must be the case that $\im{(p,t)} = U_{\varphi}$. Let $l':= \max\{ |\nu\| \nu \in B(s), s \in Q_{P}\}$ and set $l = \max\{l, |\varphi|\}$. As before, set $\mathsf{B}(p):= \{ \lambda_{P}(\nu, p) \mid \nu \in \xn^{l} \}$ and let $\mathsf{m}_{p}:= |\mathsf{B}_{p}|$, noting, as before, that $\mathsf{m}_{p} \equiv m_{p} \pmod{n-1}$. 
	
	Let $\Xi_1, \Xi_2, \ldots, \Xi_{\mathsf{m}_{p}}$ be an ordering of the elements of $\mathsf{B}(p)$ and let $t_1, t_2, \ldots,$ $ t_{\mathsf{m}_{p}}$ be defined such that $\pi_{R}(\Xi_i, t) = t_i$ for all $1 \le i \le \mathsf{m}_{p}$. Since, for all $1 \le i \le \mathsf{m}_{p}$, $U_{\Xi_{i}} \subset \im{p}$, and since $P$ is synchronous, it follows, setting $(\xn^{l+j+k})p := \{ \lambda_{P}(\nu, p) \mid \nu \in \xn^{l+j+k} \}$, that  $\{\Xi_i \nu \mid \nu \in \xn^{j+k} \} \subseteq (\xn^{l+j+k})p$ for all $1 \le i \le \mathsf{m}_{p}$,. From this, we deduce, setting $(\xn^{l+j+k})(p,t) := \{ \lambda_{P\ast R}(\nu, (p,t)) \mid \nu \in \xn^{l+j+k} \}$, that $|(\xn^{l+j+k})(p,t)| = \mathsf{m}_{p}n^{k}M_{R}$. This is because if there were $1 \le i_1, i_2 \le \mathsf{m}_{p}$ such that $\lambda_{R}(\Xi_{i_1}, t)= \lambda_{R}(\Xi_{i_2}, t)$, then $(\xn^{j+k})t_{i_1} \cap (\xn^{j+k})t_{i_2} = \emptyset$ as otherwise $R_{t}$ is not injective. 
	
	Now observe that the equality $\im{(p,t)} = U_{\varphi}$, implies, by choice of $l$ and as $P\ast R$ is synchronous, that $|(\xn^{l+j+k})(p,t)| = n^{l+j+k - |\varphi|}$. Therefore $\mathsf{m}_{p}n^{k}M_R = n^{l+j+k - |\varphi|}$. Since $\mathsf{m}_{p} \equiv m_{P} \pmod{n-1}$ and $M_{R} \equiv m_{R} \pmod{n-1}$, we therefore see that, by the fundamental theorem of arithmetic, modulo $n-1$, $m_{P}$ and $m_{R}$ are both products of divisors of $n$ as required.  
\end{proof}

\begin{cor}
	Let $n$ be a prime number. Then $\Ln{n} = \Ln{n,1}$.
\end{cor}

\begin{ntn}
	For $1 \le r \le n-1$ set $\Un{n-1,n}^{r}:= \Un{n-1,n} \cap \Un{n-1}^{r}$. That is, $\Un{n-1,n}^{r}$ is the stabiliser of $r$ in the action of the subgroup $\Un{n,n-1}$ on $\Z_{n-1}$ by multiplication (identifying $0$ and $n-1$).
\end{ntn}

\begin{cor}
	Let $1 \le r, s \le  n-1$. Then $\Ln{n,r} \subset \Ln{n,s}$ if and only if $\Un{n-1,n}^{r} \subset \Un{n-1,n}^{s}$.
\end{cor}
\begin{proof}
	This is a consequence of point~\ref{rsig:stabliser} of Theorem~\ref{thm:rsig} and Theorem~\ref{thm:Lndivisorsofn}. 
\end{proof}

\begin{prerk}
	Let $i,j, r \in \{1,2,\ldots, n-1\}$. Suppose that $i$ is minimal such that $ir \equiv i\pmod{n-1}$. If $j$ also satisfies $jr \equiv j \pmod{n-1}$, then $i \mid j$. Thus, for any  $1 \le r \le n-1$, since $\Un{n-1,n}^{r} \subset \Un{n,n-1}^{s}$, the minimal  $1 \le s \le n-1$ for which $\Un{n-1,n}^{r} = \Un{n-1,n}^{s}$ is a divisor of $n-1$.
\end{prerk}

For a subgroup $G \le \Un{n-1,n}$, set $\Ln{n,G}:= \{ L \in \Ln{n} \mid (L)\rsig \in G\}$. The following result is clear.

\begin{cor}
	Let $G,H \le \Un{n-1,n}$, then the following hold:
	\begin{enumerate}[label=(\alph*)]
		\item $\Ln{n,H} \cap \Ln{n,G} = \Ln{n, H \cap G} $,
		\item $\gen{\Ln{n,H}, \Ln{n,G}} = \Ln{n, \gen{H, G}}$, and,
		\item if $H \le G$, then $\Ln{n,H} \unlhd \Ln{n,G}$ and $G/H \cong \Ln{n,H} / \Ln{n,G}$.
	\end{enumerate}	
\end{cor}

In a sequel paper we consider and develop consequences of an extension of the map $\rsig$ which is closely related to the dimension representation $\dimr$.

\subsection{Applications to $\On{n}$}

We now consider consequences of Corollary~\ref{cor:mainresult} for $\On{n} \cong \out{G_{n,n-1}}$. We first require a few definitions.

For $k \in \N$ let $\wn^{k}$ be the set of prime words of length $k$ over the alphabet $\xn$ and let $\wnl{\ast}$ be the set of all prime words over the alphabet $\xn$.

Let $n \in \N$, for two words $\Gamma = \gamma_0\gamma_1\ldots \gamma_{n-1}$ and $\Delta = \delta_0\delta_1\ldots \delta_{n-1}$ in  $\xn^{n}$, we say that $\Delta$ is \emph{rotationally equivalent} to $\Gamma$ (written $\Delta \rot \Gamma$) if there is an $i \in \Z$ such that for $x,y \in \xnz$ with $x_{j} = \delta_{j\pmod{n}}$ and $y_{j } =\gamma_{j \pmod{n}}$, then $(x)\shift{n}^{i} = y$. Equivalently $\Gamma \rot \Delta$ if and only if there is a number $0 \le i \le n-1$ such that $\delta_i\delta_{i+1}\ldots\delta_{n-1}\delta_{0}\ldots\delta_{i-1} = \Gamma$.

Set $\rwnl{\ast} = \wns/ \rot$ and for $k \in \N$, set $\rwnl{k} = \wn^{k}/ \rot$. For a word $\Gamma \in \wns$ write $[\Gamma]_{\rot}$ for the rotational equivalence class of $\Gamma$.

Let $L \in \Ln{n}$ and $k \in \N$. Define a map $(L)\Pi_{k}: \rwnl{k} \to \rwnl{k}$ by $$(\rotc{\Gamma})(L)\Pi_{k} \mapsto \rotc{\lambda_{L}(\Gamma, q_{\Gamma})}$$ where  $\Gamma \in \rotc{\Gamma}$ and $q_{\Gamma}$ is the unique state of $L$ such that $\pi_{L}(\Gamma, q_{\Gamma})= q_{\Gamma}$ (for $k=0$ we take $(\varepsilon)(L)\Pi_{0} = \varepsilon$). Note that there is a unique circuit labelled by $\Gamma$ in $L$.  Furthermore, for $\Gamma = \gamma_0\gamma_1\ldots\gamma_{k}\in \wn^{k}$, setting $x\in \xnz$ such that $x_{i} = \gamma_{ (k-1 -i) \pmod{k}}$, and choosing any annotation  $\alpha$ of $L$ so that $(L,\alpha) \in \ALn{n}$, then for $y \in \xnz$ satisfying $(x)(L,\alpha) = y$, we have $y_0y_1\ldots y_{k-1} \in  (\rotc{\Gamma})(L)\Pi_{k}$, by definition of $(L,\alpha)$, and, since $(L, \alpha)$ is an element of $\aut{\xnz, \shift{n}}$, $y_0\ldots y_{k-1} \in \wnl{k}$ as well. Thus, we see that the map $(L)\Pi_{k}$ is a well-defined bijection for all $k \in \N$ (although one may also deduce this from results in \cite{AutGnr}). Write $(L)\Pi = ((L)\Pi_{k})_{k \in \N}$ for the element of $\prod_{k \in \N} \sym{\xn^{k}}$.  

Notice that elements of $\ALn{n}$ with the same image in $\Ln{n}$ induce the same map. The map $\Pi: \Ln{n} \to \prod_{k \in \N} \sym{\rwnl{k}}$ is the \emph{periodic orbit representation} of the paper \cite{BoyleLindRudolph88}. The following results, restated in the language of this manuscript, are from  \cite{BoyleLindRudolph88}: the first corresponds to Proposition 7.5  and the second to Theorem 7.2.

\begin{prop}\label{prop:faithfulrep}
	The map $\Pi: \Ln{n} \to \prod_{k \in \N} \sym{\rwnl{k}}$ is a monomorphism.
\end{prop}

\begin{theorem}\label{thm:faithfulrepdense}
	Let $k \ge 1$ and $\rotc{\Gamma}, \rotc{\Delta} \in \rwnl{k}$. Then there is an element $L \in \Ln{n}$ such that $(L)\Pi_{l}$ is the identity map for all $1 \le l <k$, and $(L)\Pi_{k}$  is the transposition swapping $\rotc{\Gamma}$ with $\rotc{\Delta}$
\end{theorem}

We now extend the map $\Pi$ to the monoid $\Mn{n}$. First we extend the notation $\rotc{\Gamma}$ to non-prime words. For $\Gamma \in \xns$ we write $\rotc{\Gamma}$, by a slight abuse of notation, for the the set $\rotc{\gamma}$ where $\gamma \in \wn$ and $\Gamma = \gamma^{r}$ for some $r \in \N$. Let $T \in \Mn{n}$ then, as before, define a map $(T)\Pi: \rwnl{\ast} \to \rwnl{\ast}$ by $\rotc{\Gamma} \mapsto \rotc{\lambda_{T}(\Gamma, q_{\Gamma})}$ where $\Gamma \in \rotc{\Gamma}$ and $q_{\Gamma}$ is the unique state of $T$ such that $\pi_{T}(\Gamma, q_{\Gamma}) = q_{\Gamma}$ (we take $(\varepsilon)(L)\Pi_{0} = \varepsilon$). Lemmas 8.1 and 8.2 of \cite{AutGnr} indicate that for all $T \in \On{n}$ the map $(T)\Pi$ is well-defined and is an element of $\sym{\rwnl{\ast}}$. For a general element $T \in \Mn{n}$, $(T)\Pi$ is well-defined transformation of the set $\rwnl{\ast}$, since for any prime word $\gamma \in \wn$ there is a unique state $q_{\Gamma} \in Q_{T}$ such that $\pi_{T}(\gamma, q_{\gamma}) = q_{\gamma}$.
\begin{ntn}
	Let $\Gamma \in \xnp$, and $T \in \Mn{n}$, we write $q_{\Gamma}$ for the unique state of $T$ such that $\pi_{T}(\Gamma, q_{\Gamma}) = q_{\Gamma}$.
\end{ntn}

\begin{ntn}
	For $X$ a set, we denote by $\tran{X}$, the monoid of all maps from $X$ to itself.
\end{ntn}

The following lemma is useful for the next result.

\begin{lemma}\label{lem:monolemma}
	Let $A \in \Mn{n}$ be a transducer with no states of incomplete response and let $B$ be a strongly synchronizing transducer. Let $p,q$ be any pair of states of $Q_{A}$ and $Q_{B}$ respectively. Suppose that for any word $c \in \xnp$ such that $\pi_{A}(c, p) = p$ and $\pi_{B}(c, q)= q$, $\lambda_{A}(c, p) = \lambda_{B}(c,q)$. Then the minimal representative of $B$ is equal to $A$.
\end{lemma}
\begin{proof}
	Let $x \in \xnp$ be such that $\lambda_{A}(x, p) \ne \varepsilon$.  By the strongly synchronizing hypothesis, for any $\rho \in \xnp$ there is a suffix $y \in \xnp$,  such that $\pi_{A}(x\rho y, p) = p$ and $\pi_{B}(x \rho y, q) = q$. The hypothesis of the lemma now implies that  the maps $h_{A_{p}}$ and $h_{B_{q}}$ coincide on the clopen set $U_{x}$. However, the strong connectivity of the transducers $A$ and $B$ together with the minimality of $A$, therefore implies that the minimal representative of $B$ is $A$. 
\end{proof}

\begin{prop}
	The map $\Pi: \Mn{n} \to \tran{\rwnl{\ast}}$ is a monomorphism.
\end{prop}
\begin{proof}
   We begin by showing that $\Pi$ is a homomorphism.
	
	Let $T, U \in \Mn{n}$ and $\Gamma \in \wnl{+}$. Let $\Delta = \lambda_{T}(\Gamma, q_{\Gamma})$, noting that $\rotc{\Delta} = (\rotc{\Gamma})(T)\Pi$ and let $p_{\Delta}$ be the state of $U$ such that $\pi_{U}(\Delta, p_{\Delta}) = p_{\Delta}$. Observe that $(q_{\Gamma}, p_{\Delta})$ is the unique state of $\core(T \ast U)$ satisfying $$\pi_{T \ast U}(\Gamma, (q_{\Gamma}, p_{\Delta})) = (q_{\Gamma}, p_{\Delta}).$$ Let $\Theta:= \lambda_{T\ast U}(\Gamma, (q_{\Gamma}, p_{\Delta}))$, noting that $\rotc{\Theta} = (\rotc{\Delta})(U)\Pi$.
	
	If $TU = Z_{x}$ for some prime word $x \in \wn$ so that $Z_{x}$ is minimal, then it must be the case that all states of $\core(T \ast U)$ have infinite extent of incomplete response. In particular for any state $(q,p)$ of $\core(T \ast U)$ $(T\ast U)_{(q,p)}$ is $\omega$-equivalent to $Z_{x'}$ for some rotation $x'$ of $x$. Therefore we see that $\rotc{\Theta} = \rotc{x}$ and so $(\rotc{\Gamma})(TU)\Pi = ((\rotc{\Gamma})(T)\Pi)(U)\Pi$.
 	
	If $TU \ne Z_{x}$ for some prime word $x \in \wn$ then all states of $TU$ have finite extent of incomplete response. In this case, the element $TU$ of $\Mn{n}$ is obtained by taking the core of the minimal representative of the initial transducer $\core(T \ast U)_{(q_{\Gamma}, p_{\Delta})}$. Let $(TU)'_{q_{-1}}$ be the transducer obtained from $\core(T \ast U)_{(q_{\Gamma}, p_{\Delta})}$ by applying the algorithm of Proposition~\ref{prop:algorithmforremovingincompleteresponse}. Using `$'$' above the elements of $Q_{(TU)'}$ which also appear in $Q_{\core(T\ast U)}$, we observe that $$\pi_{(TU)'}(\Gamma, (q_{\Gamma},p_{\Delta})') = (q_{\Gamma},p_{\Delta})' \mbox{ and, }$$  $$\lambda_{(TU)'}(\Gamma, (q_{\Gamma},p_{\Delta})') = \Theta \Lambda(\varepsilon,(q_{\Gamma},p_{\Delta})') - \Lambda(\varepsilon,(q_{\Gamma},p_{\Delta})').$$ Now, since $\Lambda(\varepsilon,(q_{\Gamma},p_{\Delta})')$ is a prefix of some power of $\Theta$, it follows that $\Theta \Lambda(\varepsilon,(q_{\Gamma},p_{\Delta})') - \Lambda(\varepsilon,(q_{\Gamma},p_{\Delta})')$  is an element of $\rotc{\Theta}$. By Lemma~\ref{lem:identifyinequivstatesmakesnodiff}, it follows that the unique state of $t$ of $TU$ for which $\pi_{TU}(\Gamma, t) = t$ satisfies $\lambda_{t}(\Gamma, t) = \Theta \Lambda(\varepsilon,(q_{\Gamma},p_{\Delta})') - \Lambda(\varepsilon,(q_{\Gamma},p_{\Delta})')$. Thus $((\rotc{\Gamma})(T)\Pi)(U)\Pi = \rotc{\Theta} = (\rotc{\Gamma})(TU)\Pi$ as required.
	
	We now show that $\Pi$ is injective. Let $T, U \in \Mn{n}$ such that $(T)\Pi = (U)\Pi$ 
	
	Suppose that $T = Z_{x}$ for some $x \in \wn$, noting that by definition $$\im{(T)\Pi} = \{ \rotc{x}\}.$$ If $U$ has no states of incomplete response, then $\im{(U)\Pi}$ has size at least $2$. For let $a \in \xn$ $q_{a} \in Q_{U}$ satisfy $\pi_{U}(a, q_a) = q_a$. Let $\Gamma = \lambda_{U}(a, q_a)$. There is a word $b \in \xnp$ such that $\pi_{U}(b, q_a) = q_a$ and $\Delta = \lambda_{U}(b, q_a)$ has no prefix in common with $\Gamma$. Let $\gamma \in \wn$ such that $\Gamma$ is a power of $\gamma$. Then we observe that $\Gamma \Delta$ cannot be a power of $\gamma$ and so $\rotc{\Gamma} \ne \rotc{\Gamma \Delta}$, furthermore $\rotc{\Gamma \Delta}$ is in the image of $(U)\Pi$. Thus if $(T)\Pi = (U)\Pi$, then $T = U = Z_{x}$ by minimality of elements of $\Mn{n}$.
	
	Therefore we may assume that both $T$ and $U$ have no states of incomplete response. 
	
	We begin with the following observation. Let $a \in \wn$ be a prime word and let $p_{a} \in Q_{T}$ and $q_{a} \in Q_{U}$ be the unique states for which $\pi_{T}(a, p_a) = p_a$ and $\pi_{U}(a, q_{a}) = q_a$. Let $\Gamma =\lambda_{T}(a, p_a)$ and $\Gamma' = \lambda_{U}(a, q_{a}) = q_a$. Then it must be the case that $|\Gamma| = |\Gamma'|$. For suppose $|\Gamma | \ne |\Gamma'|$. let $\gamma \in \rotc{\Gamma}$, $\gamma' \in \rotc{\Gamma'}$ and $r,r' \in \N$ be such that $\gamma^r = \Gamma$ and ${\gamma'}^{r'}= \Gamma'$. We assume (swapping $T$ and $U$ if necessary) that $r >  r'$. Since $T$ has no states of incomplete response, there is a word $b \in \xnp$, such that $\pi_{T}(b, p_{a}) = p_{a}$, and $\Delta = \lambda_{T}(b, p_{a})$ has no initial prefix in common with $\Gamma$. By prepending, if necessary, a large enough power of $a$ to $b$, we may further assume that $\pi_{U}(b, q_{a}) = q_a$. Let $\Delta' = \lambda_{U}(b, q_{a})$. Let $j \in \N$, $j \ge 2$, be such that $|\gamma^{j}|> r'j|\gamma|+|\Delta'|$. Note that, choosing a longer $j$ if necessary, since $\gamma$ is a prime word, and $\Delta$ has no prefix in common with any power of $\gamma$, $\gamma^{rj}\Delta$ is a prime word. This immediately forces that $a^{j}b$ is also a prime word since $T$ is strongly synchronizing and non-degenerate. Thus we see that $(\rotc{a^{j}b})(T)\Pi = \rotc{\gamma^{rj}\Delta}$  is not equal to $\rotc{\gamma^{r'j}\Delta'} = (\rotc{a^{j}b})(U)\Pi$ as $\gamma^{rj}\Delta$ is a prime word strictly longer than $\gamma^{r'j}\Delta'$. This yields the desired contradiction Therefore we conclude that $r = r'$.
	
	Suppose that there is a prime word $a \in \wn$ such that for $p_{a} \in Q_{T}$ and $q_{a} \in Q_{U}$ satisfying $\pi_{T}(a, p_a) = p_a$ and $\pi_{U}(a, q_{a}) = q_a$, $\lambda_{T}(a, p_a) \ne \lambda_{U}(a, q_a)$. Since $(T)\Pi = (U)\Pi$ and  $|\lambda_{T}(a, p_a)| = |\lambda_{U}(a, q_a)|$, it is the case that $\lambda_{U}(a, q_a)$ is a non-trivial rotation of $\lambda_{T}(a, p_a)$. Let $\gamma, \gamma' \in \wn$ and $r \in \N$ be such that $\Gamma:=\gamma^r = \lambda_{T}(a, p_a)$ and $\Gamma':={\gamma'}^{r} = \lambda_{U}(a, q_a)$. Let $\gamma_1 \gamma_2 \ldots \gamma_{t} \in \xn$ and $k \in \N$ with $1 < k \le t$ be such that $\gamma = \gamma_1 \gamma_2 \ldots \gamma_t$ and $\gamma' = \gamma_{k}\gamma_{k+1} \ldots\gamma_{t}\gamma_{1}\ldots\gamma_{k-1}$.
	
	We now argue that for any $N \in \N$, there is a prime word $\Theta = \Theta_1 \ldots \Theta_m$ of length $m$ larger than $N$ which is equal to $\lambda_{T}(\Omega, p_a)$ for some (necessarily prime) word $\Omega$ satisfying $\pi_{T}(\Omega, p_a) = p_a$ and $\pi_{U}(\Omega, q_a) = q_a$. Moreover, if $\Theta' = \Theta_{i}\ldots \Theta_{m}\Theta_{1}\ldots\Theta_{i-1}$ is the word $\lambda_{U}(\Omega, q_a)$, then, $i >1$ and $m-i+1 \ge N$.
	
	Let $N \in \N$, and let $c \in \xnp$ be a word such that, as above, $\pi_{T}(c, p_a) = p_a$, $\pi_{U}(c, q_a) = q_a$, $\nu' = \lambda_{U}(c, q_a)$ has no initial prefix in common with $\Gamma'$, and $|\nu'| > N$. There is a $j \in \N$, such that $|\Gamma'|j > |\nu'|$ and ${\Gamma'}^{j}\nu'$ is a prime word since $\nu'$ has no initial prefix in common with $\Gamma'$. Let $\nu = \lambda_{T}(c, p_a)$, noting that $\nu$, is a rotation of $\nu'$. Consider the words $\Gamma'^{j}\nu' = \lambda_{U}(a^{j}c,q_a)$ and $\Gamma^{j} \nu = \lambda_{T}(a^{j}c,p_a)$ which are rotations of each other. Suppose $\nu= \nu_1 \nu_2$ such that $\nu_2$ is non-empty and $\nu_2\Gamma^{j} \nu_1= \Gamma'^{j}\nu'$. Since $\gamma$ is prime and $\gamma'$ is a non-trivial rotation of $\gamma$, it must be the case choosing a larger $j$ if necessary, that $\nu_2 = (\gamma')^{l}\gamma_{k}\ldots\gamma_{t}$. We may assume that $rj - l \ge 2$ so that $(\gamma)^{rj} = (\gamma)^{rj-l-1}\gamma_{1}\ldots\gamma_{k-1}\nu'$. However this forces that $\nu$ has a non-trivial prefix in common with $\gamma'$ yielding a contradiction. Therefore it must be the case that  $\Gamma^{j} = \mu_1 \mu_2$ for some non-empty suffix $\mu_2$, and $(\Gamma')^{j} \nu' = \mu_2\nu \mu_1$. Notice that since $\gamma$ is a prime word  and $\Gamma'$ is a non-trivial rotation of $\Gamma$, $\mu_1$ cannot be empty. Setting $\Theta = \Gamma^{j}\nu$ we are done.
	
	Fix $M,m \in \N$ be minimal such that for any word $x \in \xnp$ of length $M$, and any pair of states $p \in Q_{T}$ and $q \in Q_{U}$, $1 \le |\lambda_{T}(x,p)|,|\lambda_{U}(x,q)| \le m$. 
	
	By the above, there is a prime word $\Theta= \Theta_1 \Theta_2 \ldots \Theta_{s}$, $i \in \N$ such that $i >1$ and $s-i-1 \ge 2m$, and a prime word $\Omega \in \wn$ such that $\pi_{T}(\Omega, p_a) = p_a$, $\pi_{U}(\Omega, q_a) = q_a$, $\Theta = \lambda_{T}(\Omega, p_a)$ and $\Theta' = \Theta_{i}\ldots \Theta_{s}\Theta_{1}\ldots\Theta_{i-1} = \lambda_{U}(\Omega, q_a)$. Using the fact that $T$ has no states of incomplete response, let $b \in \xn^{M}$ be a word,  such that $\lambda_{T}(b,p_a)$ has no prefix in common with $\Theta$. In a similar way, and using the fact that $s-i-1 \ge 2m$, there is a word $c \in \xn^{M}$ such that $\lambda_{U}(bc, q_a)$ is incomparable with  $\Theta_{i} \ldots \Theta_{s}$. Further there is a word $d \in \xnp$ such that $\pi_{T}(bcd, p_a) = p_a$ and $\pi_{T}(bcd, q_a)= q_a$. Let $\Delta = \lambda_{T}(bcd, p_a)$ and $\Delta' = \lambda_{U}(bcd, q_a)$, noting that $\Delta'$ is a rotation of $\Delta$, $\Delta$ has no initial prefix in common with $\Theta$ and $\Delta'$ is incomparable with $\Theta_{i} \ldots \Theta_{s}$.
	
	As before there is a $j \in \N$ large enough, so that $\Theta^{j}\Delta$ is a prime word, as $\Theta$ is a prime word and $\Delta$ has no prefix in common with $\Theta$. Since $\Theta^{j}\Delta = \lambda_{T}(\Omega^{j} bcd, p_a)$ and $\Theta'^{j}\Delta' = \lambda_{U}(\Omega^{j} bcd, q_a)$, it has to be the case that $\Theta'^{j}\Delta'$ is a rotation of $\Theta^{j}\Delta$. We show that this yields a contradiction.
	
	First suppose that  $\Delta' = \mu_1 \mu_2$ for some non-empty suffix $\mu_2$ such that $\mu_2(\Theta')^{j}\mu_1 = \Theta^{j}\Delta$. since $\Theta'$ is a prime word, for sufficiently large $j$, it must be the case that $\mu_2 = \Theta^{l}\Theta_{1}\ldots\Theta_{i-1}$ for some $l< j$ (since the length of $\Delta$ is fixed we can choose $j$ such that $j-l$ is as large as we please). This then implies that $(\Theta')^{j-l-1}\Theta_i \ldots\Theta_{s} \Delta = (\Theta')^{j}\mu_1$. However, for sufficiently large $j$, this implies that $\Delta$ has a non-trivial prefix in common with $\Theta$ yielding the desired contradiction.
	
	Thus we may assume that $(\Theta')^{j} =\mu_1 \mu_2$ for some non-empty suffix $\mu_2$ such that $\mu_2\Delta'\mu_1 = \Theta^{j}\Delta$. Since $\Theta$ is a prime word, it must be the case that, for sufficiently large $j$, $\mu_2 = (\Theta)^{l}\Theta_{1}\ldots\Theta_{i-1}$ for some $1 \le l < j-1$. This then means that $\Theta_1 \ldots \Theta_{i-1} \Delta'$ is a prefix of $(\Theta)^{j-l}\Delta$, from which we deduce that $\Delta'$ is not incomparable to $\Theta_{i}\ldots\Theta_{s}$ yielding the desired contradiction.
	
	Therefore it must be the case that for any prime word $a \in \wn$ such that for $p_{a} \in Q_{T}$ and $q_{a} \in Q_{U}$ satisfying $\pi_{T}(a, p_a) = p_a$ and $\pi_{U}(a, q_{a}) = q_a$, $\lambda_{T}(a, p_a) = \lambda_{U}(a, q_a)$. However, in this case, by Lemma~\ref{lem:monolemma}, $T = U$ as required. 

	\end{proof}
\begin{cor}\label{cor:periodicperm}
	The restrictions $\Pi: \SOn{n} \to \tran{\rwnl{\ast}}$, and $\Pi: \On{n} \to \sym{\rwnl{\ast}}$  are monomorphisms.
\end{cor}

\begin{prop}\label{prop:bijectioniffOn}
	Let $T \in \Mn{n}$. Then $(T)\Pi$ is bijective if and only if  $T \in \On{n}$.
\end{prop}
Before proving this proposition, we prove the following auxiliary result.
	\begin{lemma}\label{lem:injectiveiffstatesinjective}
		Let $T \in \Mn{n}$ and suppose that $(T)\Pi$ is injective. Then every state of $T$ is injective.
	\end{lemma}
	\begin{proof}
	Suppose there exists a state $q \in Q_T$ and distinct elements $\nu, \mu \in \xn^{\omega}$ such that $(\nu)h_{q} = (\mu)h_{q} = \delta = \delta_1\delta_2 \ldots$ for $\delta_i \in \xn$.
	
	Since $T$ has finitely many states there is a state $p \in Q_{T}$ and elements $\nu_i \in \xnp$, $i \in \N$  with $\nu_1  \nu_2  \ldots = \nu$, $\pi_{T}(\nu_1,  q)= p$ and $\pi_{T}(\nu_i, p) = p$. For $i \in \N$ set $\eta_i = \nu_1 \ldots \nu_i$.
	
	By our assumption that transducers always write finite strings on finite input,  there is a $j \in \N$ such that for any element $t \in Q_{T}$ and any $x \in \xn$, $|\lambda_{T}(x, t)| < j$. In particular, there is a constant $J \in \N$, such that for every $i \in \N$ we may find a proper prefix $\tau_i$ of $\mu$ satisfying  $\left | |\lambda_{T}(\tau_i, q)| - |\lambda_{T}(\eta_i, q)| \right | \le J$.  We note that the length of $\tau_i$ increases with $i$ since by assumption transducers always write infinite strings on infinite inputs. 
	
	Once more, using finiteness of the states of $T$, there is a state $t \in Q_{T}$ such that for infinitely many $i \in \N$, $\pi_{T}(\tau_i, q) = t$. We may further assume, swapping the roles of $p$ and $q$ if need be, that for infinitely many $i \in \N$, $| \lambda_{T}(\tau_i, q)| < |\lambda_{T}(\eta_i, q)|$. Let  $\mathscr{J}$ be the set of those  $ i \in \N$ such that
	\begin{itemize}
		\item $\tau_i$ is incomparable with $\eta_i$ (there are infinitely may such $i$ since $\mu \ne \nu$),
		\item $\pi_{T}(\tau_i, q) = t$,
		\item $| \lambda_{T}(\tau_i, q)| < |\lambda_{T}(\eta_i, q)|$,  and ,
		\item  for $a, b \in \mathscr{J}$ $\tau_a = \tau_{b}$ if and only if $a = b$.
	\end{itemize}
	   In particular we note that for $a, b \in \mathscr{J}$ with $a < b$, $\tau_{a}$ is a proper prefix of $\tau_{b}$. We note that $|\mathscr{J}| = |\N|$. Let $a_1< a_2< \ldots$ be an ordering of the elements of $\mathscr{J}$. Set  $\mu_{a_1} = \tau_{a_1}$ and for $1<b$ we set $\mu_{a_b} = \tau_{a_{b}} - \tau_{a_{b-1}}$. We note that $\mu_{a_{1}}\mu_{a_{2}}\ldots = \mu$ since the $\tau_{a}$ are longer and longer prefixes of $\mu$.
	
	By construction for all $a \in \mathscr{J}$, $0 \le |\lambda_{T}(\eta_{a}, q)| - |\lambda_{T}(\tau_{a}, q)| \le J$. Therefore, since $\lambda_{T}(\tau_{a}, q)$ is a prefix of $\lambda_{T}(\eta_{a}, q)$,  there is an $a \in \mathscr{J}$ such that for infinitely many $b \in \mathscr{J}$, $$\lambda_{T}(\eta_a, q) - \lambda_{T}(\tau_{a}, q) = \lambda_{T}(\eta_{b}, q) - \lambda_{T}(\tau_{b}, q).$$  Fix such  $a, b \in \mathscr{J}$ and set $r = |\lambda_{T}(\eta_a, q)| - |\lambda_{T}(\tau_{a}, q)|$. By choosing $b$ large enough, we may assume that $r$ is strictly less than the maximum of $|\lambda_{T}(\eta_{b} - \eta_a, p)|$ and $|\lambda_{T}(\tau_b - \tau_{a}, t)|$.
	
   Let $m \in \N$ be such that $\lambda_{T}(\eta_{a}, q) = \delta_1 \ldots \delta_{m} \delta_{m+1} \ldots \delta_{m+r}$ and set $\lambda_{T}(\eta_{b} - \eta_a, p)= \delta_{m+r+1} \ldots \delta_{m+r+l}$. We note that $l > r$ by assumption. It now follows that  $\lambda_{T}(\tau_{a}, q) =  \delta_{1} \ldots \delta_{m}$ and  $\lambda_{T}(\tau_{b} - \tau_{a}, t) =  \delta_{m+1} \ldots \delta_{m+r}\ldots \delta_{m+ l}$. Using the fact that $\lambda_{T}(\eta_a, q) - \lambda_{T}(\tau_{a}, q) = \lambda_{T}(\eta_{b}, q) - \lambda_{T}(\tau_{b}, q)$, we conclude that  $\delta_{m+l+1} \ldots \delta_{m+r+l} = \delta_{m+1} \ldots \delta_{m+r}$. In particular, $$\delta_{m+r+1} \ldots \delta_{m+l} \delta_{m+l+1} \ldots \delta_{m+r+l}$$ is a rotation of $$\delta_{m+1} \ldots \delta_{m+r} \delta_{m+r+1} \ldots \delta_{m+r+l}$$ and so  $\lambda_{T}(\eta_{b} - \eta_{a}, p)$ is a rotation of $\lambda_{T}(\tau_{b} - \tau_{a}, t)$. 
   
   As $(T)\Pi$ is injective, we conclude  that  $\tau_{b} - \tau_{a}$ is a rotation of $\eta_{b} - \eta_{a}$. In this case, there is  a word $w \in \xns$ such that $\pi_{T}(w, t) = p$ and $\lambda_{T}(w, t) = \delta_{m+1} \ldots \delta_{m+r}$.  We therefore have that $\lambda_{T}(\tau_{b}w, q) = \lambda_{T}(\eta_{b}, q)$ and $\pi_{T}(\eta_{b}, q) = \pi_{T}(\tau_{b} w, q) = p$. By the strong synchronizing condition we may find a word $\rho \in \xnp$  such that $\pi_{T}(\rho, p) = q$.  Using the fact that $(T)\Pi$ is injective we thus conclude that $\tau_{b}w\rho$ is a rotation of $\eta_{b}\rho$.
   
   Since $\tau_{b}w\rho$ and $\eta_{b}\rho$ are both circuits based at $q$, which are non-trivial rotations of each other ($\tau_{b}$ is incomparable with $\eta_{b}$ by assumption), it follows that there are distinct prime words $x_1, x_2, \ldots, x_{c} \in \xnp$, such that each $x_{i}$, $1 \le i \le c$, represents a circuit based at $q$, and so that the following condition is satisfied. There is a word $w_1 w_2 \ldots w_d \in \{ x_1, \ldots, x_c\}^{+}$  and $1 < i \le d$ such that $w_1 \ne w_{i}$, $\tau_{b}w\rho = w_1\ldots w_d$, $\eta_{b}\rho = w_{i} \ldots w_{d}w_{1}\ldots w_{i-1}$. Since $\lambda_{T}(x_e, q) \ne \varepsilon$ for any $1 \le e \le c$, it follows  that $\lambda_{T}(\tau_{b} w\rho, q) = \lambda_{T}(\eta_{b}\rho, q)$ is equal to a non-trivial rotation of itself. Therefore there is a prime word $\gamma \in \xnp$, such that $\lambda_{T}(\tau_{b} w, q)$ is a non-trivial power of $\gamma$. In particular  $\lambda_{T}(w_1\ldots, w_{i-1}, q)$ and $\lambda_{T}( w_{i} \ldots w_{d}, q)$ are both powers of $\gamma$. 
   
   Injectivity of  $(T)\Pi$ now means that there is a prime word $\xi$ such that $w_1 \ldots w_{i}$ is a power of $\xi$ and $w_{i+1}\ldots w_{d}$ is a power of a non-trivial  rotation $\xi'$ of $\xi$. Since $w_1 \ldots w_{i-1}$ and $w_{i}\ldots w_{d}$ are both circuits based at $q$ and $w_1 \ne w_{i}$, we may decompose $\xi$ as $\xi = \xi_1 \xi_2$ for non-empty words $\xi_1$ and $\xi_2$ where $\xi' = \xi_2 \xi_1$ and $\pi_{T}(\xi_1, q) = \pi_{T}(\xi_2, q) = q$.  Observe that $\lambda_{T}(\xi_1, q) $ and $\lambda_{T}(\xi_2, q)$ are both powers of $\gamma$, therefore,  injectivity of $(T)\Pi$ now forces that $\xi_1$ and $\xi_2$ are both powers of $\xi$. This yields the desired contradiction.
   %
\end{proof}

\begin{proof}[of Proposition~\ref{prop:bijectioniffOn}]
	Clearly if $T \in \On{n}$ then $(T)\Pi$ is bijective.  Therefore it suffices to show that if $T \in \Mn{n}$  and $(T)\Pi$ is a bijection, then $T \in \On{n}$.
	
	Let $T \in \Mn{n}$ and suppose that $(T)\Pi$ is bijective. By Lemma~\ref{lem:injectiveiffstatesinjective} all the states of $T$ are injective. It therefore suffices to show that the image of every state of $T$ is clopen.
   
   Let $q \in Q_{T}$. Since the map $h_{q}: \xnn \to \xnn$ is injective and continuous, the image of $q$ is closed. We now prove that the image of $q$ is an open subset of $\xnn$.
   
   Let  $\gamma \in \xnp$ be such that $\pi_{T}(\gamma, q) = q$ and set $\mu^{m} = \lambda_{T}(\gamma, q)$ for $\mu$ a prime word. Since elements of $\Mn{n}$ are minimal, and $(T)\Pi$ is a bijection, then  $T$ has no states of incomplete response. In particular we may fine $\delta \in  \im{q}$ such that $\delta$ has no prefix in common with $\mu$. Let $\rho \in  \xnn$ be such that $(\rho)h_{q} = \delta$.
   
   Since $q$ is injective and every state of $T$ has closed image, there is a $j \in \N$ and  proper prefixes $\delta_1$ of $\delta$ and  $\rho_1$ of $\rho$ with $\lambda_{T}(\rho_1, q) \ne \varepsilon$, such that:
   \begin{quote}
   	for any word $\psi$ with $\psi$ incomparable to $\gamma$, and  $\mu^{i}\delta_1 \le \lambda_{T}(\psi, q)$ for some $i \in \N$, then $i \le j$ and there is an $a \in \N$ such that $\gamma^{a} \rho_1 \le \psi$. 
   \end{quote}

   Since $(T)\Pi$ is injective, there is an $N$ in $\N$ such that for any word $\tau \in \xnp$ and any state $p \in Q_{T}$ with $\mu^{N}$ a prefix of $\lambda_{T}(\tau, p)$,  there is a prefix $\tau_1$ of $\tau$ such that $\tau_1 \gamma$ is also a prefix of $\tau$ and $\lambda_{T}(\tau_1 \gamma, p)$ is non-empty and a power of $\mu$.
  
   Consider the word $\mu^{M} \delta_1 \mu^{M}$ .  Let $\eta$ be such that $(\eta)(T)\Pi = \mu^{M} \delta_{1} \mu^{M}$. By increasing $M$ we may assume that $\eta$ is longer than the minimal synchronizing length of $T$. Let  $p \in Q_{T}$ be the state of $T$ forced by $\eta$. 
   
   By preceding observations we may decompose	$\eta$ as $\eta= \eta_1 \gamma \eta_2$, where $\lambda_{T}(\eta_1 \gamma, p)$ is a power of $\mu$. Therefore,  $\lambda_{T}(\eta_2,q) = \mu^{i}\delta_1 \mu^{M}$, and there is an $a \in \N$ such that $\gamma^{a}\rho_1$ is a prefix of $\eta_2$. Decompose  $\eta_2$ as $\gamma^{a}\rho_1 \eta_3$. Since $\lambda_{T}(\rho_1 \eta_3,q) = \delta_1  \mu^{M}$, as $\lambda_{T}(\rho_1,q)$ is a non-empty prefix of $\delta_1$, for sufficiently large $M$, we may write $\eta_3 = \eta_{3,1}\gamma^{l+1} \eta_{3,2}$ for some $l \in \N_{1}$. Thus, $\pi_{T}(\eta_{1}\gamma^{a+1}\eta_{3,2}, p) = p$ and $\lambda_{T}(\eta_1\gamma^{a+1}\eta_{3,2}, p)$ is a power of $\mu$. Therefore, there is  a prefix $\xi$ of $\eta_1$  such that $\pi_{T}( \xi,p) = q$ and $\lambda_{p}(\xi, p) = \varepsilon$. In particular we may assume that $p =q$ since, by replacing $\eta$ with $(\eta - \xi)\xi$ we have $\lambda_{T}(\eta, q) = \mu^{M} \delta_1  \mu^{M}$, and $\pi_{T}(\eta, q) = q$. In particular for $M$  sufficiently large, there is a word $\eta \in \xnp$ such that $\pi_{T}(\eta, q) = q$ and $\lambda_{T}(\eta, q) = \mu^{M}\delta_1 \mu^{M}$.
   
  Let $M \in \N$ be sufficiently large. By assumptions above, since $\lambda_{T}(\eta, q)$ has $\mu^{M} \delta_{1}$ as a prefix, we have $\gamma^{a}\rho_1$ is a prefix of $\eta$. Moreover, it must also be the case, as $\delta$ has no prefix in common with $\mu$, and $\lambda_{T}(\rho_1, q)$ is a prefix of $\delta$, that $\lambda_{T}(\gamma^{a}, q) = \mu^{am} = \mu^{M}$. However, this now means that for any sufficiently large $M \in \N$, $m |M$. This occurs if and only if $m =1$. In particular, $a = M$.  
  
  
  Let $\delta \in \im{q}$ and write $\delta = \mu^{b}  \tau$, where $\mu$ is not a prefix of $\tau$ and $b \in \N_{0}$. (Note that all elements of $\im{q}$ can be written in this way). Let $\rho \in \xnn$ such that $(\rho)h_{q} = \delta$. 
  
  As above, using injectivity of $q$ and the fact that every state of $T$ is closed, there is a $j \in \N$ and  proper prefixes $\delta_1$ of $\delta$ and  $\rho_1$ of $\rho$ with $\lambda_{T}(\rho_1, q) \ne \varepsilon$, such that:
  \begin{quote}
  	for any word $\psi$ with $\psi$ incomparable to $\gamma$, and  $\mu^{i}\delta_1 \le \lambda_{T}(\psi, q)$ for some $i \in \N$, then $i \le j$ and there is an $a \in \N$ such that $\gamma^{a} \rho_1 \le \psi$.  
  \end{quote}
We note that $a$ must in fact be equal to $i$.

Let $N \in \N$ be determined as before.
  
 Consider the word $\mu^{M} \delta_1 \nu \mu^{M}$ for an arbitrary word $\nu \in \xnp$ and $M \in \N_{N}$.  Let $\eta$ be such that $(\eta)(T)\Pi = \mu^{M} \rho_{1}\nu \mu^{M}$. By increasing $M$ we may assume that $\eta$ is longer than the minimal synchronizing length of $T$. Let  $p \in Q_{T}$ be the state of $T$ forced by $\eta$.  
 
 We may decompose	$\eta$ as $\eta= \eta_1 \gamma \eta_2$, where $\lambda_{T}(\eta_1 \gamma, p) = \mu^{m_1}$ for $m_1 \le M$. Therefore, $\lambda_{T}(\eta_2, q) = \mu^{m_2}\delta_1\nu \mu^{M}$, where $m_2+ m_1 = M$. Thus, we may decompose $\eta_2$ as $\gamma^{m_2}\rho_1 \eta_3$. It then follows that $\lambda_{T}(\rho_1 \eta_3,q) = \delta_1 \nu \mu^{M}$. Since $\nu \in \xnp$ was chosen arbitrarily, it follows that $U_{\delta_{1}} \subset \im{q}$. Since $\delta \in \im{q}$ was arbitrarily, we conclude that $\im{q}$ is open.
\end{proof}

Applying Proposition~\ref{prop:bijectioniffOn} to the submonoid $\SLn{n}$ of $\Mn{n}$, we obtain the following stronger result. 

\begin{cor}\label{prop:perodicperm}
	Let $T \in \SLn{n}$. Then the following are equivalent: 
	\begin{enumerate}[label=(\alph*)]
	    \item $(T)\Pi$ is injective.
	    \item $(T)\Pi$ is bijective.
	    \item $T \in \Ln{n}$.
	\end{enumerate}
\end{cor}
\begin{proof}
	We note that for $T \in \SLn{n}$, $(T)\Pi$ is injective if and only if it is surjective, since the map   $(T)\Pi$ maps an element of $\rwnl{k}$ to an element of $\rwnl{l}$ for some $l \le k$, injectivity, a simple induction argument and the fact that $\rwnl{k}$ is finite for all $k \in \N$, then demonstrates that $(\rwnl{k})(T)\Pi = \rwnl{k}$. Therefore, $(T)\Pi$ is a bijection and the  follows from Proposition~\ref{prop:bijectioniffOn}.
\end{proof}

Corollary~\ref{prop:perodicperm} above prompts the following questions for an element $T \in \Mn{n}$. 
\begin{enumerate}[label= {\bfseries Q.\arabic*}]
  \item Is it true that $T \in \On{n}$ if and only if $(T)\Pi$ is injective? 
	\item Is it true that $T \in \SOn{n}$ if and only if $(T)\Pi$ is surjective?
\end{enumerate}

Notice that  the existence of elements of $\shn{n}\backslash \hn{n}$ indicates that there are elements $T$ of $\SLn{n} \backslash \Ln{n}$ for which $(T)\Pi$ is surjective but not injective. 

We note that Corollary~\ref{prop:perodicperm} has a straightforward topological proof. Every element of $\SLn{n}$ induces an element of $\End(\xnz, \shift{n})$. Since periodic points are dense in $\xnz$, $(T)\Pi$ is injective for an element $T \in \SLn{n}$ precisely when $(T, \alpha)$ is injective on periodic points for any annotation $\alpha$ of $T$. Now $(T, \alpha)$ is injective on periodic points if and only if it is a bijection on periodic points. In this case, by density, $(T, \alpha)$ is then a continuous bijection from $\xnz$ to itself and so it is a homeomorphism. Hence $T \in \Ln{n}$.  We note that, as there is no apparent well-defined action by homeomorphisms of $\On{n}$ on $\xnz$ or another `nice' space, we are unable to find a similar topological argument for $\On{n}$.

\begin{prop}
	The group $\On{n,r}$ is centreless for any $1 \le r \le n-1$
\end{prop}
\begin{proof}
	We first observe that a consequence of a famous result of Ryan \cite{Ryan2} is that the group $\Ln{n}$ is centerless (one may also deduce this fact from Theorem~\ref{thm:faithfulrepdense}). It therefore suffices to show that the center of $\On{n}$ is contained in the center of $\Ln{n}$.
	
	The result is a consequence of the following claims.
	
	\begin{cla}\label{cla:centerless}
		Let $T \in \On{n}\backslash \Ln{n}$ and let $i,j \in \N$ . Then there is an element $\Gamma \in \wnl{+}$ of length at least $i$ such that $|\lambda_{T}(\Gamma, q_{\Gamma})| - |\Gamma| \ge j$.
    \end{cla} 
    \begin{proof}
    	To see this observe that for $T \in \On{n}$, there is a word $\gamma \in \wnl{+}$ such that $|\gamma|< |\lambda_{T}(\gamma,q_{\gamma})$. Fix such a word $\gamma \in \wnl{+}$. 
    	
    	Let $k$ be the minimal synchronizing level of $T$. Since $T$ is finite, there is a (non-zero) number $D \in \N$ so that for any state $q \in Q_{T}$ and any word $\tau \in \xn^{k+1}$, $\vert |\lambda(\tau,q)| - |\tau|\vert \le D$. As $|\lambda_{T}(\gamma,q_{\gamma})| > |\gamma|$, there is number $m \in \N$ such that $m|\lambda_{T}(\gamma,q_{\gamma})| > m|\gamma| + D +k+1 +j$ and $m|\gamma| \ge \max\{i,1,k\}$. Let $x \in \xn$ be a letter distinct from the first letter of $\gamma$ and let $\delta$ be a word of length $k$ such that the state of $T$ forced by $\delta$ is $q_{\gamma}$. Since $\gamma$ is a prime word then, $\gamma^{m} x \delta$ is a prime word also. Moreover, $q_{\gamma^{m}x\delta} = q_{\gamma}$ by construction. Therefore,  setting $\Gamma = \gamma^{m}x \delta$, we have $|\Gamma| > i$ and $|\lambda_{T}(\Gamma, q_{\Gamma})| = |\lambda_{T}(\gamma,q_{\gamma})^{m} \lambda_{T}(x\Delta, q)| \ge |\Gamma| + j$.  
    \end{proof}

    \begin{cla}\label{cla2}
    	Let $d_r$ be the index of $\On{n,r}$ in $\On{n,n-1}$ and let $m \in \N$ be such that $|\rwnl{m}|  > \max\{2d_r,3\}$. Let  $\rotc{\Gamma},\rotc{\Delta}, \rotc{\psi}$ be distinct elements of $\rwnl{m}$. Then there is an element of $T \in \Ln{n,r}$ such that $(\rotc{\psi})(T)\Pi = \rotc{\psi}$, and  $(\rotc{\Gamma})(T)\Pi = \rotc{\Delta}$.
    \end{cla}
   \begin{proof}
   	 By Theorem~\ref{thm:faithfulrepdense} and hypothesis on the size of $\rwnl{m}$ there is an element $U \in \Ln{n}$ such that  $T:= U^{d_r} \in \Ln{n,r}$, $(\rotc{\psi})(T)\Pi = \rotc{\psi}$, and $(\rotc{\Gamma})(T)\Pi = \rotc{\Delta}$. 
   \end{proof}
    	
    	We may now conclude the proof of the proposition.
    	
    	Fix $T \in \On{n,r}\backslash \Ln{n}$, since $T^{-1} \in \On{n,r}\backslash \Ln{n}$ as well, then it follows, by Claim~\ref{cla:centerless}, that for every $i,j \in \N$, there is an element $\Gamma \in \wnl{+}$ of length at least $i$ such that $|\Gamma| - |\lambda_{T}(\Gamma,q)| \ge j$. Since the set of words $\wnl{i}$ increases in size as $i$ increases, by the pigeon-hole there is an $M \in \N$ and there are distinct words $\Gamma, \Delta \in \wnl{M}$ such that $|\lambda_{T}(\Gamma, q_{\Gamma})| \ne |\lambda_{T}(\Delta, q_{\Delta})|$. Writing $d_r$ for the index of $\On{n,r}$ in $\On{n,n-1}$, we may assume that $M > \max\{2d_r, 3\}$. Without loss of generality we may also assume that $|{\Gamma}| \ne |{\lambda_{T}(\Gamma, q_{\Gamma})}|$, so that $\rotc{\Gamma}, \rotc{{\lambda_{T}(\Gamma, q_{\Gamma})}}$ and $\rotc{\Delta}$ are distinct elements of $\rwnl{M}$.

	 By Claim~\ref{cla2} there is an element $L \in \Ln{n,r}$ which induces the permutation of the set $\rwnl{M}$ that maps $\rotc{\Gamma}$ to $\rotc{\Delta}$ and fixes $\rotc{{\lambda_{T}(\Gamma, q_{\Gamma})}}$. Consider the maps $(LT)\Pi$ and $(TL)\Pi$. We have $$(\rotc{\Gamma})(LT)\Pi = \rotc{\lambda_{T}(\Delta, q_{\Delta})}  \mbox{ and }$$  $$(\rotc{\Gamma})(TL)\Pi = (\rotc{\lambda_{T}(\Gamma, q_{\Gamma})})(L)\Pi = \rotc{\lambda_{T}(\Gamma, q_{\Gamma})}.$$ Since $|\lambda_{T}(\Gamma, q_{\Gamma})| \ne |\lambda_{T}(\Delta, q_{\Delta})|$, we conclude that $(\rotc{\Gamma})(LT)\Pi \ne  (\rotc{\Gamma})(TL)\Pi$. Therefore, as $T \in \On{n,r} \backslash \Ln{n,r}$ was arbitrarily chosen, we see that the centre of $\On{n,r}$ must be contained in the centre of $\Ln{n,r}$. 
	 
	 It is now a consequence of the following claim and  Claim~\ref{cla2}  that the centre of $\Ln{n,r}$ is trivial.
	 
	 \begin{cla}
	 	Let $L \in \Ln{n}$ be non-trivial. Then for every $i \in \N$ there is an $m \in \N$ such that $(L)\Pi: \rwnl{m} \to \rwnl{m}$ moves at least $i$ points.
	 \end{cla}
   \begin{proof}
   	 Since $L \in \Ln{n}$ is non-trivial, there is a state $q \in Q_{L}$ and a word $\Gamma \in \wnl{k}$ for some $k \in \N$, such that $\pi_{L}(\Gamma,q) = q$ and $\rotc{\lambda_{L}(\Gamma,q)} \ne \rotc{\Gamma}$.
    
    Since $L$ is minimal, there is a number $N \in \N$ and elements $\Delta_1, \ldots, \Delta_t \in \xn^{N}$, for $t \ge i$, such that  $\pi_{L}(\Delta_{j}, q) = q$ and $\lambda_{L}(\Delta_{j}, q)$ has no non-empty initial prefix in common with $\rotc{\lambda_{L}(\Gamma,q)}$ for all $j$ between $1$ and $t$. We note that $\Gamma$ is not a prefix of $\Delta_{j}$ for any $1 \le j \le t$. 

  Now using the fact that $\Gamma$ is a prime word, there is a number $M \in \N$, such that  $\Gamma^{M}\Delta_{j}$ is a prime word for all $1 \le j \le t$. Moreover, since $\Delta$ has no-initial prefix in common with $\Gamma$, we see that $|\{ \rotc{\Gamma^{M}\Delta_{j} }\mid 1 \le j \le t \}| = t$.
  	
  Set $\psi =\lambda_{L}(\Gamma, q)$ and set $\varphi_{j} = \lambda_{L}(\Delta_{j},q)$, for $1 \le j \le t$. Since $\rotc{\psi} \ne \rotc{\Gamma}$. We see that, for $M$ large enough, $\rotc{\psi^{M}\varphi_{j}}\ne \rotc{\Gamma^{M}\Delta_{j}}$ for any $1 \le j \le t$. Since $\Pi: \Ln{n} \to \Ln{n}$ is a bijection, this yields the claim. 
   \end{proof}
	 
\end{proof}

\begin{prop}
	For all $1 \le r \le n-1$, the index of $\Ln{n,r}$ in $\On{n,r}$ is infinite.
\end{prop}
\begin{proof}
	It suffices, for all $k \in \N$ to construct elements $T_{k}$ of $\On{n,1}$ such that $(\rotc{1})(T_{k})\Pi$ is equal to $\rotc{\delta}$ for some word $\delta \in \xn^{k}$. This is because $\On{n,1} \le \On{n,r}$ for all $1 \le r \le n-1$ and for all $L \in \Ln{n}$ the restriction of $(L)\Pi$ to the set $\xn^{k}$ is the map $(L)\Pi_{k}$.
	
	Figure~\ref{fig:infiniteindex} gives the required examples. The reader can verify that the element described  in that figure  is in fact an element of $\On{n,1}$ (the state $q_0$ and $q_5$ are the only homeomorphism states) and has order 2. 	
\end{proof}

	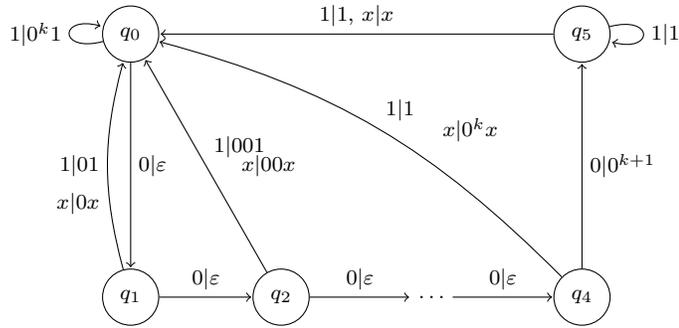
\begin{figure}[htbp]
	\begin{center} 
		\begin{tikzpicture}[shorten >=0.5pt,node distance=3cm,on grid,auto] 
		\tikzstyle{every label}=[blue]
		\node[state] (q_0) [xshift=3.5cm, yshift=3.5cm] {$q_0$}; 
		\node[state] (q_1) [xshift=3.5cm, yshift=0cm] {$q_1$}; 
		\node[state] (q_2) [xshift=5.5cm,yshift = 0cm]{$q_2$}; 
		\node        (q_3) [xshift=7.5cm, yshift=0cm] {$\ldots$};
		\node[state] (q_4) [xshift=9.5cm,yshift=0cm] {$q_4$};
		\node[state] (q_5) [xshift=9.5cm,yshift=3.5cm] {$q_5$};
		\path[->] 
		(q_0) edge node {$0|\varepsilon$} (q_1)
		edge [loop left] node[swap] {$1|0^{k}1$} ()     
		(q_1) edge[in=255, out=105] node{$1|01$} node[xshift=0cm, yshift=-.5cm] {$x|0x$}  (q_0)
		edge node{$0|\varepsilon$} (q_2)
		(q_2) edge node{$0|\varepsilon$} (q_3)
		      edge node[swap]{$1|001$} node[xshift=1.3cm, yshift=0.25cm]{$x|00x$} (q_0)   
		(q_3) edge node {$0|\varepsilon$} (q_4)
		(q_4) edge[in=345, out=135] node[swap] {$1|1$} node[xshift=1.7cm,yshift=0.25cm]{$x|0^{k}x$} (q_0)
		edge node[swap] {$0|0^{k+1}$} (q_5)
		(q_5) edge[loop right] node[swap] {$1|1$} ()
		edge node[swap] {$1|1$, $x|x$} (q_0);
		\end{tikzpicture}
		\caption{An element of $\On{n}$ with loop labelled 1 writing an output of length $k+1$. The symbol $x$ is an element of $\xn \backslash\{0,1\}$.}
		\label{fig:infiniteindex}
	\end{center}
\end{figure}

Clearly the problem of deciding when a product of elements of $\On{n}$ is the trivial element is solvable. The next result shows that the order problem is unsolvable in the group $\On{n}$.

\begin{theorem}
	Let $1 \le r \le n-1$. Then the group $\On{n,r}$ has unsolvable order problem.
\end{theorem}
\begin{proof}
	We begin with the observation that if the order problem is unsolvable in $\ALn{n}$ then it is unsolvable in $\Ln{n}$. Suppose the order problem is solvable in $\Ln{n}$. This means, by Corollary~\ref{cor:mainresult}, that it is possible to decide in finite time if an element of $\ALn{n}$ is an element of the group $\{ (\id, k ) \mid k \in \Z \}$ or if it is an element of infinite order. Since any non-identity element of $\gen{\shift{n}}$ has infinite order, and it is decidable whether an element of $\ALn{n}$ is trivial, it is therefore possible to decide if an element of $\ALn{n}$ has finite order or not.
	
	By results in the papers \cite{KariOllinger}, the order problem in $\ALn{n}$ is undecidable for some $n \in \N$. Results of \cite{KimRoush} or \cite{VilleS18} now imply that the order problem in $\ALn{n}$ is undecidable for any $n \ge 2$. Thus the order problem is undecidable in $\Ln{n}$ and so is undecidable in $\On{n}$ as well. 
	To conclude we observe that for any $1 \le r \ne n-1$, the group $\On{n,r}$ has finite index in~$\On{n}$.  
\end{proof}


\bibliographystyle{amsplain}
\bibliography{ploiBib}

\affiliationone{
   James Belk\\
Department of Mathematics\\
Cornell University\\
Ithaca, NY 14853\\
USA
   \email{jmb226@cornell.edu}}
\affiliationtwo{
   Collin Bleak, Peter J. Cameron and Feyishayo Olukoya\\
   School of Mathematics and Statistics\\
   University of St Andrews\\
   North Haugh\\
   St Andrews, Scotland\\
   KY16 9SS\\
   UK
   \email{cb211@st-andrews.ac.uk\\
   pjc20@st-andrews.ac.uk\\
   fo55@st-andrews.ac.uk}}

\end{document}